\documentclass[12pt]{article}
\usepackage[latin1]{inputenc}

\usepackage{fullpage}
\usepackage{amsfonts}
\usepackage{amssymb}
\usepackage{amsmath}
\usepackage{color}
\usepackage[mathscr]{euscript}
\usepackage{mathtools}
\usepackage{stmaryrd}
\usepackage{wasysym}
\usepackage{cases}
\usepackage{tikz}

\usepackage[colors]{optsys}
\def\figures{./}
\numberwithin{theorem}{section}

\newcommand{\EuclidianTop}[2]{\TopNorm{#2}{#1}} 
\newcommand{\EuclidianTopNorm}[2]{\TopNorm{\norm{#2}}{#1}} 
\newcommand{\EuclidianSupport}[2]{\SupportNorm{{#2}}{#1}} 
\newcommand{\EuclidianSupportNorm}[2]{\SupportNorm{\norm{#2}}{#1}} 
\newcommand{\ECapra}{E-Capra} 
\newcommand{\ESPHERE}{{\mathbb S}} 
\newcommand{\EBALL}{{\mathbb B}}

\title{Hidden Convexity in the $l_0$ Pseudonorm}

\author{Jean-Philippe Chancelier 
  and Michel De Lara\footnote{michel.delara@enpc.fr}
  \\ CERMICS, Ecole des Ponts, Marne-la-Vall\'ee, France}

\begin{document}

\maketitle

\begin{abstract}
  The so-called \lzeropseudonorm\ on~$\RR^d$
  counts the number of nonzero components of a vector.
  It is well-known that the \lzeropseudonorm\ is not convex, as 
  its Fenchel biconjugate is zero.
  In this paper, we introduce a suitable conjugacy, induced by 
  a novel coupling, \ECapra, that has
  the property of being constant along primal rays
  like the \lzeropseudonorm.
  The coupling \ECapra\ belongs to the class of one-sided linear couplings,
  that we introduce; we show that they induce conjugacies 
  that share nice properties with the classic Fenchel conjugacy.
  For the \ECapra\ conjugacy, induced by the coupling \ECapra,
  we relate the \ECapra\ conjugate and biconjugate 
  of the \lzeropseudonorm, 
  the characteristic functions of its level sets
  and the sequence of so-called top-$k$ norms.
  In particular,   
  we prove that the \lzeropseudonorm\ is equal to its biconjugate:
  hence, the \lzeropseudonorm\ is \ECapra-convex 
  in the sense of generalized convexity.
  As a corollary, we show that there exists 
  a proper convex lower semicontinuous function on~$\RR^d$
  such that this function and the \lzeropseudonorm\ coincide
  on the Euclidian unit sphere.
  This hidden convexity property 
  is somewhat surprising as the \lzeropseudonorm\ is a highly
  nonconvex function of combinatorial nature.
  We provide different expressions for this proper convex lower semicontinuous function,
  and we give explicit formulas in the two-dimensional case. 
\end{abstract}

{{\bf Keywords}: \lzeropseudonorm, coupling, Fenchel-Moreau conjugacy,
  top-$k$ norms, $k$-support norms, hidden convexity.}

\section{Introduction}

The \emph{counting function}, also called \emph{cardinality function}
or \emph{\lzeropseudonorm}, 
counts the number of nonzero components of a vector in~$\RR^d$.
It is related to the rank function defined over matrices
\cite{Hiriart-Urruty-Le:2013}.
It is well-known that the \lzeropseudonorm\ is 
lower semi continuous (lsc) but is not convex,
and that the Fenchel conjugacy fails to provide relevant analysis.
Indeed, the Fenchel biconjugate of the characteristic function of the level sets of the 
\lzeropseudonorm\ is zero, and 
the Fenchel biconjugate of the \lzeropseudonorm\ is also zero.

In this paper, we display a suitable conjugacy 
for which we prove that the \lzeropseudonorm\ is ``convex''
in the sense of generalized convexity, that is, is equal to its biconjugate. 
As a corollary, we also show that the \lzeropseudonorm\ function displays hidden convexity in the 
following sense\footnote{%
  In \cite{BenTal-Ben-Teboulle:1996}, the vocable ``hidden convexity'' refers to
  \emph{optimization problems} (when an original problem is equivalent to a
  \emph{convex optimization problem}). Here, the vocable ``hidden convexity''
  refers to \emph{functions} (when a function is the composition of a \emph{convex
    function} with a mapping). \label{ft:hidden_convexity}}: 
the \lzeropseudonorm\ is equal to the composition of a  
proper convex lower semicontinuous function on~$\RR^d$
with the normalization mapping from~$\RR^d$ to the Euclidian unit sphere.
\smallskip

The paper is organized as follows.
In Sect.~\ref{sec:One-sided_linear_couplings},
we provide background on Fenchel-Moreau conjugacies,
then we introduce a novel class of \emph{one-sided linear couplings},
which includes the 
\emph{Euclidian constant along primal rays coupling~$\CouplingCapra$ (\ECapra)}.
We show that one-sided linear couplings induce conjugacies that share nice properties
with the classic Fenchel conjugacy, by giving expressions for 
conjugate and biconjugate functions.
We also elucidate the structure of \ECapra-convex functions.
Then, 
in Sect.~\ref{CAPRAC_conjugates_and_biconjugates_related_to_the_pseudo_norm},
we relate the \ECapra\ conjugate and biconjugate 
of the \lzeropseudonorm, 
the characteristic functions of its level sets
and the top-$k$ norms.
In particular, we show that the \lzeropseudonorm\ is 
\ECapra\ biconjugate (that is, a \ECapra-convex function).
In Sect.~\ref{Hidden_convexity_in_the_pseudonorm}, we deduce
that the \lzeropseudonorm\ coincides, on the Euclidian unit sphere, with a proper convex lsc function
\( {\cal L}_0 \) defined on~$\RR^d$. 
We provide various expression for the function~\( {\cal L}_0 \).
The Appendix~\ref{Appendix} gathers 
properties of top-$k$ norms and of $k$-support norms, 
properties of the \lzeropseudonorm\ level sets,
and technical results on the function~\( {\cal L}_0 \).

\section{One-sided linear couplings}
\label{sec:One-sided_linear_couplings}

After having recalled background on Fenchel-Moreau conjugacies
in~\S\ref{Background_on_Fenchel-Moreau_conjugacies},
we introduce \emph{one-sided linear couplings}
in~\S\ref{One-sided_linear_couplings}.
\smallskip

When we manipulate functions with values 
in~$\barRR = [-\infty,+\infty] $,
we adopt the Moreau \emph{lower addition} \cite{Moreau:1970} 
that extends the usual addition with 
\( \np{+\infty} \LowPlus \np{-\infty} = \np{-\infty} \LowPlus \np{+\infty} =
-\infty \).
Let \( \UNCERTAIN \) be a set.
For any function \( \fonctionuncertain : \UNCERTAIN \to \barRR \),
its \emph{epigraph} is \( \epigraph\,\fonctionuncertain= 
\defset{ \np{\uncertain,t}\in\UNCERTAIN\times\RR}%
{\fonctionuncertain\np{\uncertain} \leq t} \),
its \emph{effective domain} is 
\( \dom\,\fonctionuncertain= 
\defset{\uncertain\in\UNCERTAIN}{ \fonctionuncertain\np{\uncertain} <+\infty}
\).
A function \( \fonctionuncertain: \UNCERTAIN \to \barRR \)
is said to be \emph{proper} if it never takes the value~$-\infty$
and if \( \dom\,\fonctionuncertain \not = \emptyset \).
When \( \UNCERTAIN \) is equipped with a topology,
the function \( \fonctionuncertain: \UNCERTAIN \to \barRR \)
is said to be \emph{lower semi continuous (lsc)}
if its epigraph is a closed subset of~\( \UNCERTAIN \times \RR \).

\subsection{Background on Fenchel-Moreau conjugacies}
\label{Background_on_Fenchel-Moreau_conjugacies}

We review concepts and notations related to the Fenchel conjugacy
(we refer the reader to \cite{Rockafellar:1974}),
then present how they are extended to general conjugacies
\cite{Singer:1997,Rubinov:2000,Martinez-Legaz:2005}.

\subsubsection*{The Fenchel conjugacy}

Let $\PRIMAL$ and $\DUAL$ be two (real) vector spaces that 
are \emph{paired} in the following sense \cite[p.~13]{Rockafellar:1974}:
there exists a  bilinear form 
\( \proscal{}{} : \PRIMAL \times \DUAL \to \RR \)
and locally convex topologies that are compatible 
in the sense
that the continuous linear forms on~$\PRIMAL$
are the functions 
\( \primal \in \PRIMAL \mapsto \proscal{\primal}{\dual} \),
for all \( \dual \in \DUAL \),
and 
that the continuous linear forms on~$\DUAL$
are the functions 
\( \dual \in \DUAL \mapsto \proscal{\primal}{\dual} \),
for all \( \primal \in \PRIMAL \).
The classic \emph{Fenchel conjugacy}~$\star$ is defined, 
for any functions \( \fonctionprimal : \PRIMAL  \to \barRR \)
and \( \fonctiondual : \DUAL \to \barRR \), by\footnote{%
  In convex analysis, one does not use the notation~\( \LFMr{} \),
  but simply~\( \LFM{} \). We use~\( \LFMr{} \)
  to be consistent with the notation~\eqref{eq:Fenchel-Moreau_reverse_conjugate} for general conjugacies.}
\begin{subequations}
  \begin{align}
    \LFM{\fonctionprimal}\np{\dual} 
    &= 
      \sup_{\primal \in \PRIMAL} \Bp{ \proscal{\primal}{\dual} 
      \LowPlus \bp{ -\fonctionprimal\np{\primal} } } 
      \eqsepv \forall \dual \in \DUAL
      \eqfinv
      \label{eq:Fenchel_conjugate}
    \\
    \LFMr{\fonctiondual}\np{\primal} 
    &= 
      \sup_{ \dual \in \DUAL } \Bp{ \proscal{\primal}{\dual} 
      \LowPlus \bp{ -\fonctiondual\np{\dual} } } 
      \eqsepv \forall \primal \in \PRIMAL
      \eqfinv
      \label{eq:Fenchel_conjugate_reverse}
    \\
    \LFMbi{\fonctionprimal}\np{\primal} 
    &= 
      \sup_{\dual \in \DUAL} \Bp{ \proscal{\primal}{\dual} 
      \LowPlus \bp{ -\LFM{\fonctionprimal}\np{\dual} } } 
      \eqsepv \forall \primal \in \PRIMAL
      \eqfinp
      \label{eq:Fenchel_biconjugate}
  \end{align}
\end{subequations}
Recall that a function is said to be \emph{convex} if its
epigraph is a convex subset of $\PRIMAL \times \RR$.
Recall that a function is said to be \emph{closed} if it is either lsc
and nowhere having the value $-\infty$,
or is the constant function~$-\infty$
\cite[p.~15]{Rockafellar:1974}.
It is proved that the Fenchel conjugacy induces a one-to-one correspondence
between the closed convex functions on~$\PRIMAL$
and the closed convex functions on~$\DUAL$
\cite[Theorem~5]{Rockafellar:1974}.
Closed convex functions are the two constant functions~$-\infty$ and~$+\infty$
united with all proper convex lsc functions.\footnote{%
In particular, any closed convex function that takes at least one finite value
is necessarily proper convex~lsc. \label{ft:closed_convex_function}}

\subsubsection*{The general case}

Let be given two sets $\PRIMAL$ (``primal''), $\DUAL$ (``dual''),
not necessarily vector spaces, together 
with a \emph{coupling} function
\begin{equation}
  \coupling : \PRIMAL \times \DUAL \to \barRR 
  \eqfinp 
\end{equation}
With any coupling, one associates \emph{conjugacies} 
from the set \( \barRR^\PRIMAL \) of functions \( \PRIMAL  \to \barRR \)
to the set \( \barRR^\DUAL \)  of functions \( \DUAL  \to \barRR \),
and from \( \barRR^\DUAL \) to \( \barRR^\PRIMAL \) 
as follows.

\begin{subequations}
  \begin{definition}
    The \emph{$\coupling$-Fenchel-Moreau conjugate} of a 
    function \( \fonctionprimal : \PRIMAL  \to \barRR \), 
    with respect to the coupling~$\coupling$, is
    the function \( \SFM{\fonctionprimal}{\coupling} : \DUAL  \to \barRR \) 
    defined by
    \begin{equation}
      \SFM{\fonctionprimal}{\coupling}\np{\dual} = 
      \sup_{\primal \in \PRIMAL} \Bp{ \coupling\np{\primal,\dual} 
        \LowPlus \bp{ -\fonctionprimal\np{\primal} } } 
      \eqsepv \forall \dual \in \DUAL
      \eqfinp
      \label{eq:Fenchel-Moreau_conjugate}
    \end{equation}
    With the coupling $\coupling$, we associate 
    the \emph{reverse coupling~$\coupling'$} defined by 
    \begin{equation}
      \coupling': \DUAL \times \PRIMAL \to \barRR 
      \eqsepv
      \coupling'\np{\dual,\primal}= \coupling\np{\primal,\dual} 
      \eqsepv
      \forall \np{\dual,\primal} \in \DUAL \times \PRIMAL
      \eqfinp
      \label{eq:reverse_coupling}
    \end{equation}
    The \emph{$\coupling'$-Fenchel-Moreau conjugate} of a 
    function \( \fonctiondual : \DUAL \to \barRR \), 
    with respect to the coupling~$\coupling'$, is
    the function \( \SFM{\fonctiondual}{\coupling'} : \PRIMAL \to \barRR \) 
    defined by
    \begin{equation}
      \SFM{\fonctiondual}{\coupling'}\np{\primal} = 
      \sup_{ \dual \in \DUAL } \Bp{ \coupling\np{\primal,\dual} 
        \LowPlus \bp{ -\fonctiondual\np{\dual} } } 
      \eqsepv \forall \primal \in \PRIMAL 
      \eqfinp
      \label{eq:Fenchel-Moreau_reverse_conjugate}
    \end{equation}
    The \emph{$\coupling$-Fenchel-Moreau biconjugate} of a 
    function \( \fonctionprimal : \PRIMAL  \to \barRR \), 
    with respect to the coupling~$\coupling$, is
    the function \( \SFMbi{\fonctionprimal}{\coupling} : \PRIMAL \to \barRR \) 
    defined by
    \begin{equation}
      \SFMbi{\fonctionprimal}{\coupling}\np{\primal} = 
      \bp{\SFM{\fonctionprimal}{\coupling}}^{\coupling'} \np{\primal} = 
      \sup_{ \dual \in \DUAL } \Bp{ \coupling\np{\primal,\dual} 
        \LowPlus \bp{ -\SFM{\fonctionprimal}{\coupling}\np{\dual} } } 
      \eqsepv \forall \primal \in \PRIMAL 
      \eqfinp
      \label{eq:Fenchel-Moreau_biconjugate}
    \end{equation}
  \end{definition}
\end{subequations}
The biconjugate of a 
function \( \fonctionprimal : \PRIMAL  \to \barRR \) satisfies
\begin{equation}
  \SFMbi{\fonctionprimal}{\coupling}\np{\primal}
  \leq \fonctionprimal\np{\primal}
  \eqsepv \forall \primal \in \PRIMAL 
  \eqfinp
  \label{eq:galois-cor}
\end{equation}
With the notion of $\coupling$-biconjugate, 
the classic notion of convex function is generalized.
\begin{definition}
  \label{de:coupling-convex_function}
  A function \( \fonctionprimal : \PRIMAL \to \barRR \) 
  is said to be \emph{$\coupling$-convex} it is equal to its
  $\coupling$-biconjugate:
  \begin{equation}
    \fonctionprimal \textrm{ is } \coupling\textrm{-convex }
    \iff 
    \SFMbi{\fonctionprimal}{\coupling}=\fonctionprimal 
    \eqfinp
  \end{equation}
\end{definition}
In generalized convexity, it is established that 
$\coupling$-convex functions 
are all functions of the form 
$\SFM{ \fonctiondual }{\coupling'}$, 
for any \( \fonctiondual : \DUAL \to \barRR \),
or, equivalently,
all functions of the form $\SFMbi{\fonctionprimal}{\coupling}$, 
for any \( \fonctionprimal : \PRIMAL \to \barRR \)
\cite{Singer:1997,Rubinov:2000,Martinez-Legaz:2005}.
As an illustration, the $\star$-convex functions are the closed convex functions
since, as recalled above, the Fenchel conjugacy induces a one-to-one correspondence
between the closed convex functions on~$\PRIMAL$
and the closed convex functions on~$\DUAL$.

\subsection{One-sided linear couplings}
\label{One-sided_linear_couplings}

Now, we introduce  one-sided linear couplings,
and we show that they induce conjugacies that share nice properties
with the classic Fenchel conjugacy.
In what follows, we let $\PRIMAL$ and $\DUAL$ be two paired vector spaces,
$\UNCERTAIN$ be a set and
\( \theta: \UNCERTAIN \to \PRIMAL \) be a mapping.

\begin{definition}
  We define the \emph{one-sided linear coupling} $\coupling_{\theta}$
  between the set~$\UNCERTAIN$ and the vector space~$\DUAL$ by\footnote{%
    In a one-sided linear coupling, 
    the second set~$\DUAL$ possesses a linear structure 
    (and is even paired with a vector space by means of a bilinear form),
    whereas the first set~$\UNCERTAIN$ is not required to carry any structure.}
  \begin{equation}
    \coupling_{\theta}: \UNCERTAIN \times \DUAL \to \barRR 
    \eqsepv 
    \coupling_{\theta}\np{\uncertain, \dual} = 
    \proscal{\theta\np{\uncertain}}{\dual} 
    \eqsepv \forall \np{\uncertain,\dual}  \in \UNCERTAIN\times\DUAL
    \eqfinp 
    \label{eq:one-sided_linear_coupling}
  \end{equation}
  \label{de:one-sided_linear_coupling}
\end{definition}
For any subset \( \Uncertain \subset \UNCERTAIN \),
$\delta_{\Uncertain} : \UNCERTAIN \to \barRR $ denotes the \emph{characteristic function} of the
set~$\Uncertain$:
\begin{equation}
  \delta_{\Uncertain}\np{\uncertain} = 0 \text{ if } \uncertain \in \Uncertain
  \eqsepv
  \delta_{\Uncertain}\np{\uncertain} = +\infty \text{ if } \uncertain \not\in \Uncertain 
  \eqfinp 
  \label{eq:characteristic_function}
\end{equation}
For any subset \( \Primal\subset\PRIMAL \),
\( \sigma_{\Primal} : \DUAL \to \barRR\) denotes the 
\emph{support function of the subset~$\Primal$}:
\begin{equation}
  \sigma_{\Primal}\np{\dual} = 
  \sup_{\primal\in\Primal} \proscal{\primal}{\dual}
  \eqsepv \forall \dual \in \DUAL
  \eqfinp
  \label{eq:support_function}
\end{equation}
Now, we turn to the $\coupling_{\theta}$-conjugacy induced by the
coupling~$\coupling_{\theta}$. 
For this purpose, we introduce the notion of conditional infimum.

\begin{definition}
  Let \( \fonctionuncertain : \UNCERTAIN \to \barRR \) be a function.
  We define the \emph{conditional infimum}
  (of the function~$\fonctionuncertain$ knowing the mapping~$\theta$) 
  as the function \( \ConditionalInfimum{\theta}{\fonctionuncertain}:
  \PRIMAL \to \barRR \) given by 
  \begin{equation}
    \bp{\ConditionalInfimum{\theta}{\fonctionuncertain}}\np{\primal}
    =
    \inf\defset{\fonctionuncertain\np{\uncertain}}{%
      \uncertain\in\UNCERTAIN \eqsepv \theta\np{\uncertain}=\primal}
    \eqsepv \forall \primal \in \PRIMAL 
    \eqfinp
    \label{eq:ConditionalInfimum}
  \end{equation}  
\end{definition}
If \( \primal \not\in \theta\np{\UNCERTAIN} \), we get that 
\(  \bp{\InfimalPostComposition{\theta}{\fonctionuncertain}}\np{\primal}
=+\infty \) by the convention \( \inf \emptyset = +\infty \).
Therefore, regarding effective domains, we have the inclusion 
$\dom\bp{\InfimalPostComposition{\theta}{\fonctionuncertain}} \subset
\theta\np{\UNCERTAIN}$.
The notation \( \ConditionalInfimum{\theta}{\fonctionuncertain} \)
comes from the analogy with a conditional expectation,
and the expression ``conditional infimum'' is taken from \cite{Witsenhausen:1975b}.
The conditional infimum is also called 
epi-composition in \cite[p.~27]{Rockafellar-Wets:1998}
and
infimal postcomposition in \cite[p.~214]{Bauschke-Combettes:2017}.
\smallskip

Here are expressions for the $\coupling_{\theta}$-conjugates and $\coupling_{\theta}$-biconjugates of a function. 
\begin{subequations}
  \begin{proposition}
    \label{pr:one-sided_linear_Fenchel-Moreau_conjugate}
    For any function \( \fonctiondual : \DUAL \to \barRR \), 
    the $\coupling_{\theta}'$-Fenchel-Moreau conjugate
\( \SFMr{\fonctiondual}{\coupling_{\theta}} : \UNCERTAIN \to \barRR \) 
    is given by 
    \begin{equation}
      \SFMr{\fonctiondual}{\coupling_{\theta}}=
      \LFMr{ \fonctiondual } \circ \theta
      \eqfinp
      \label{eq:one-sided_linear_c'-Fenchel-Moreau_conjugate}
    \end{equation}
    For any function \( \fonctionuncertain : \UNCERTAIN \to \barRR \), 
    the $\coupling_{\theta}$-Fenchel-Moreau conjugate
    \( \SFM{\fonctionuncertain}{\coupling_{\theta}}
 : \DUAL \to \barRR \) is given by 
    \begin{equation}
      \SFM{\fonctionuncertain}{\coupling_{\theta}}=
      \LFM{ \bp{\ConditionalInfimum{\theta}{\fonctionuncertain}} }
      \eqfinv 
      \label{eq:one-sided_linear_Fenchel-Moreau_conjugate}
    \end{equation}
    and the $\coupling_{\theta}$-Fenchel-Moreau biconjugate
    \( \SFMbi{\fonctionuncertain}{\coupling_{\theta}}
    : \UNCERTAIN \to \barRR \)     is given by
    \begin{equation}
      \SFMbi{\fonctionuncertain}{\coupling_{\theta}}
      = 
      \LFMr{ \bp{ \SFM{\fonctionuncertain}{\coupling_{\theta}} } }
      \circ \theta
      =
      \LFMbi{ \bp{\ConditionalInfimum{\theta}{\fonctionuncertain}} }
      \circ \theta
      \eqfinp
      \label{eq:one-sided_linear_Fenchel-Moreau_biconjugate}
    \end{equation}
    For any subset \( \Uncertain \subset \UNCERTAIN \), we have 
    \begin{equation}
      \SFM{ \delta_{\Uncertain} }{\coupling_{\theta}}
      = \sigma_{ \theta\np{\Uncertain} } 
      \eqfinp
      \label{eq:one-sided_linear_Fenchel-Moreau_characteristic}
    \end{equation}
  \end{proposition}
\end{subequations}

\begin{proof}
  We prove~\eqref{eq:one-sided_linear_c'-Fenchel-Moreau_conjugate}.
  Letting \( \uncertain \in \UNCERTAIN  \), we have that 
  \begin{align*}
    \SFM{ \bp{ \fonctiondual } }{\coupling_{\theta}'}\np{\uncertain}
    &=
      \sup_{\dual \in \DUAL} 
      \Bp{ \proscal{\theta\np{\uncertain}}{\dual}
      \LowPlus \bp{ -
      \fonctiondual\np{\dual} } }
      \tag{by the conjugate formula~\eqref{eq:Fenchel-Moreau_conjugate} 
      and the coupling~\eqref{eq:one-sided_linear_coupling}}
    \\
    &=
      \LFMr{ \fonctiondual }\bp{\theta\np{\uncertain} } 
      \tag{by the expression~\eqref{eq:Fenchel_conjugate_reverse} of the Fenchel conjugate}
      \eqfinp
  \end{align*}
  \smallskip
  We prove~\eqref{eq:one-sided_linear_Fenchel-Moreau_conjugate}.
  Letting \( \dual \in \DUAL \), we have that 
  \begin{align*}
    \SFM{\fonctionuncertain}{\coupling_{\theta}}\np{\dual} 
    &=
      \sup_{\uncertain \in \UNCERTAIN }
      \Bp{ \proscal{\theta\np{\uncertain}}{\dual}
      \LowPlus \bp{ -\fonctionuncertain\np{\uncertain} } }
      \tag{by the conjugate formula~\eqref{eq:Fenchel-Moreau_conjugate} 
      and the coupling~\eqref{eq:one-sided_linear_coupling}}
    \\
    &=
      \sup_{\primal \in \PRIMAL }
      \sup_{\uncertain \in \UNCERTAIN, \theta\np{\uncertain}=\primal}
      \Bp{ \proscal{\theta\np{\uncertain}}{\dual}
      \LowPlus \bp{ -\fonctionuncertain\np{\uncertain} } }
    \\
    &=
      \sup_{\primal \in \PRIMAL }
      \sup_{\uncertain \in \UNCERTAIN, \theta\np{\uncertain}=\primal}
      \Bp{ \proscal{\primal}{\dual}
      \LowPlus \bp{ -\fonctionuncertain\np{\uncertain} } }
    \\
    &=
      \sup_{\primal \in \PRIMAL }
      \Bp{ \proscal{\primal}{\dual}
      \LowPlus
      \sup_{\uncertain \in \UNCERTAIN, \theta\np{\uncertain}=\primal}
      \bp{ -\fonctionuncertain\np{\uncertain} } }
      \tag{since \( \sup_{b \in \mathbb{B}}\np{a \LowPlus g(b)} =
      a \LowPlus \sup_{b \in \mathbb{B}}g(\uncertain)\) \cite{Moreau:1970}
      }
    \\
    &=
      \sup_{\primal \in \PRIMAL }
      \Bp{ \proscal{\primal}{\dual}
      \LowPlus  \bp{ - \inf_{\uncertain \in \UNCERTAIN, \theta\np{\uncertain}=\primal}
      \fonctionuncertain\np{\uncertain} } }
    \\
    &=
      \sup_{\primal \in \PRIMAL }
      \bgp{ \proscal{\primal}{\dual}
      \LowPlus  \Bp{ - 
      \bp{\ConditionalInfimum{\theta}{\fonctionuncertain}}\np{\primal} } }
      \tag{by the \infimalpostcomposition\
      expression~\eqref{eq:ConditionalInfimum}}
    \\
    &=
      \LFM{ \bp{\ConditionalInfimum{\theta}{\fonctionuncertain}} }\np{\dual}
      \tag{by the expression~\eqref{eq:Fenchel_conjugate} of the Fenchel conjugate}
      \eqfinp
  \end{align*}
  \smallskip
  We prove~\eqref{eq:one-sided_linear_Fenchel-Moreau_biconjugate}.
  Letting \( \uncertain \in \UNCERTAIN  \), we have that 
  \begin{align*}
    \SFMbi{\fonctionuncertain}{\coupling_{\theta}}\np{\uncertain} 
    &=
      \SFM{ \bp{ \SFM{\fonctionuncertain}{\coupling_{\theta}} } }{\coupling_{\theta}'}\np{\uncertain}
      \tag{by the definition~\eqref{eq:Fenchel-Moreau_biconjugate}
      of the biconjugate}
    \\
    &=
      \SFM{ \bp{ \LFM{ \bp{\ConditionalInfimum{\theta}{\fonctionuncertain}} } } }%
      {\coupling_{\theta}'}\np{\uncertain}
      \tag{by~\eqref{eq:one-sided_linear_Fenchel-Moreau_conjugate}}
    \\
    &=
      \LFMbi{ \bp{\ConditionalInfimum{\theta}{\fonctionuncertain}} }
      \bp{\theta\np{\uncertain} } 
      \tag{by~\eqref{eq:one-sided_linear_c'-Fenchel-Moreau_conjugate}}
      \eqfinp
  \end{align*}
  \smallskip
  We prove~\eqref{eq:one-sided_linear_Fenchel-Moreau_characteristic}:
  \begin{align*}
    \SFM{ \delta_{\Uncertain} }{\coupling_{\theta}}
    &= 
      \LFM{ \bp{\ConditionalInfimum{{\theta}}{ \delta_{\Uncertain} }} }
      \tag{ by~\eqref{eq:one-sided_linear_Fenchel-Moreau_conjugate} }
    \\
    &= 
      \LFM{ \delta_{ \theta\np{\Uncertain} } }
      \tag{ because \( \ConditionalInfimum{\theta}{\delta_{\Uncertain}} =
      \delta_{ \theta\np{\Uncertain} } \) by~\eqref{eq:ConditionalInfimum}
      and \eqref{eq:characteristic_function}}
    \\
    &= 
      \sigma_{ \theta\np{\Uncertain} } 
      \tag{ by~\eqref{eq:Fenchel_conjugate}, \eqref{eq:characteristic_function} 
      and \eqref{eq:support_function}}
      \eqfinp
  \end{align*}
  \smallskip

  This ends the proof.
\end{proof}

Now, we are able to characterize the so-called  $\coupling_{\theta}$-convex
functions (see Definition~\ref{de:coupling-convex_function}).
\begin{proposition}
  \label{pr:one-sided_linear_convex_functions}
  A function \( \fonctionuncertain : \UNCERTAIN \to \barRR \) is $\coupling_{\theta}$-convex 
  if and only if it is the composition of 
  a closed convex function \( \fonctionprimal: \PRIMAL \to \barRR \)
  with the mapping~\( \theta: \UNCERTAIN \to \PRIMAL \). 
  More precisely, for any function \( \fonctionuncertain : \UNCERTAIN \to \barRR \),
  we have the equivalences
  \begin{subequations}
    \begin{align}
      &  
        \fonctionuncertain \textrm{ is $\coupling_{\theta}$-convex }
        \label{eq:coupling_theta-convex_function_a}
      \\
      \iff & 
             \fonctionuncertain =
             \SFMbi{\fonctionuncertain}{\coupling_{\theta}}
             \label{eq:coupling_theta-convex_function_b}
      \\
      \iff & 
             \fonctionuncertain = 
             \LFMr{ \bp{ \SFM{\fonctionuncertain}{\coupling_{\theta}} } } 
             \circ \, \theta
             \textrm{ (where } \LFMr{ \bp{
             \SFM{\fonctionuncertain}{\coupling_{\theta}}}}: \PRIMAL  \to \barRR 
             \textrm{ is a closed convex function) }
             \label{eq:coupling_theta-convex_function_c}
      \\
      \iff & \textrm{there exists a closed convex function }
             \fonctionprimal: \PRIMAL \to \barRR 
             \textrm{ such that }
             \fonctionuncertain = \fonctionprimal \circ \theta
             \eqfinp
             \label{eq:coupling_theta-convex_function_d}
    \end{align}
  \end{subequations}
\end{proposition}

\begin{proof}
  The equivalence between \eqref{eq:coupling_theta-convex_function_a}
  and \eqref{eq:coupling_theta-convex_function_b} follows from 
  Definition~\ref{de:coupling-convex_function}.
  The equivalence between~\eqref{eq:coupling_theta-convex_function_b}
  and~\eqref{eq:coupling_theta-convex_function_c} follows
  from~\eqref{eq:one-sided_linear_Fenchel-Moreau_biconjugate};
  Moreover, the function \( \LFMr{ \bp{
      \SFM{\fonctionuncertain}{\coupling_{\theta}} } } \)
  is closed convex since, as recalled above, the Fenchel conjugacy induces a one-to-one correspondence
between the closed convex functions on~$\PRIMAL$
and the closed convex functions on~$\DUAL$.
  Obviously, \eqref{eq:coupling_theta-convex_function_c}
  implies \eqref{eq:coupling_theta-convex_function_d}.

  Finally, there remains to prove that
  \eqref{eq:coupling_theta-convex_function_d}
  implies~\eqref{eq:coupling_theta-convex_function_b}.
  If there exists a closed convex function 
  \( \fonctionprimal: \PRIMAL  \to \barRR \) such that
  \( \fonctionuncertain = \fonctionprimal \circ \theta \), 
  then \( \ConditionalInfimum{\theta}{\fonctionuncertain}
  = \fonctionprimal \UppPlus \delta_{ \theta\np{\UNCERTAIN} } \)
  as easily computed, 
  and therefore \( \SFMbi{\fonctionuncertain}{\coupling_{\theta}}
  = \LFMbi{ \bp{\ConditionalInfimum{\theta}{\fonctionuncertain}} } \circ \theta 
  = \LFMbi{ \bp{ \fonctionprimal \UppPlus \delta_{ \theta\np{\UNCERTAIN} } } } \circ \theta  \)
  by~\eqref{eq:one-sided_linear_Fenchel-Moreau_biconjugate}.
  Now, as \( \fonctionprimal \UppPlus \delta_{ \theta\np{\UNCERTAIN} } 
  \geq \fonctionprimal \) by~\eqref{eq:characteristic_function},
  we get that 
  \( \LFMbi{ \bp{ \fonctionprimal \UppPlus \delta_{ \theta\np{\UNCERTAIN} } } }
  \geq \LFMbi{ \fonctionprimal }= \fonctionprimal \),
  where the last equality holds because the  function 
  \( \fonctionprimal: \PRIMAL  \to \barRR \) is closed convex.
  As a consequence, we obtain that 
  \( \SFMbi{\fonctionuncertain}{\coupling_{\theta}}
  \geq \fonctionprimal \circ \theta = \fonctionuncertain \). 
  Now, by~\eqref{eq:galois-cor}, we always have the inequality
  \( \SFMbi{\fonctionuncertain}{\coupling_{\theta}}
  \leq \fonctionuncertain \).
  Thus, we conclude that \( \SFMbi{\fonctionuncertain}{\coupling_{\theta}}
  = \fonctionuncertain \).
  \smallskip

  This ends the proof. 
\end{proof}

Let us say that \emph{a function \( \fonctionuncertain : \UNCERTAIN \to \barRR \) 
  displays hidden convexity} with respect to the mapping~\( \theta:
\UNCERTAIN \to \PRIMAL \) if there exists 
a closed convex function \( \fonctionprimal: \PRIMAL \to \barRR \)
such that \( \fonctionuncertain = \fonctionprimal \circ \theta \).
Then, we have just proved that this notion of hidden convexity
for functions (see Footnote~\ref{ft:hidden_convexity})
coincides with the notion of $\coupling_{\theta}$-convex functions.

\section{The \ECapra\ conjugacy and the \lzeropseudonorm}
\label{CAPRAC_conjugates_and_biconjugates_related_to_the_pseudo_norm}

From now on, we work on the Euclidian space~$\RR^d$
(with $d \in \NN^*$), equipped with the scalar product 
\( \proscal{\cdot}{\cdot} \)
and with the Euclidian norm
\( \norm{\cdot} = \sqrt{ \proscal{\cdot}{\cdot} } \).
In particular, we consider the following \emph{Euclidian unit sphere}~$\SPHERE$ and
\emph{Euclidian unit ball}~$\BALL$:
\begin{equation}
  \SPHERE= \defset{\primal \in \RR^d}{\norm{\primal} = 1} 
  \mtext{ and }
  \BALL= \defset{\primal \in \RR^d}{\norm{\primal} \leq 1} 
  \eqfinp
  \label{eq:Euclidian_SPHERE}
\end{equation}
In~\S\ref{Euclidian Constant_along_primal_rays_coupling}, we introduce
the \emph{(Euclidian) constant along primal rays coupling~$\CouplingCapra$
  (\ECapra)}.
Then, we recall definitions of the \lzeropseudonorm, and of the top-$k$ and
$k$-support norms
in~\S\ref{Background_on_the_lzeropseudonorm_and_the_top-k_and_k-support_norms}. 
Finally, 
in~\S\ref{ECapra-conjugates_and_biconjugates_related_to_the_lzeropseudonorm},
we provide expressions for the \ECapra-conjugates and \ECapra-biconjugates 
of functions related to the \lzeropseudonorm.

\subsection{Euclidian Constant along primal rays coupling (\ECapra)}
\label{Euclidian Constant_along_primal_rays_coupling}

We introduce a novel coupling, which is a special case of
one-sided linear coupling.

\begin{definition}
  The \emph{\ECapra\ coupling}~$\CouplingCapra$ between $\RR^d$ and $\RR^d$ 
  is defined by
  \begin{equation}
    \forall \dual \in \RR^d \eqsepv 
    \begin{cases}
      \CouplingCapra\np{\primal, \dual} 
      &= \displaystyle
      \frac{ \proscal{\primal}{\dual} }{ \norm{\primal} }
      = \displaystyle
      \frac{ \proscal{\primal}{\dual} }{ \sqrt{ \proscal{\primal}{\primal} } } 
      \eqsepv \forall \primal \in \RR^d\backslash\{0\} \eqfinv
      \\[4mm]
      \CouplingCapra\np{0, \dual} &= 0.
    \end{cases}
    \label{eq:Euclidian_coupling_CAPRAC}
  \end{equation}
\end{definition}
The coupling \ECapra\ has the property of being 
\emph{constant along primal rays}, hence the acronym\footnote{%
  In fact, there is large class of couplings that are constant along primal rays.
  It suffices to replace the Euclidian norm
  in~\eqref{eq:Euclidian_coupling_CAPRAC}
  with any norm. Such couplings are studied in 
  \cite{Chancelier-DeLara:2020_CAPRA,Chancelier-DeLara:2020_Variational}.
  In this paper, we focus on the constant along primal rays coupling
  induced by the \emph{Euclidian} norm, hence the acronym~\ECapra.}~\ECapra\
(Euclidian Constant Along Primal RAys Coupling).
We introduce 
the primal \emph{normalization mapping}~$\normalized$ as follows:
\begin{equation}
  \normalized : \RR^d \to \SPHERE \cup \{0\} 
  \eqsepv
  \normalized\np{\primal}=
  \begin{cases}
    \frac{ \primal }{ \norm{\primal} }
    & \mtext{ if } \primal \neq 0 \eqfinv 
    \\
    0 
    & \mtext{ if } \primal = 0 \eqfinp
  \end{cases}  
  \label{eq:normalization_mapping}
\end{equation}
With these notations, the coupling~\ECapra\ 
in~\eqref{eq:Euclidian_coupling_CAPRAC} is a special case
of one-sided linear coupling  (see
Definition~\ref{de:one-sided_linear_coupling}):
\( \CouplingCapra = \coupling_{\normalized} \), 
as in~\eqref{eq:one-sided_linear_coupling} with \(\theta=\normalized\),
is the \emph{Fenchel coupling after primal normalization}.
The following Proposition --- that provides
expressions for the \ECapra-conjugates and \ECapra-biconjugates
of a function --- simply is
Proposition~\ref{pr:one-sided_linear_Fenchel-Moreau_conjugate}
in the case where the mapping \(\theta \) is the
normalization mapping~\(\normalized\) in~\eqref{eq:normalization_mapping}. 

\begin{subequations}
  \begin{proposition}
    \label{pr:CAPRA_Fenchel-Moreau_conjugate}
    For any function \( \fonctiondual : \RR^d \to \barRR \), 
    the $\CouplingCapra'$-Fenchel-Moreau conjugate
\( \SFMr{\fonctiondual}{\CouplingCapra}: \RR^d \to \barRR \)
    is given by 
    \begin{equation}
      \SFMr{\fonctiondual}{\CouplingCapra}=
      \LFM{ \fonctiondual } \circ \normalized
      \eqfinp 
      \label{eq:CAPRA'_Fenchel-Moreau_conjugate}
    \end{equation}
    For any function \( \fonctionprimal : \RR^d \to \barRR \), 
    the $\CouplingCapra$-Fenchel-Moreau conjugate
    \( \SFM{\fonctionprimal}{\CouplingCapra}: \RR^d \to \barRR \)
    is given by 
    \begin{equation}
      \SFM{\fonctionprimal}{\CouplingCapra}=
      \LFM{ \bp{\ConditionalInfimum{\normalized}{\fonctionprimal}} }
      \eqfinv 
      \label{eq:CAPRA_Fenchel-Moreau_conjugate}
    \end{equation}
    where the \infimalpostcomposition~\eqref{eq:ConditionalInfimum}
    has the expression
    \begin{equation}
      \bp{\ConditionalInfimum{\normalized}{\fonctionprimal}}\np{\primal}
      =
      \inf\defset{\fonctionprimal\np{\primal'}}{
        \normalized\np{\primal'}=\primal}
      =
      \begin{cases}
        \inf_{\lambda > 0} \fonctionprimal\np{\lambda\primal}
        & \text{if } \primal \in \SPHERE  \cup \{0\}
        \eqfinv
        \\
        +\infty  
        & \text{if } \primal \not\in \SPHERE \cup \{0\}
        \eqfinv
      \end{cases}
      \label{eq:CAPRA_ConditionalInfimum}
    \end{equation}
    and the $\CouplingCapra$-Fenchel-Moreau biconjugate
    \( \SFMbi{\fonctionprimal}{\CouplingCapra}: \RR^d \to \barRR \)
    is given by
    \begin{equation}
      \SFMbi{\fonctionprimal}{\CouplingCapra}
      = 
      \LFMr{ \bp{ \SFM{\fonctionprimal}{\CouplingCapra} } }
      \circ \normalized
      =
      \LFMbi{ \bp{\ConditionalInfimum{\normalized}{\fonctionprimal}} }
      \circ \normalized
      \eqfinp
      \label{eq:CAPRA_Fenchel-Moreau_biconjugate}
    \end{equation}
  \end{proposition}
\end{subequations}

Thanks to Proposition~\ref{pr:one-sided_linear_convex_functions}, 
we easily deduce the following result.

\begin{proposition}
  \label{pr:CAPRA_convex_functions}  
  A function on~$\RR^d$ is $\CouplingCapra$-convex 
  if and only if it is the composition of 
  a closed convex function on~$\RR^d$
  with the normalization mapping~\eqref{eq:normalization_mapping}.
  More precisely, for any function \( \fonctionuncertain : \RR^d \to \barRR \),
  we have the equivalences
  \begin{subequations}
    \begin{align*}
      &      \fonctionuncertain \textrm{ is } \CouplingCapra\textrm{-convex }
      \\
      \iff &
             \fonctionuncertain =
             \SFMbi{\fonctionuncertain}{\CouplingCapra}
      \\
      \iff & 
             \fonctionuncertain = \LFMr{ \bp{ \SFM{\fonctionuncertain}{\CouplingCapra} } } 
             \circ \, \normalized
             \textrm{ (where } \LFMr{ \bp{ \SFM{\fonctionuncertain}{\CouplingCapra} } } : \RR^d  \to \barRR 
             \textrm{ is a closed convex function) }
      \\
      \iff & \textrm{there exists a
             closed convex function }
             \fonctionprimal: \RR^d  \to \barRR 
             \textrm{ such that }
             \fonctionuncertain = \fonctionprimal \circ \normalized
             \eqfinp
    \end{align*}
  \end{subequations}
\end{proposition}

Now, we turn to analyze the \lzeropseudonorm\ by means of the \ECapra\
conjugacy. 

\subsection{The \lzeropseudonorm, and the top-$k$ and $k$-support norms}
\label{Background_on_the_lzeropseudonorm_and_the_top-k_and_k-support_norms}

We recall definitions of the so-called \lzeropseudonorm, and of the top-$k$ and
$k$-support norms.

\subsubsubsection{The \lzeropseudonorm}

The \emph{\lzeropseudonorm} is the function
\( \lzero : \RR^d \to \na{0,1,\ldots,d} \)
defined~by 
\begin{equation}
  \lzero\np{\primal} =
  \bcardinal{ \bset{ j \in \na{1,\ldots,d} }{\primal_j \not= 0 } }
  \eqsepv \forall \primal \in \RR^d
  \eqfinv
  \label{eq:pseudo_norm_l0}  
\end{equation}
where $\cardinal{K}$ denotes the cardinal of 
a subset \( K \subset \na{1,\ldots,d} \). 
The \lzeropseudonorm\ shares three out of the four axioms of a norm:
nonnegativity, positivity except for \( \primal =0 \), subadditivity.
The axiom of 1-homogeneity does not hold true;
by contrast, the \lzeropseudonorm\ is 0-homogeneous as 
\( \lzero\np{\rho\primal} = \lzero\np{\primal} \),
\( \forall \rho \in \RR\backslash\{0\} \), \( \forall \primal \in \RR^d \).
Thus, the \lzeropseudonorm\ displays the invariance property
\begin{equation}
  \lzero \circ \normalized = \lzero   
  \label{eq:pseudonormlzero_invariance_normalized}
\end{equation}
with respect to the normalization mapping~\eqref{eq:normalization_mapping}.
This property will be instrumental to 
show that the \lzeropseudonorm\ 
is a $\CouplingCapra$-convex function.

\subsubsubsection{The level sets of the \lzeropseudonorm}

The \lzeropseudonorm\ is used in exact sparse optimization problems of the form
\( \inf_{ \lzero\np{\primal} \leq k } \fonctionprimal\np{\primal} \).
Thus, we introduce the \emph{level sets}
\begin{subequations}
  \begin{align}
    \LevelSet{\lzero}{k} 
    &= 
      \defset{ \primal \in \RR^d }{ \lzero\np{\primal} \leq k }
      \eqsepv \forall k \in \ba{0,1,\ldots,d} 
      \eqfinv
      \label{eq:pseudonormlzero_level_set}
      \intertext{and the \emph{level curves}}
    \LevelCurve{\lzero}{k} 
    &= 
      \defset{ \primal \in \RR^d }{ \lzero\np{\primal} = k }
      \eqsepv \forall k \in \ba{0,1,\ldots,d} 
      \eqfinp
      \label{eq:pseudonormlzero_level_curve}
  \end{align}
\end{subequations}

For any subset \( K \subset \na{1,\ldots,d} \), 
we denote the subspace of~$\RR^d$ made of vectors
whose components vanish outside of~$K$ by\footnote{%
  Here, following notation from Game Theory, 
  we have denoted by $-K$ the complementary subset 
  of~$K$ in \( \na{1,\ldots,d} \): \( K \cup (-K) = \na{1,\ldots,d} \)
  and \( K \cap (-K) = \emptyset \).}
\begin{equation}
  \FlatRR_{K} = \RR^K \times \{0\}^{-K} =
  \bset{ \primal \in \RR^d }{ \primal_j=0 \eqsepv \forall j \not\in K } 
  \subset \RR^d 
  \eqfinv
  \label{eq:FlatRR}
\end{equation}
where \( \FlatRR_{\emptyset}=\{0\} \).
For any \( \primal \in \RR^d \) and \( K \subset \ba{1,\ldots,d} \), 
we denote by
\( \primal_K \in \RR^d \) the vector which coincides with~\( \primal \),
except for the components outside of~$K$ that vanish:
\( \primal_K \) is the orthogonal projection of~\( \primal \) onto
the subspace~\( \FlatRR_{K} \).
The level sets of the \lzeropseudonorm\
in~\eqref{eq:pseudonormlzero_level_set}
are easily related to the subspaces~\( \FlatRR_{K} \) of~\( \RR^d \),
as defined in~\eqref{eq:FlatRR}, by\footnote{%
  The notation \( \bigcup_{\cardinal{K} \leq k} \)
  is a shorthand for 
  \( \bigcup_{ { K \subset \na{1,\ldots,d}, \cardinal{K} \leq k}} \)
  (and the same for \( \bigcup_{\cardinal{K} = k} \)).}  
\begin{equation}
  \LevelSet{\lzero}{k} 
  = 
  \bset{ \primal \in \RR^d }{ \lzero\np{\primal} \leq k }
  = \bigcup_{\cardinal{K} \leq k} \FlatRR_{K} 
  \eqsepv \forall k=0,1,\ldots,d 
  \eqfinp
  \label{eq:level_set_pseudonormlzero}
\end{equation}

\subsubsubsection{The top-$k$ and $k$-support norms}

\begin{definition}
  \label{de:symmetric_gauge_norm}
  For \( k \in \ba{1,\ldots,d} \), we define\footnote{%
    The notation  \( \sup_{\cardinal{K} \leq k} \) 
    is a shorthand for 
    \( \sup_{ { K \subset \na{1,\ldots,d}, \cardinal{K} \leq k}} \)
    (and the same for \( \sup_{\cardinal{K} = k} \)).
    The property that \( \sup_{\cardinal{K} \leq k}\norm{\primal_K} 
    = \sup_{\cardinal{K} = k}\norm{\primal_K} \) 
    in~\eqref{eq:symmetric_gauge_norm} 
    comes from the easy observation that 
    \( K\subset K' \Rightarrow \norm{\primal_K} \leq \norm{\primal_{K'}} \).}
  \begin{equation}
    \EuclidianTopNorm{k}{\primal}
    =
    \sup_{\cardinal{K} \leq k}\norm{\primal_K} 
    =
    \sup_{\cardinal{K} = k}\norm{\primal_K} 
    \eqsepv \forall \primal \in \RR^d 
    \eqfinp
    \label{eq:symmetric_gauge_norm}
  \end{equation}
  Thus defined, \( \EuclidianTopNorm{k}{\cdot} \) is a norm,
  the so-called \emph{top-$k$ norm}.
  Its dual norm, as in~\eqref{eq:dual_norm}, denoted by\footnote{%
    We use the symbol~$\star$ in the superscript to indicate that the 
    $k$-support norm  \( \EuclidianSupportNorm{k}{\cdot} \) is a dual norm.}
  \( \EuclidianSupportNorm{k}{\cdot} \), is called the \emph{$k$-support norm}
  \cite{Argyriou-Foygel-Srebro:2012}:
  \begin{equation}
    \EuclidianSupportNorm{k}{\cdot} = \bp{ \EuclidianTopNorm{k}{\cdot} }_{\star}
    \eqfinp
    \label{eq:support_norm}
  \end{equation}
\end{definition}
We follow the terminology of \cite{Tono-Takeda-Gotoh:2017},
where the top-$k$ norm is also called the top-$(k,1)$ norm.
Indeed, the norm of a vector is obtained
with a subvector of size~$k$ having the~$k$ largest components in module:
letting $\sigma$ be a permutation of \( \{1,\ldots,d\} \) such that
\( \module{ \primal_{\sigma(1)} } \geq \module{ \primal_{\sigma(2)} } 
\geq \cdots \geq \module{ \primal_{\sigma(d)} } \),
we have that 
\( \EuclidianTopNorm{k}{\primal} = 
\sqrt{ \sum_{l=1}^{k} \module{ \primal_{\sigma(l)} }^2 } \). 
The top-$k$ norm is also known as the \emph{$2$-$k$-symmetric gauge norm},
or \emph{Ky Fan vector norm}.

\subsection{\ECapra-conjugates and biconjugates of the \lzeropseudonorm}
\label{ECapra-conjugates_and_biconjugates_related_to_the_lzeropseudonorm}

With the Fenchel conjugacy, we calculate that 
\( \LFM{ \delta_{ \LevelSet{\lzero}{k} } }= \delta_{\{0\}} \) 
and \( \LFMbi{ \delta_{ \LevelSet{\lzero}{k} } }= 0 \),
for all $k=1,\ldots,d $,
and that 
\( \LFM{ \lzero }= \delta_{\{0\}} \) and
\( \LFMbi{ \lzero }= 0 \).
Hence, the Fenchel conjugacy is not suitable
to handle the \lzeropseudonorm. 
We will now see that we obtain more interesting formulas with the \ECapra-conjugacy.
Indeed, the \lzeropseudonorm\ in~\eqref{eq:pseudo_norm_l0}, 
the characteristic functions \( \delta_{ \LevelSet{\lzero}{k} } \) 
of its level sets~\eqref{eq:level_set_pseudonormlzero}
and the top-$k$ norms in~\eqref{eq:symmetric_gauge_norm}
are related by the following conjugate formulas.
The proof relies on results gathered in the Appendix~\ref{Appendix}.

\begin{theorem}
  \label{th:pseudonormlzero_conjugate}
  Let~$\CouplingCapra$ be the Euclidian coupling 
  \ECapra\ as defined in~\eqref{eq:Euclidian_coupling_CAPRAC}.
  Let \( k \in \ba{0,1,\ldots,d} \). We have that
  (with the convention, in~\eqref{eq:conjugate_delta_l0norm}
  and in~\eqref{eq:conjugate_l0norm}, 
  that \( \EuclidianTopNorm{0}{\cdot} = 0 \))   
  \begin{subequations}
    \begin{align}
      \SFM{ \delta_{ \LevelSet{\lzero}{k} } }{-\CouplingCapra}
      =       \SFM{ \delta_{ \LevelSet{\lzero}{k} } }{\CouplingCapra}
      &=
        \EuclidianTopNorm{k}{\cdot} 
        \eqfinv
        \label{eq:conjugate_delta_l0norm}
      \\ 
      \SFMbi{ \delta_{ \LevelSet{\lzero}{k} } }{\CouplingCapra} 
      &=
        \delta_{ \LevelSet{\lzero}{k} } 
        \eqfinv
        \label{eq:biconjugate_delta_l0norm}
      \\
      \SFM{ \lzero }{\CouplingCapra} 
      &=
        \sup_{l=0,1,\ldots,d} \Bc{ \EuclidianTopNorm{l}{\cdot} -l }  
        \eqfinv
        \label{eq:conjugate_l0norm}
      \\
      \SFMbi{ \lzero }{\CouplingCapra} 
      &=\lzero 
        \label{eq:biconjugate_l0norm}
        \eqfinp
    \end{align}
  \end{subequations}
\end{theorem}

\begin{proof} 
  We will use the framework and results of
  Sect.~\ref{sec:One-sided_linear_couplings}
  with \( \PRIMAL=\DUAL=\RR^d \),
  equipped with the scalar product \( \proscal{\cdot}{\cdot} \)
  and with the Euclidian norm
  \( \norm{\cdot} = \sqrt{ \proscal{\cdot}{\cdot} } \).
  \smallskip

  \noindent $\bullet$ We prove the first equality in~\eqref{eq:conjugate_delta_l0norm}:
  \begin{align*}
    \SFM{ \delta_{ \LevelSet{\lzero}{k} } }{-\CouplingCapra}
    &= 
      \sigma_{ -\normalized\np{ \LevelSet{\lzero}{k} } }
      \tag{by~\eqref{eq:one-sided_linear_Fenchel-Moreau_characteristic}
      because \( -\CouplingCapra = \coupling_{-\normalized} \) in~\eqref{eq:one-sided_linear_coupling} }
    \\
    &= 
      \sigma_{ \normalized\np{ \LevelSet{\lzero}{k} } }
      \tag{ by symmetry of the set \( \LevelSet{\lzero}{k} \) 
      in~\eqref{eq:pseudonormlzero_level_set}
      and of the mapping~\( \normalized \) in~\eqref{eq:normalization_mapping} }
    \\
    &= 
      \SFM{ \delta_{ \LevelSet{\lzero}{k} } }{\CouplingCapra} 
      \eqfinp
      \tag{by~\eqref{eq:one-sided_linear_Fenchel-Moreau_characteristic}}
      \intertext{We now turn to prove the second equality
      in~\eqref{eq:conjugate_delta_l0norm}:}
      \SFM{ \delta_{ \LevelSet{\lzero}{k} } }{-\CouplingCapra}
    &= 
      \sigma_{ \normalized\np{  \LevelSet{\lzero}{k} } }
      \tag{by~\eqref{eq:one-sided_linear_Fenchel-Moreau_characteristic}}
    \\
    &= 
      \sigma_{ \np{ \LevelSet{\lzero}{k} \cap \SPHERE } \cup \{0\} } 
      \tag{by the expression~\eqref{eq:normalization_mapping}
      of the normalization mapping~$\normalized$}
    \\
    &= 
      \sup \ba{ \sigma_{ \LevelSet{\lzero}{k} \cap \SPHERE }, 0 }
      \tag{ as the support function turns a union of sets into a supremum }
    \\
    &= 
      \sup \ba{ \sigma_{ \bigcup_{ {\cardinal{K} \leq k}} \np{ \FlatRR_{K} \cap \SPHERE }  } , 0 }
      \tag{ as \( \LevelSet{\lzero}{k} \cap \SPHERE =
      \bigcup_{ {\cardinal{K} \leq k}} \np{ \FlatRR_{K} \cap \SPHERE } \)
      by~\eqref{eq:level_set_pseudonormlzero}  }
    \\
    &= 
      \sup \ba{ 
      \sup_{ \cardinal{K} \leq k} \sigma_{ \np{ \FlatRR_{K} \cap \SPHERE }  }, 0 }
      \tag{ as the support function turns a union of sets into a supremum }
    \\
    &= 
      \sup \ba{ \EuclidianTopNorm{k}{\cdot}, 0 }
      \tag{ as \( \sup_{ \cardinal{K} \leq k} \sigma_{ \np{ \FlatRR_{K} \cap \SPHERE }  }
      = \EuclidianTopNorm{k}{\cdot} \) by~\eqref{eq:k-norm-from-support} }
    \\
    &= 
      \EuclidianTopNorm{k}{\cdot} 
      \eqfinp
  \end{align*}

  \noindent $\bullet$ Before proving~\eqref{eq:biconjugate_delta_l0norm},
  observe that, by definition~\eqref{eq:normalization_mapping}
  of the normalization mapping~$\normalized$, 
  we have:
  \begin{equation}
    \label{eq:pre-image_normalization_mapping}
    0 \in D \subset \RR^d \Rightarrow
    \normalized^{-1}\np{D} 
    = \normalized^{-1}\bp{ \np{\{0\} \cup \SPHERE} \cap D } 
    = \{0\} \cup \normalized^{-1}\np{\SPHERE \cap D } 
    \eqfinp
  \end{equation}

  \noindent $\bullet$ We prove~\eqref{eq:biconjugate_delta_l0norm}:
  \begin{align*}
    \SFMbi{ \delta_{ \LevelSet{\lzero}{k} } }{\CouplingCapra}
    &=
      \LFM{ \bp{ \SFM{ \delta_{ \LevelSet{\lzero}{k} } }{\CouplingCapra} }}
      \circ \normalized
      \tag{ by the formula~\eqref{eq:CAPRA_Fenchel-Moreau_biconjugate} 
      for the biconjugate }
    \\
    &=
      \LFM{ \bp{ \EuclidianTopNorm{k}{\cdot} } }
      \circ \normalized
      \tag{ by~\eqref{eq:conjugate_delta_l0norm} }
    \\
    &=
      \LFM{ \bp{ \sigma_{ \EuclidianSupport{k}{\BALL} } } }
      \circ \normalized
      \tag{ by~\eqref{eq:norm_dual_norm}, that expresses a norm as 
      the support function of the unit ball of the dual norm} 
    \\
    &=
      \delta_{ \EuclidianSupport{k}{\BALL} } \circ \normalized
      \tag{ as \( \LFM{ \bp{ \sigma_{ \EuclidianSupport{k}{\BALL} } } }=
      \delta_{ \EuclidianSupport{k}{\BALL} } \) since $\EuclidianSupport{k}{\BALL}$
      is closed convex \cite[Theorem 13.2]{Rockafellar:1970} }
    \\
    &=
      \delta_{ \normalized^{-1}\np{\EuclidianSupport{k}{\BALL} } }
      \tag{ by the definition~\eqref{eq:characteristic_function} of a characteristic
      function }
    \\
    &=
      \delta_{ \{0\} \cup \normalized^{-1}\np{\EuclidianSupport{k}{\BALL} \cap \SPHERE } }
      \tag{by~\eqref{eq:pre-image_normalization_mapping}
      since $0\in\EuclidianSupport{k}{\BALL}$ }
    \\
    &=
      \delta_{ \{0\} \cup \normalized^{-1}
      \np{ \LevelSet{\lzero}{k} \cap \SPHERE } }
      \tag{ as \( \EuclidianSupport{k}{\BALL} \cap \SPHERE
      =  \LevelSet{\lzero}{k} \cap \SPHERE \) 
      by~\eqref{eq:level_set_l0_inter_sphere_b} }
    \\
    &=
      \delta_{ \normalized^{-1}\np{\LevelSet{\lzero}{k}}  } 
      \tag{by~\eqref{eq:pre-image_normalization_mapping}
      since $0\in\LevelSet{\lzero}{k}$ }
    \\
    &=
      \delta_{ \LevelSet{\lzero}{k}  } 
      \tag{ as \( \lzero \circ \normalized = \lzero \) 
      by~\eqref{eq:pseudonormlzero_invariance_normalized} }
      \eqfinp
  \end{align*}

  \noindent$\bullet$ We prove~\eqref{eq:conjugate_l0norm}:
  \begin{align*}
    \SFM{ \lzero }{\CouplingCapra}
    &= 
      \SFM{\Bp{ \inf_{l=0,1,\ldots,d} \ba{ \delta_{ \LevelCurve{\lzero}{l} } \UppPlus
      l } } }{\CouplingCapra}
      \tag{ since \(  \lzero = 
      \inf_{l=0,1,\ldots,d} \ba{ \delta_{ \LevelCurve{\lzero}{l} } \UppPlus l } \)
      by using the level curves~\eqref{eq:pseudonormlzero_level_curve} }
    \\
    &= 
      \sup_{l=0,1,\ldots,d} \ba{ 
      \SFM{ \delta_{ \LevelCurve{\lzero}{k} }}{\CouplingCapra}
      \LowPlus \np{-l} }  
      \tag{ as conjugacies, being dualities, turn infima into suprema}
    \\
    &= 
      \sup_{l=0,1,\ldots,d} \ba{ 
      \sigma_{ \normalized\np{\LevelCurve{\lzero}{l} } } 
      \LowPlus \np{-l} }  
      \tag{ as \( \SFM{ \delta_{ \LevelCurve{\lzero}{k} }}{\CouplingCapra}
      = \sigma_{ \normalized\np{\LevelCurve{\lzero}{l} } } \)
      by~\eqref{eq:one-sided_linear_Fenchel-Moreau_characteristic}}
    \\
    &= 
      \sup \Ba{ 0, \sup_{l=1,\ldots,d} \ba{ 
      \sigma_{ \LevelCurve{\lzero}{l} \cap \SPHERE } 
      \LowPlus \np{-l} } }
      \tag{ as \( \sigma_{ \{0\} }=0 \) and \( \normalized\np{\LevelCurve{\lzero}{l} }
      = \LevelCurve{\lzero}{l} \cap \SPHERE \) when $l\geq 1$ 
      by~\eqref{eq:normalization_mapping} }
      \intertext{} 
    &= 
      \sup \Ba{ 0, \sup_{l=1,\ldots,d} \ba{ 
      \sigma_{ \overline{\LevelCurve{\lzero}{l} \cap \SPHERE} } 
      \LowPlus \np{-l} } }
      \tag{ as \( \sigma_{ \Primal } = \sigma_{ \overline{\Primal} } \) 
      for any \( \Primal \subset \RR^d \)
      \cite[Proposition~2.2.1]{Hiriart-Urruty-Lemarechal-I:1993} }
    \\
    &=
      \sup \Ba{ 0, \sup_{l=1,\ldots,d} \ba{ 
      \sigma_{ \LevelSet{\lzero}{l} \cap \SPHERE } \LowPlus \np{-l} } }
      \tag{ as \( \overline{\LevelCurve{\lzero}{l} \cap \SPHERE} =
      \LevelSet{\lzero}{l} \cap \SPHERE \) by~\eqref{eq:closure_level_curve} }
    \\
    &= 
      \sup \Ba{ 0, \sup_{l=1,\ldots,d} \ba{ 
      \sigma_{ \cup_{ \cardinal{K} \leq l} \np{ \FlatRR_{K} \cap \SPHERE }  } \LowPlus \np{-l} } }
      \tag{ as \( \LevelSet{\lzero}{l} \cap \SPHERE 
      = \cup_{ \cardinal{K} \leq l} \np{ \FlatRR_{K} \cap \SPHERE }   \) 
      by~\eqref{eq:level_set_pseudonormlzero} } 
    \\
    &= 
      \sup \Ba{ 0, \sup_{l=1,\ldots,d} \ba{ 
      \sup_{ \cardinal{K} \leq l} \sigma_{ \FlatRR_{K} \cap \SPHERE } \LowPlus \np{-l} } }
      \tag{ as the support function turns a union of sets into a supremum } \\
    &= 
      \sup \Ba{0, 
      \sup_{ l=1,\ldots,d} 
      \Bc{ \EuclidianTopNorm{l}{\dual} -l } }
      \tag{ as \( \sup_{ \cardinal{K} \leq l} \sigma_{ \FlatRR_{K} \cap \SPHERE }
      = \EuclidianTopNorm{l}{\cdot} \) by~\eqref{eq:k-norm-from-support}}
    \\
    &= 
      \sup_{ l=0,1,\ldots,d} 
      \Bc{ \EuclidianTopNorm{l}{\dual} -l } 
      \tag{ using the convention that \( \EuclidianTopNorm{0}{\cdot} = 0 \) }
      \eqfinp
  \end{align*}
  \smallskip

  \noindent$\bullet$ 
  We prove~\eqref{eq:biconjugate_l0norm}.
  It is easy to check that 
  \( \SFMbi{ \lzero }{\CouplingCapra}\np{0}=0=\lzero\np{0} \). 
  Therefore, let $\primal \in \RR^d\backslash\{0\} $ be given and assume that 
  $\lzero\np{\primal}=l \in \ba{1,\ldots,d}$. 
  We consider the mapping $\phi: ]0,+\infty[ \to \RR$ defined by
  \begin{equation}
    \phi(\lambda) =
    \frac{ \proscal{\primal}{ \lambda \primal} }{ \norm{\primal} }
    + \Bp{ - \sup \Ba{0, \sup_{j=1,\ldots,d} 
        \Bc{ \EuclidianTopNorm{j}{\lambda\primal} - j
        } } }
    \eqsepv \forall \lambda > 0 \eqfinv
    \label{eq:phi}
  \end{equation}
  and we are going to show that 
  \( \lim_{\lambda \to +\infty} \phi(\lambda) =l \).
  We have 
  \begin{align*}
    \phi(\lambda) 
    &=
      \lambda \norm{\primal}
      + \Bp{ - \sup \Ba{0, \sup_{j=1,\ldots,d} 
      \Bc{ \EuclidianTopNorm{j}{\lambda\primal} - j
      } } }
      \tag{ by definition~\eqref{eq:phi} of~$\phi$ }
    \\
    &= 
      \lambda \EuclidianTopNorm{l}{\primal}
      + \inf \Ba{0, -\sup_{j=1,\ldots,d} 
      \Bc{ \lambda \EuclidianTopNorm{j}{\primal} - j } }
      \tag{ as $\norm{\primal}=\EuclidianTopNorm{l}{\primal}$ 
      when $\lzero\np{\primal}=l$
      by~\eqref{eq:level_curve_l0_characterization} } 
    \\
    &= 
      \inf \Ba{ \lambda \EuclidianTopNorm{l}{\primal} ,
      \lambda \EuclidianTopNorm{l}{\primal} + 
      \inf_{j=1,\ldots,d} \Bp{ -\Bc{ \lambda \EuclidianTopNorm{j}{\primal} - j } } }
    \\
    &= 
      \inf \Ba{ \lambda \EuclidianTopNorm{l}{\primal} ,
      \inf_{j=1,\ldots,d} \Bp{ \lambda \bp{ \EuclidianTopNorm{l}{\primal} -
      \EuclidianTopNorm{j}{\primal} } + j } }
    \\
    &= 
      \inf \Ba{ \lambda \EuclidianTopNorm{l}{\primal} ,
      \inf_{j=1,\ldots,l-1} \Bp{ \lambda \bp{ \EuclidianTopNorm{l}{\primal} -
      \EuclidianTopNorm{j}{\primal} } + j } , 
      \inf_{j=l,\ldots,d} \Bp{ \lambda \bp{ \EuclidianTopNorm{l}{\primal} -
      \EuclidianTopNorm{j}{\primal} } + j } }
    \\
    &= 
      \inf \Ba{ \lambda \EuclidianTopNorm{l}{\primal} ,
      \inf_{j=1,\ldots,l-1} \Bp{ \lambda \bp{ \EuclidianTopNorm{l}{\primal} -
      \EuclidianTopNorm{j}{\primal} } + j } , l }
  \end{align*}
  as $\EuclidianTopNorm{j}{\primal}= \EuclidianTopNorm{l}{\primal}$ for
  $j \geq l$ by~\eqref{eq:level_curve_l0_characterization}. 
  Let us show that the two first terms in the infimum
  go to $+\infty$ when \( \lambda \to +\infty \).
  The first term \( \lambda \EuclidianTopNorm{l}{\primal} \) goes to $+\infty$ 
  because \( \EuclidianTopNorm{l}{\primal}=\norm{\primal}>0 \) by assumption
  ($\primal \neq 0$).
  The second term \( \inf_{j=1,\ldots,l-1} \Bp{ \lambda \bp{ \EuclidianTopNorm{l}{\primal} -
      \EuclidianTopNorm{j}{\primal} } + j } \)
  also goes to $+\infty$ because 
  $\lzero\np{\primal}=l $, so that 
  \( \norm{\primal}=\EuclidianTopNorm{l}{\primal} > 
  \EuclidianTopNorm{j}{\primal} \)
  for $j=1,\ldots,l-1$
  by~\eqref{eq:level_curve_l0_characterization}.
  Therefore, 
  \( \lim_{\lambda \to +\infty} \phi(\lambda) = 
  \inf \{ +\infty, +\infty, l \}=l \).
  This concludes the proof since 
  \begin{align*}
    l = \lim_{\lambda \to +\infty} \phi(\lambda) 
    & \leq 
      \sup_{\dual \in \RR^d } \bgp{ \frac{ \proscal{\primal}{\dual} }{ \norm{\primal} }
      \LowPlus 
      \Bp{ - \sup \Ba{0, \sup_{j=1,\ldots,d} 
      \Bc{ \EuclidianTopNorm{j}{\dual} - j } } } }
      \tag{ by definition~\eqref{eq:phi} of~$\phi$ }
    \\
    &=
      \sup_{\dual \in \RR^d } \bgp{ \frac{ \proscal{\primal}{\dual} }{ \norm{\primal} }
      \LowPlus 
      \Bp{ - \sup_{j=0,1,\ldots,d} 
      \Bc{ \EuclidianTopNorm{j}{\dual} - j } } }
      \tag{ by the convention \( \EuclidianTopNorm{0}{\cdot} = 0 \) }
    \\
    &=
      \sup_{\dual \in \RR^d } \bgp{ \frac{ \proscal{\primal}{\dual} }{ \norm{\primal} }
      \LowPlus \Bp{ - \SFM{ \lzero }{\CouplingCapra}\np{\dual} } }
      \tag{ by the formula~\eqref{eq:conjugate_l0norm} for 
      \( \SFM{ \lzero }{\CouplingCapra} \) }
    \\
    &=
      \SFMbi{ \lzero }{\CouplingCapra}\np{\primal}
      \tag{ by the biconjugate formula~\eqref{eq:Fenchel-Moreau_biconjugate} }
    \\
    & \leq 
      \lzero\np{\primal} 
      \tag{ by~\eqref{eq:galois-cor} giving 
      \( \SFMbi{ \lzero }{\CouplingCapra} \leq \lzero \)}
    \\
    & = l 
      \tag{ by assumption }
      \eqfinp
  \end{align*}
  Therefore, we have obtained \( l=\SFMbi{ \lzero
  }{\CouplingCapra}\np{\primal}=\lzero\np{\primal}\).
  \smallskip

  This ends the proof.
\end{proof}

In the next Section, we present a (rather unexpected) consequence of the just established
property that \( \SFMbi{ \lzero }{\CouplingCapra} = \lzero \).

\section{Hidden convexity in the pseudonorm~$\lzero$}
\label{Hidden_convexity_in_the_pseudonorm}

In~\S\ref{The_pseudonorm_lzero_coincides_on_the_Euclidian_unit_sphere},
we show that there exists a proper convex lsc function on~\( \RR^d \) 
which takes the same values as the \lzeropseudonorm\
on the Euclidian unit sphere~$\SPHERE$.
This property of hidden convexity somehow comes as a surprise as the \lzeropseudonorm\ is a highly
nonconvex function of combinatorial nature.
Then, we provide various expression for the underlying
proper convex lsc function and,
in~\S\ref{Graphical_representations_of_the_proper_convex_lsc_function}, 
we display mathematical expressions 
and graphical representations in the two-dimensional case.

\subsection{Hidden convexity in the pseudonorm~$\lzero$}
\label{The_pseudonorm_lzero_coincides_on_the_Euclidian_unit_sphere}

We introduce the function \( {\cal L}_0 : \RR^d \to \barRR \) defined by 
\begin{equation}
  {\cal L}_0= 
  \LFMr{ \Bp{ \sup_{l=0,1,\ldots,d} \Bc{ \EuclidianTopNorm{l}{\cdot} -l } } }
  \eqfinp
  \label{eq:definition_calL0}
\end{equation}

\begin{theorem}
  \label{th:pseudonormlzero_convex}

  The function \( {\cal L}_0 \) 
  in~\eqref{eq:definition_calL0} is a proper convex lsc function on~$\RR^d$.
  The pseudo\-norm~$\lzero$ coincides, on the Euclidian unit sphere~$\SPHERE$
  of~$\RR^d$, with the function~${\cal L}_0$, that is,
  \begin{equation}
    \lzero\np{\primal} = {\cal L}_0\np{\primal} 
    \eqsepv 
    \forall \primal \in \SPHERE 
    \eqfinp
    \label{eq:pseudonormlzero_convex_sphere}
  \end{equation}
  As a consequence, the pseudonorm~$\lzero$ displays hidden convexity,
  as it can be expressed as the 
  composition of the proper convex lsc function~${\cal L}_0$
  in~\eqref{eq:definition_calL0}
  with the 
  normalization mapping~$\normalized$ in~\eqref{eq:normalization_mapping}:
  \begin{equation}
    \lzero\np{\primal} = 
    {\cal L}_0 \Bp{ \frac{ \primal }{ \norm{\primal} } } 
    \eqsepv 
    \forall \primal \in \RR^d\backslash\{0\}
    \eqfinp 
    \label{eq:pseudonormlzero_convex_normalization_mapping}
  \end{equation}
  The proper convex lsc function~${\cal L}_0$ has the property
  \begin{equation}
    {\cal L}_0\bp{\np{\primal_1,\ldots, \primal_d}}=
    {\cal L}_0\bp{\np{\module{\primal_1},\ldots, \module{\primal_d}}}
    \eqsepv \forall \np{\primal_1,\ldots, \primal_d} \in \RR^d \eqfinp
    \label{eq:calL0_module}
  \end{equation}
\end{theorem}

\begin{proof}
  First, 
  it is easily seen that the closed convex function~${\cal L}_0$
  in~\eqref{eq:definition_calL0} is proper lsc (see Footnote~\ref{ft:closed_convex_function}).
  \smallskip

  Second, we prove~\eqref{eq:pseudonormlzero_convex_sphere}.
  For \( \primal \in \SPHERE \), we have 
  \begin{align*}
    \lzero\np{\primal} 
    &=
      \SFMbi{ \lzero }{\CouplingCapra}\np{\primal}
      \tag{ by~\eqref{eq:biconjugate_l0norm} }
    \\
    &=
      \sup_{ \dual \in \RR^d } \Bp{ \CouplingCapra\np{\primal,\dual} 
      \LowPlus \bp{ -\SFM{ \lzero }{\CouplingCapra}\np{\dual} } } 
      \tag{ by the biconjugate formula~\eqref{eq:Fenchel-Moreau_biconjugate} }
    \\
    &=
      \sup_{ \dual \in \RR^d } \Bp{ \proscal{\primal}{\dual} 
      \LowPlus \bp{ -\SFM{ \lzero }{\CouplingCapra}\np{\dual} } } 
      \tag{ by~\eqref{eq:Euclidian_coupling_CAPRAC} with
      \( \norm{\primal}=1 \) since \( \primal \in \SPHERE \) } 
    \\
    &=
      \sup_{ \dual \in \RR^d } \bgp{ \proscal{\primal}{\dual} 
      \LowPlus \Bp{ -\bp{ 
      \sup_{l=0,1,\ldots,d} \Bc{ \EuclidianTopNorm{l}{\dual} -l }
      } }
      }
      \tag{ by~\eqref{eq:conjugate_l0norm} }
    \\
    &=
      \LFMr{ \Bp{ 
      \sup_{l=0,1,\ldots,d} \Bc{ \EuclidianTopNorm{l}{\dual} -l }
      } }\np{\primal}  
      \tag{by the expression~\eqref{eq:Fenchel_conjugate} of the Fenchel conjugate}
    \\
    &=
      {\cal L}_0\np{\primal} 
      \tag{by~\eqref{eq:definition_calL0} }
      \eqfinp 
  \end{align*}

  Third, the equality~\eqref{eq:pseudonormlzero_convex_normalization_mapping}
  is an easy consequence of the 
  property~\eqref{eq:pseudonormlzero_invariance_normalized}
  that the pseudonorm~$\lzero$ is invariant along
  any open ray of~$\RR^d$.
  \smallskip

  Fourth, we prove~\eqref{eq:calL0_module}.
  For this purpose, we take any \( \epsilon \in \{-1,1\}^d \)
  and we consider the symmetry~\( \tilde\epsilon \) of~$\RR^d$, 
  defined by \(     \tilde\epsilon\np{\primal_1,\ldots, \primal_d}=
  \np{\epsilon_1\primal_1,\ldots, \epsilon_d\primal_d} \), 
  for all \( \np{\primal_1,\ldots, \primal_d} \in \RR^d \).
  We will show that the proper convex lsc function~${\cal L}_0$ is
  invariant under the symmetry~\( \tilde\epsilon \), hence
  satisfies~\eqref{eq:calL0_module}.
  Indeed, for any \( \primal \in \RR^d \), 
  we have
  
  \begin{align*}
    {\cal L}_0\np{\tilde\epsilon\primal}   
    &=
      \LFMr{ \Bp{ \sup_{l=0,1,\ldots,d} \Bc{ \EuclidianTopNorm{l}{\cdot} -l } } }
      \np{\tilde\epsilon\primal}   
      \tag{ by~\eqref{eq:definition_calL0} }
    \\
    &=
      \sup_{ \dual \in \RR^d } \bgp{ \proscal{\tilde\epsilon\primal}{\dual} 
      \LowPlus \Bp{ -\bp{ 
      \sup_{l=0,1,\ldots,d} \Bc{ \EuclidianTopNorm{l}{\dual} -l }
      } }
      }
      \tag{by the expression~\eqref{eq:Fenchel-Moreau_reverse_conjugate} 
      of the reverse Fenchel conjugate}
    \\
    &=
      \sup_{ \dual \in \RR^d } \bgp{ \proscal{\primal}{\tilde\epsilon\dual} 
      \LowPlus \Bp{ -\bp{ 
      \sup_{l=0,1,\ldots,d} \Bc{ \EuclidianTopNorm{l}{\dual} -l }
      } }
      }
      \tag{ as easily seen }
    \\
    &=
      \sup_{ \dual \in \RR^d } \bgp{ \proscal{\primal}{\tilde\epsilon\dual} 
      \LowPlus \Bp{ -\bp{ 
      \sup_{l=0,1,\ldots,d} \Bc{ \EuclidianTopNorm{l}{\tilde\epsilon\dual} -l }
      } }
      }
      \intertext{as \( \EuclidianTopNorm{0}{\cdot} \equiv 0 \) (by convention)
      and all norms \( \EuclidianTopNorm{l}{\cdot}, l=1,\ldots,d \), 
      are invariant under the symmetry~$\tilde\epsilon$}
    &=
      \LFMr{ \Bp{ \sup_{l=0,1,\ldots,d} \Bc{ \EuclidianTopNorm{l}{\cdot} -l } } }
      \np{\primal}   
      \tag{ as \( \tilde\epsilon^{-1}(\RR^d)=\RR^d \) } 
    \\
    &=
      {\cal L}_0\np{\primal} 
      \eqfinp 
      \tag{ by~\eqref{eq:definition_calL0} }
  \end{align*}
  
  This ends the proof.
\end{proof}

Now, we provide three expressions for the 
proper convex lsc function~${\cal L}_0$ in~\eqref{eq:definition_calL0}.

\begin{proposition}
  \label{pr:calL0}
  The proper convex lsc function~${\cal L}_0$ in~\eqref{eq:definition_calL0}
  can also be characterized 
  \begin{itemize}
  \item 
    either by its epigraph
    \begin{equation}
      \epigraph\,{\cal L}_0 = 
      \closedconvexhull\Bp{ \bigcup_{l=0}^d 
        \EuclidianSupport{l}{\BALL} \times [l,+\infty[ } 
      \eqfinv
      \label{eq:pseudonormlzero_convex_epigraph}
    \end{equation}
    where \( \EuclidianSupport{0}{\BALL}=\{0\} \) (by convention)
    and \( \EuclidianSupport{1}{\BALL} 
    \subset \cdots \subset
    \EuclidianSupport{l-1}{\BALL} \subset \EuclidianSupport{l}{\BALL} 
    \subset \cdots \subset \EuclidianSupport{d}{\BALL} = \BALL \) denote the unit
    balls associated with the $l$-support norms defined in~\eqref{eq:support_norm}
    for $l=1,\ldots,d$,
  \item 
    or, as the largest proper convex lsc function
    below the (extended integers valued) function~$L_0$ defined by
    \begin{equation}
      L_0\np{\primal} =
      \begin{cases}
        0 & \text{ if } \primal=0,\\
        l & \text{ if } \primal \in
        \EuclidianSupport{l}{\BALL} \backslash \EuclidianSupport{l-1}{\BALL}
        \eqsepv l=1,\ldots,d, \\
        +\infty & \text{ if } \primal \not\in \EuclidianSupport{d}{\BALL} = \BALL,
      \end{cases}
      \label{eq:L_0}
    \end{equation}
  \item 
    or also by the expression
    \begin{equation}
      {\cal L}_0\np{\primal} = 
      \min_{ \substack{%
          \primal^{(1)} \in \RR^d, \ldots, \primal^{(d)} \in \RR^d 
          \\
          \sum_{ l=1 }^{ d } \EuclidianSupportNorm{l}{\primal^{(l)}} \leq 1 
          \\
          \sum_{ l=1 }^{ d } \primal^{(l)} = \primal
        } }
      \sum_{ l=1 }^{ d } l \EuclidianSupportNorm{l}{\primal^{(l)}} 
      \eqsepv \forall \primal \in \RR^d 
      \eqfinp 
      \label{eq:pseudonormlzero_convex_minimum}
    \end{equation}
  \end{itemize}
\end{proposition}

\begin{proof} \quad
  
  \noindent$\bullet$ 
  First, we prove that the epigraph of~${\cal L}_0$
  in~\eqref{eq:definition_calL0} is given
  by~\eqref{eq:pseudonormlzero_convex_epigraph}.
  Indeed, we have that 
  \begin{align*}
    \epigraph\, {\cal L}_0
    &=
      \epigraph 
      \LFMr{ \Bp{ \sup_{l=0,1,\ldots,d} \Bc{ \EuclidianTopNorm{l}{\cdot} -l } } }
      \tag{ by~\eqref{eq:definition_calL0} }
    \\
    &=
      \closedconvexhull\Bp{ \bigcup_{l=0}^d 
      \epigraph \LFMr{ \Bc{ \EuclidianTopNorm{l}{\cdot} -l } } }
      \tag{ by~\cite[Theorem 16.5]{Rockafellar:1970} } 
    \\
    &=
      \closedconvexhull\Bp{ \bigcup_{l=0}^d 
      \epigraph \LFMr{ \Bc{ \sigma_{ \EuclidianSupport{l}{\BALL} } -l } } }
      \tag{ by~\eqref{eq:symmetric_gauge_norm_egale_dual_norm} }
    \\
    &=
      \closedconvexhull\Bp{ \bigcup_{l=0}^d 
      \epigraph \Bc{ \delta_{ \EuclidianSupport{l}{\BALL} } +l } }
      \tag{ as \( \LFMr{ \Bc{ \sigma_{ \EuclidianSupport{l}{\BALL} } -l } }
      = \delta_{ \EuclidianSupport{l}{\BALL} } +l \) } 
    \\
    &=
      \closedconvexhull\Bp{ \bigcup_{l=0}^d 
      \EuclidianSupport{l}{\BALL} \times [l,+\infty[ }
      \tag{ as is easily concluded}
      \eqfinp
  \end{align*}
  \smallskip

  \noindent$\bullet$
  Second, we prove that the function~${\cal L}_0$  in~\eqref{eq:definition_calL0}
  is the largest proper convex lsc function
  below the function~$L_0$ defined by~\eqref{eq:L_0}.
  Indeed, we have that 
  \begin{align*}
    {\cal L}_0
    &= 
      \LFMr{ \Bp{ \sup_{l=0,1,\ldots,d} \Bc{ \EuclidianTopNorm{l}{\cdot} -l } } }
      \tag{ by~\eqref{eq:definition_calL0} }
    \\
    &=
      \LFMr{ \Bp{ \sup_{l=0,1,\ldots,d} \Bc{ \sigma_{ \EuclidianSupport{l}{\BALL} } -l } } }
      \tag{ by~\eqref{eq:symmetric_gauge_norm_egale_dual_norm} }
    \\
    &=
      \LFMr{ \Bp{ \sup_{l=0,1,\ldots,d} 
      \LFM{ \Bc{ \delta_{ \EuclidianSupport{l}{\BALL} } +l } } } } 
      \tag{ as \( \LFM{ \Bc{ \delta_{ \EuclidianSupport{l}{\BALL} } +l } }
      = \sigma_{ \EuclidianSupport{l}{\BALL} } -l \) } 
    \\
    &=
      \LFMr{ \Bp{ \LFM{ \Bc{ \inf_{l=0,1,\ldots,d} 
      \bc{ \delta_{ \EuclidianSupport{l}{\BALL} } +l } } } } }
      \tag{ as conjugacies, being dualities, turn infima into suprema }
    \\
    &=
      \LFMbi{ \Bp{ \inf_{l=0,1,\ldots,d} 
      \Bc{ \delta_{ \EuclidianSupport{l}{\BALL} } +l } } }
      \tag{ by definition~\eqref{eq:Fenchel_biconjugate} of the Fenchel biconjugate}
    \\
    &=
      \LFMbi{ L_0 }
  \end{align*}
  as it is easy to establish that the function \( \inf_{l=0,1,\ldots,d} 
  \Bc{ \delta_{ \EuclidianSupport{l}{\BALL} } +l } \) coincides 
  with the function~$L_0$ defined by~\eqref{eq:L_0}.
  Indeed, it is deduced from~\eqref{eq:dual_support_norm_unit_ball} that 
  \( \{0\}=\EuclidianSupport{0}{\BALL} \subset \EuclidianSupport{1}{\BALL} 
  \subset \cdots \subset
  \EuclidianSupport{l-1}{\BALL} \subset \EuclidianSupport{l}{\BALL} 
  \subset \cdots \subset \EuclidianSupport{d}{\BALL} = \BALL \).
  Finally, from \( {\cal L}_0=\LFMbi{ L_0 } \), we conclude that 
  ${\cal L}_0$ is the largest proper convex lsc function
  below the function~$L_0$. 
  \smallskip

  \noindent$\bullet$
  Third, we prove that ${\cal L}_0$ in~\eqref{eq:definition_calL0} is given
  by~\eqref{eq:pseudonormlzero_convex_minimum}.
  For this purpose, we use a general formula 
  \cite[Corollary~2.8.11]{Zalinescu:2002}
  for the Fenchel conjugate of the supremum of proper convex functions
  \(\fonctionprimal_l  : \RR^d \to \barRR \), $l=0,1,\ldots,d$:
  \begin{equation}
    \bigcap_{l=0,1,\ldots,d} \dom\,\fonctionprimal_l \neq \emptyset
    \Rightarrow
    \LFM{ \bp{ \sup_{l=0,1,\ldots,d} \fonctionprimal_l } }
    =
    \min_{\lambda \in \Delta_{d+1}} 
    \LFM{ \Bp{ \sum_{ l=0}^{d} \lambda_l \fonctionprimal_l } }
    \eqfinv
    \label{eq:Fenchel_conjugate_of_the_supremum_of_proper_convex_functions}
  \end{equation}
  where \( \Delta_{d+1} \) is the simplex of~$\RR^d$.
  As the functions
  \(\fonctionprimal_l = \EuclidianTopNorm{l}{\cdot} -l \) 
  are proper convex, we obtain
  
  \begin{align*}
    {\cal L}_0
    &= 
      \LFMr{ \Bp{ \sup_{l=0,1,\ldots,d} \Bc{ \EuclidianTopNorm{l}{\cdot} -l } } }
      \tag{ by~\eqref{eq:definition_calL0} }
    \\
    &=
      \LFMr{ \Bp{ \sup_{l=0,1,\ldots,d} \Bc{ \sigma_{ \EuclidianSupport{l}{\BALL} } -l } } }
      \tag{ by~\eqref{eq:symmetric_gauge_norm_egale_dual_norm} }
    \\
    &=
      \min_{\lambda \in \Delta_{d+1}} 
      \LFM{ \Bp{ \sum_{ l=0}^{d} \lambda_l \Bc{ \sigma_{ \EuclidianSupport{l}{\BALL} } -l 
      } } } 
      \tag{ by~\eqref{eq:Fenchel_conjugate_of_the_supremum_of_proper_convex_functions} }
    \\
    &=
      \min_{\lambda \in \Delta_{d+1}} 
      \LFM{ \Bp{ { \sigma_{ \sum_{ l=0}^{d} \lambda_l \EuclidianSupport{l}{\BALL} }
      - \sum_{ l=0}^{d} \lambda_l l 
      } } }
      \intertext{as, for all $l=0,\ldots,d$, 
      \( \lambda_l \sigma_{ \EuclidianSupport{l}{\BALL} } = 
      \sigma_{ \lambda_l \EuclidianSupport{l}{\BALL} } \) since \( \lambda_l \geq 0 \), 
      and then using the well-known property that the support function of 
      a Minkowski sum of subsets is the sum of the support functions of the individual subsets \cite[p.~113]{Rockafellar:1970} }
      \intertext{}
    &=
      \min_{\lambda \in \Delta_{d+1}} \Bp{ 
      \delta_{ \sum_{ l=0}^{d} \lambda_l \EuclidianSupport{l}{\BALL} } 
      +
      \sum_{ l=0}^{d} \lambda_l l} 
      \eqfinp
      \tag{ as \( \LFMr{ \Bc{ \sigma_{ \Convex} -t } }
      = \delta_{ \Convex } +t \) for any closed convex subset~$\Convex$} 
  \end{align*}
  Therefore, for all \( \primal \in \RR^d \), we have 
  \begin{subequations}
    \begin{align}
      {\cal L}_0\np{\primal} 
      &= 
        \min_{ \substack{%
        \lambda \in \Delta_{d+1} 
      \\
      \primal \in \sum_{ l=0 }^{ d } \lambda_l \EuclidianSupport{l}{\BALL} 
      } } 
      \sum_{ l=0 }^{ d } \lambda_l l 
      \eqsepv
      \\
      &=
        \min_{ \substack{%
        z^{(1)} \in \EuclidianSupport{1}{\BALL}, 
        \ldots, z^{(d)} \in \EuclidianSupport{d}{\BALL} 
      \\
      \lambda_1 \geq 0, \ldots, \lambda_d \geq 0
      \\
      \sum_{ l=1 }^{ d } \lambda_l \leq 1 
      \\
      \sum_{ l=1 }^{ d } \lambda_l z^{(l)} = \primal
      } }  
      \sum_{ l=1 }^{ d } \lambda_l l 
      \intertext{by ignoring~$\lambda_0 \geq 0$ since 
      \( \EuclidianSupport{0}{\BALL} =\{0\} \) by convention }
      &=
        \min_{ \substack{%
        s^{(1)} \in \EuclidianSupport{1}{\SPHERE}, 
        \ldots, s^{(d)} \in \EuclidianSupport{d}{\SPHERE} 
      \\
      \mu_1 \geq 0, \ldots, \mu_d \geq 0
      \\
      \sum_{ l=1 }^{ d } \mu_l \leq 1 
      \\
      \sum_{ l=1 }^{ d } \mu_l s^{(l)} = \primal
      } }  
      \sum_{ l=1 }^{ d } \mu_l l 
      \intertext{where \( \EuclidianSupport{l}{\SPHERE} \) is the unit sphere
      of the 
      $l$-support norm~\( \EuclidianSupportNorm{l}{\cdot} \), and
      the inequality $\leq$ is obvious as 
      \( \EuclidianSupport{l}{\SPHERE} \subset \EuclidianSupport{l}{\BALL} \)
      for all $l=1,\ldots,d$;
      the inequality $\geq$ comes from putting,
      for $l=1,\ldots,d$, 
      \(  \mu_l = \lambda_l \EuclidianSupportNorm{l}{z^{(l)}} \)
      and observing that 
      i) 
      there exist \( s^{(l)} \in \EuclidianSupport{l}{\SPHERE} \) such that 
      \( \lambda_l z^{(l)} = \mu_l s^{(l)} \) 
      (take any $s^{(l)}$ when $z^{(l)}=0$ and
      \( s^{(l)}=\frac{z^{(l)}}{ \EuclidianSupportNorm{l}{z^{(l)}} } \) 
      when $z^{(l)} \neq 0$)
      ii) 
      \( \sum_{ l=1 }^{ d } \lambda_l l \geq
      \sum_{ l=1 }^{ d } \lambda_l \EuclidianSupportNorm{l}{z^{(l)}} l
      = \sum_{ l=1 }^{ d } \mu_l l \) 
      because \( \EuclidianSupportNorm{l}{z^{(l)}} \leq 1 \) 
      }
      &=
        \min_{ \substack{%
        \primal^{(1)} \in \RR^d, \ldots, \primal^{(d)} \in \RR^d 
      \\
      \sum_{ l=1 }^{ d } \EuclidianSupportNorm{l}{\primal^{(l)}} \leq 1 
      \\
      \sum_{ l=1 }^{ d } \primal^{(l)} = \primal
      } }
      \sum_{ l=1 }^{ d } \EuclidianSupportNorm{l}{\primal^{(l)}} l 
      \eqfinv
    \end{align}
    by putting \( \primal^{(l)} = \mu_l s^{(l)}\),
    for all $l=1,\ldots,d$.
  \end{subequations}
  \smallskip

  This ends the proof.
\end{proof}

With Proposition~\ref{pr:calL0}, we dispose of expressions that 
make it possible to obtain more involved formulas for the function~${\cal L}_0$
in~\eqref{eq:definition_calL0}.
In particular, 
we will now obtain graphical representations
and mathematical formulas for the proper convex lsc function~${\cal L}_0$ on~$\RR^2$.

\subsection{Graphical representations of the function~${\cal  L}_0$ on~$\RR^2$}
\label{Graphical_representations_of_the_proper_convex_lsc_function}

In dimension~$d=1$, it is easily computed that the function ${\cal L}_0$ 
in~\eqref{eq:definition_calL0} is 
the absolute value function~\( \module{\cdot} \) on the segment~\( [-1,1] \)
and $+\infty$ outside the segment~\( [-1,1] \).
The pseudonorm~$\lzero$ coincides with~${\cal L}_0$ 
on the one-dimensional unit sphere~$\{-1,1\}$ --- but also with any convex function taking the value~$1$
on~$\{-1,1\}$ (the function~\( \module{\cdot} \),
the constant function~$1$, etc.). 
\smallskip

In dimension~$d=2$, the function~${\cal L}_0$ 
in~\eqref{eq:definition_calL0} is,
by Proposition~\ref{pr:calL0},
the largest proper convex lsc function which is below the function
which takes the value~$0$ on the zero~$(0,0)$, 
the value~$1$ on the unit lozenge of~$\RR^2$
deprived of~$(0,0)$, 
and the value~$2$ on the unit disk of~$\RR^2$
deprived of the unit lozenge
(see Proposition~\ref{pr:calL0}). 
As a consequence, the graph of~${\cal L}_0$ 
contains segments (in~$\RR^3$) that join 
the zero~$(0,0,0)$ of the horizontal plane at height~$z=0$ 
with the unit lozenge of the horizontal plane at height~$z=1$,
and this latter 
with the unit circle of the horizontal plane at height~$z=2$.
In Figure~\ref{fig:graph1}, 
we have displayed two views of 
the topological closure of the graph of~${\cal L}_0$.
As the function~${\cal L}_0$ is not continuous 
at the four extremal points --- 
$(0,1)$, $(1,0)$, $(0,-1)$, $(-1,0)$ --- of the unit lozenge,
it is delicate to depict the graph and easier to do so for its
topological closure. 
\begin{figure}
  \begin{center}
    \includegraphics[width=6cm]{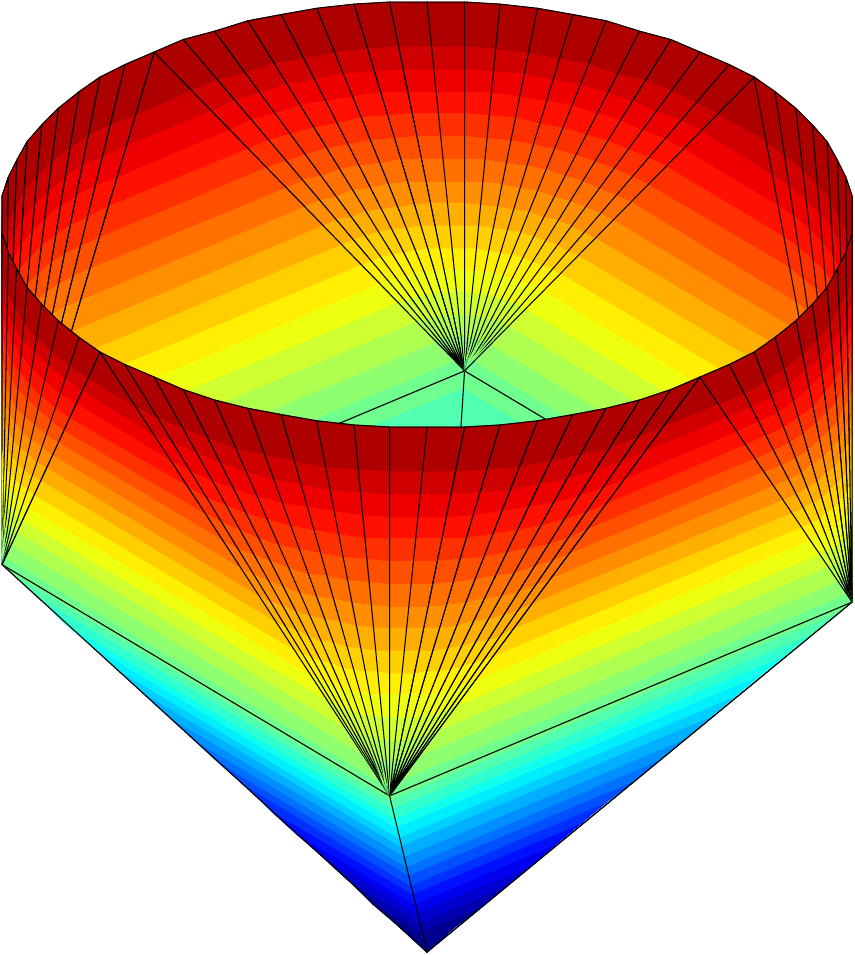}\hspace{2cm}
    \includegraphics[width=6cm]{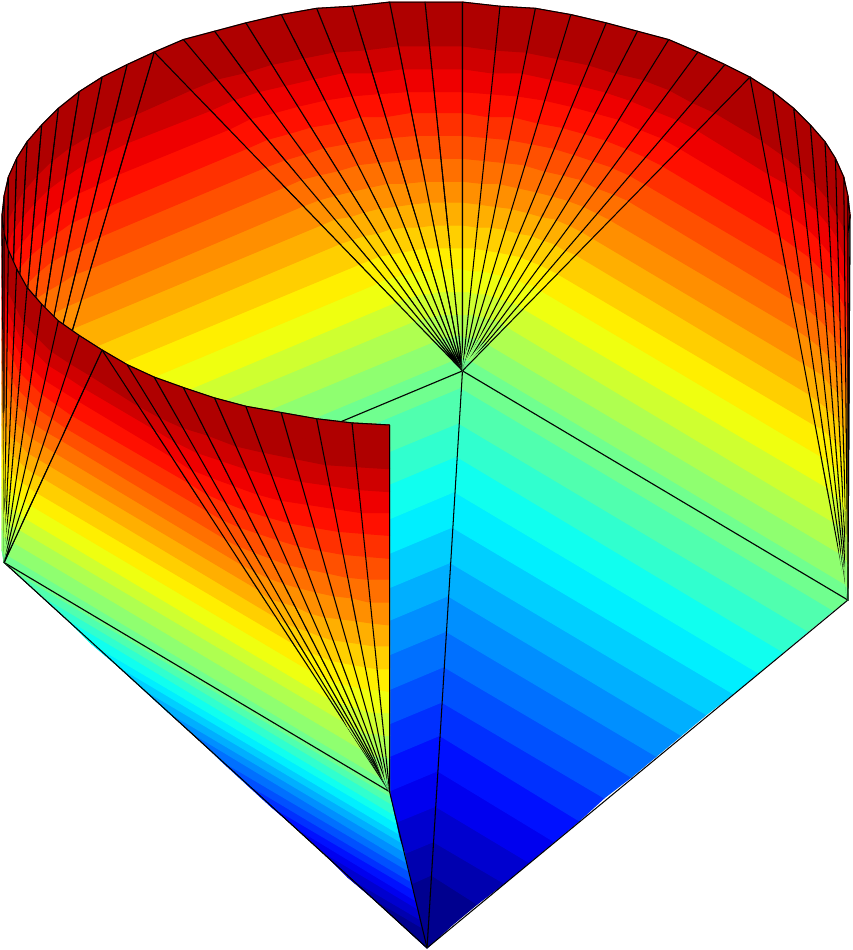}   
  \end{center}
  \caption{\label{fig:graph1}
    Topological closure of the graph, between heights~$z=0$ and $z=2$,
    of the proper convex lsc function~${\cal L}_0$ which coincides, 
    on the Euclidian unit sphere~$\SPHERE$, with the \lzeropseudonorm}
\end{figure}
In dimension~$d=2$, the function~${\cal L}_0$ 
in~\eqref{eq:definition_calL0} is 
given by the following explicit formulas
(see also Figure~\ref{fig:l0-disk}). 

\begin{figure}
  \begin{center}


    \begin{tikzpicture}[scale=6,cap=round]
      \def\unsursqrt2{0.7071}
      \def\sqrt3{1.7321}
      \colorlet{anglecolor}{green!50!black}
      \colorlet{sincolor}{red}
      \colorlet{tancolor}{orange!80!black}
      \colorlet{coscolor}{blue}

      \tikzstyle{axes}=[]
      \tikzstyle{important line}=[very thick]
      \tikzstyle{information text}=[rounded corners,fill=red!10,inner sep=1ex]


      \draw[blue] (0,0) -- (1,0) arc(0:90:1) ;
      
      \begin{scope}[style=axes]
        \draw[->] (-0.1,0) -- (1.2,0) node[right] {$\primal_1$};
        \draw[->] (0,-0.1) -- (0,1.2) node[above] {$\primal_2$};

        \foreach \x/\xtext in { 0.5/\frac{1}{2}, 1}
        \draw[xshift=\x cm] (0pt,1pt) -- (0pt,-1pt) node[below,fill=white] {$\xtext$};

        \foreach \y/\ytext in {0.5/\frac{1}{2}, 1}
        \draw[yshift=\y cm] (1pt,0pt) -- (-1pt,0pt) node[left,fill=white] {$\ytext$};
      \end{scope}


      
      \draw[blue] (0,1) --  (1,0);
      \draw[blue] (0,1) --  (\unsursqrt2,\unsursqrt2);
      \draw[blue] (1,0) --  (\unsursqrt2,\unsursqrt2);

      \draw (1/4,0.93) node[anglecolor] {\eqref{eq:calL0_d=2_inside_nail_up}};
      \draw (0.93,1/4) node[anglecolor] {\eqref{eq:calL0_d=2_inside_nail_down}};
      \draw (1/2+1/8,1/2+1/8) node[anglecolor] {\eqref{eq:calL0_d=2_inside_triangle}};
      \draw (1/4,1/2) node[anglecolor] {\eqref{eq:calL0_d=2_inside_lozenge}};
    \end{tikzpicture}

    \caption{\label{fig:l0-disk}Companion figure for Proposition~\ref{pr:calL0_d=2}}
  \end{center}
\end{figure}

\begin{proposition}
  In dimension~$d=2$, the function~${\cal L}_0$ 
  in~\eqref{eq:definition_calL0} is 
  given by 
  \begin{subequations}
    \begin{numcases}{ {\cal L}_0\np{\primal_1,\primal_2} = }
      + \infty & \text{if } \( \primal_1^2 + \primal_2^2 > 1 \eqsepv \)
      \\
      1 & \text{if } 
      \( \np{\primal_1,\primal_2} \in \{ (1,0), (0,1), (-1,0), (0,-1) \} \eqsepv \)
      \\
      2 & \text{if } \( \primal_1^2 + \primal_2^2 = 1 \mtext{and}
      \np{\primal_1,\primal_2} \not \in \{ (1,0), (0,1), (-1,0), (0,-1) \} \eqsepv \)
    \end{numcases}
    and, for any \( \np{\primal_1,\primal_2} \) such that 
    \( \primal_1^2 + \primal_2^2 < 1 \) by 
    \begin{numcases}{ {\cal L}_0\np{\primal_1,\primal_2} = }
      \module{\primal_1}+\module{\primal_2} 
      & \text{if } 
      \( \module{\primal_1}+\module{\primal_2} \leq 1 \eqfinv \)
      \label{eq:calL0_d=2_inside_lozenge}
      \\ 
      \frac{ \module{\primal_1}+\module{\primal_2} -2 + \sqrt{2} }%
      {\sqrt{2}-1} 
      & 
      \text{if } 
      \( \begin{cases}
        \np{\sqrt{2}-1}\module{\primal_1}+\module{\primal_2}
        < 1 < \module{\primal_1}+\module{\primal_2}
        \\
        \text{and} 
        \\
        \module{\primal_1}+\np{\sqrt{2}-1}\module{\primal_2} 
        < 1 < \module{\primal_1}+\module{\primal_2}  \eqfinv
      \end{cases} \)
      \label{eq:calL0_d=2_inside_triangle}
      \\ 
      \frac{ 3-\module{\primal_2} }{ 2 } +
      \frac{ \primal_1^2 }{ 2 \np{ 1-\module{\primal_2} } } 
      & 
      \text{if } \( 
      \np{\sqrt{2}-1}\module{\primal_1}+\module{\primal_2} \geq 1 
      \text{ and } \module{\primal_2} > \module{\primal_1}  
      \eqfinv \)
      \label{eq:calL0_d=2_inside_nail_up}
      \\ 
      \frac{ 3-\module{\primal_1} }{ 2 } +
      \frac{ \primal_2^2 }{ 2 \np{ 1-\module{\primal_1} } } 
      & 
      \text{ if } \( 
      \module{\primal_1}+\np{\sqrt{2}-1}\module{\primal_2} \geq 1  
      \text{ and } \module{\primal_1} > \module{\primal_2}  
      \eqfinp \)
      \label{eq:calL0_d=2_inside_nail_down}
    \end{numcases}
    \label{eq:calL0_d=2}
  \end{subequations}
  \label{pr:calL0_d=2}
\end{proposition}

\begin{proof}
  By~\eqref{eq:pseudonormlzero_convex_minimum} for $d=2$, we find that 
  \begin{subequations}
    \begin{equation}
      {\cal L}_0\np{\primal}= 
      \min_{ \np{ \primal^{(1)} , \primal^{(2)} } \in \Convex\np{\primal} }
      \EuclidianSupportNorm{1}{\primal^{(1)}} + 2 \EuclidianSupportNorm{2}{\primal^{(2)}} 
      \eqfinv
      \label{eq:calL0_d=2_minimization}
    \end{equation}
    where the constraints set is given by 
    \begin{equation}
      \Convex\np{\primal} = 
      \Bset{ \bp{ \primal^{(1)} , \primal^{(2)} } \in (\RR^2)^2 }%
      { \EuclidianSupportNorm{1}{\primal^{(1)}} + \EuclidianSupportNorm{2}{\primal^{(2)}} \leq 1 
        \eqsepv \primal^{(1)} + \primal^{(2)} = \primal } 
      \eqfinp
      \label{eq:calL0_d=2_constraints_set}
    \end{equation}
  \end{subequations}
  If \( \np{ \primal^{(1)} , \primal^{(2)} } \in \Convex\np{\primal} \), 
  we have that 
  \begin{subequations}
    \begin{align}
      \norm{\primal} 
      & \leq 
        \EuclidianSupportNorm{1}{\primal^{(1)}} + \EuclidianSupportNorm{2}{\primal^{(2)}} 
        \leq 1
        \eqfinv 
        \label{eq:calL0_d=2_constraints_one}
      \\
      \norm{\primal} 
      &= 
        1 \Rightarrow \EuclidianSupportNorm{2}{\primal^{(1)}} 
        = \EuclidianSupportNorm{1}{\primal^{(1)}} 
        \mtext{ and }
        \EuclidianSupportNorm{1}{\primal^{(1)}} + \EuclidianSupportNorm{2}{\primal^{(2)}} =1 
        \eqfinv
        \label{eq:calL0_d=2_constraints_two}
    \end{align}
  \end{subequations}
  \begin{align*}
    \text{because } \norm{\primal}
    &=
      \EuclidianSupportNorm{2}{\primal} 
      \tag{ by~\eqref{eq:dual_support_norm_inequalities} }
    \\
    &\leq 
      \EuclidianSupportNorm{2}{\primal^{(1)}} + \EuclidianSupportNorm{2}{\primal^{(2)}} 
      \tag{ because \( \np{ \primal^{(1)} , \primal^{(2)} } \in \Convex\np{\primal} 
      \Rightarrow \primal^{(1)} + \primal^{(2)} = \primal \) }
    \\
    &\leq 
      \EuclidianSupportNorm{1}{\primal^{(1)}} + \EuclidianSupportNorm{2}{\primal^{(2)}} 
      \tag{ because \( \EuclidianSupportNorm{2}{\primal^{(1)}} \leq 
      \EuclidianSupportNorm{1}{\primal^{(1)}} \)
      by~\eqref{eq:dual_support_norm_inequalities} }
    \\
    &\leq 
      1 
      \tag{ because \( \np{ \primal^{(1)} , \primal^{(2)} } \in \Convex\np{\primal} 
      \Rightarrow \EuclidianSupportNorm{1}{\primal^{(1)}} + \EuclidianSupportNorm{2}{\primal^{(2)}} \leq 1 \) }
      \eqfinp
  \end{align*}
  We are now going to describe the constraints set~\( \Convex\np{\primal} \)
  in~\eqref{eq:calL0_d=2_constraints_set}
  according to~\( \norm{\primal} \), then to deduce~\( {\cal L}_0\np{\primal} \)
  from~\eqref{eq:calL0_d=2_minimization}. 
  \begin{enumerate}
  \item 
    Suppose that \( \norm{\primal} = \sqrt{ \primal_1^2 + \primal_2^2 } > 1 \).
    Then, by~\eqref{eq:calL0_d=2_constraints_one}, 
    we deduce that \( \Convex\np{\primal} =\emptyset \)
    in~\eqref{eq:calL0_d=2_constraints_set}, hence 
    that \( {\cal L}_0\np{\primal} = +\infty \) by~\eqref{eq:calL0_d=2_minimization}. 
  \item 
    Suppose that \( \norm{\primal} = \sqrt{ \primal_1^2 + \primal_2^2 } = 1 \).
    If \( \np{ \primal^{(1)} , \primal^{(2)} } \in \Convex\np{\primal} \), 
    we obtain by~\eqref{eq:calL0_d=2_constraints_two}
    that 
    \[
      \sqrt{ \module{\primal^{(1)}_1}^2 + \module{\primal^{(1)}_2}^2 } 
      = \EuclidianSupportNorm{2}{\primal^{(1)}} 
      = \EuclidianSupportNorm{1}{\primal^{(1)}} 
      = \module{\primal^{(1)}_1}+\module{\primal^{(1)}_2} 
      \eqfinv
    \]
    from which we deduce that
    \( \module{\primal^{(1)}_1}\times\module{\primal^{(1)}_2}=0 \).
    From \( \primal^{(1)} + \primal^{(2)} = \primal \)
    and \( \EuclidianSupportNorm{1}{\primal^{(1)}} + \EuclidianSupportNorm{2}{\primal^{(2)}} =1 \),
    by~\eqref{eq:calL0_d=2_constraints_one}, 
    we deduce that either \( \primal^{(1)}_1=0 \) and 
    \( \module{\primal^{(1)}_2} + \sqrt{ \primal_1^2 + \np{\primal_2 -\primal^{(1)}_2}^2 }  =1 \),
    or \( \primal^{(1)}_2=0 \) and 
    \( \module{\primal^{(1)}_1} + \sqrt{ \np{\primal_1 -\primal^{(1)}_1}^2 + \primal_2^2 }  =1 \),
    that is, after calculations,
    either \( \primal^{(1)}_1=0 \) and 
    \( \module{\primal^{(1)}_2} =\primal_2 \times \primal^{(1)}_2 \), 
    or \( \primal^{(1)}_2=0 \) and 
    \( \module{\primal^{(1)}_1} =\primal_1 \times \primal^{(1)}_1 \).
    Therefore, we have the following two subcases.
    \begin{enumerate}
    \item 
      If \( \primal \not\in \{ (1,0), (0,1), (-1,0), (0,-1) \} \),
      then necessarily \( \np{ \primal^{(1)} , \primal^{(2)} } = \np{ 0, \primal } \),
      that is, \( \Convex\np{\primal} = \{ \np{ 0, \primal } \} \).
      As a consequence, 
      \( {\cal L}_0\np{\primal}= 
      \EuclidianSupportNorm{1}{0} + 2 \EuclidianSupportNorm{2}{\primal} = 2 \norm{\primal}=2\)
      by~\eqref{eq:calL0_d=2_minimization}. 
    \item 
      If \( \primal \in \{ (1,0), (0,1), (-1,0), (0,-1) \} \),
      it is easy to check that \( \np{ \primal, 0 } \in \Convex\np{\primal} \) 
      by~\eqref{eq:calL0_d=2_constraints_set}.
      Therefore, \( {\cal L}_0\np{\primal} \leq \EuclidianSupportNorm{1}{\primal} +
      2\EuclidianSupportNorm{2}{0} = 1 \) by~\eqref{eq:calL0_d=2_minimization}. 
      Now, for any  
      \( \np{ \primal^{(1)} , \primal^{(2)} } \in \Convex\np{\primal} \), 
      we have that 
      \[
        \EuclidianSupportNorm{1}{\primal^{(1)}} + 2 \EuclidianSupportNorm{2}{\primal^{(2)}} 
        \geq
        \EuclidianSupportNorm{1}{\primal^{(1)}} + \EuclidianSupportNorm{2}{\primal^{(2)}} 
        \geq \norm{\primal}=1 
      \]
      by~\eqref{eq:calL0_d=2_constraints_one}.
      To conclude, we obtain that \( 1 \leq {\cal L}_0\np{\primal}  \) 
      by~\eqref{eq:calL0_d=2_minimization}, hence that 
      \( {\cal L}_0\np{\primal} = 1 \).
    \end{enumerate}
  \item 
    Suppose that \( \norm{\primal} = \sqrt{ \primal_1^2 + \primal_2^2 } < 1 \).
    Then, the proof is an application of
    Proposition~\ref{pr:KKT_optimization_problem_calL0_d=2} in
    Appendix~\ref{Appendix},
    combined with the formula~\eqref{eq:calL0_module}.
  \end{enumerate}
  \smallskip

  This ends the proof.
\end{proof}

\section{Conclusion}

In this paper, we have introduced a novel class of 
one-sided linear couplings, and we have shown that they induce conjugacies 
that share nice properties with the classic Fenchel conjugacy.
Among them, we have distinguished a novel coupling, \ECapra, 
having the property of being constant along primal rays, 
like the \lzeropseudonorm.
For the \ECapra\ conjugacy, induced by the coupling \ECapra, 
we have proved that the \lzeropseudonorm\ is equal to its biconjugate:
hence, the \lzeropseudonorm\ is \ECapra-convex 
in the sense of generalized convexity.
We have also provided expressions for 
the \ECapra\ conjugate and biconjugate 
of the \lzeropseudonorm, 
and of the characteristic functions of its level sets,
in terms of the sequence of so-called top-$k$ norms.
Finally, we have shown that the \lzeropseudonorm\ displays
hidden convexity as we have proved that it coincides, 
on the Euclidian unit sphere, with a proper convex lsc function.
This is somewhat surprising as the \lzeropseudonorm\ is a highly
nonconvex function of combinatorial nature.
\bigskip

\appendix

\section{Appendix}
\label{Appendix}

\subsection{Properties of top-$k$ norms and of $k$-support norms}

\begin{subequations}
  Before studying properties of top-$k$ norms and of $k$-support norms,
  we recall the notion of dual norm. 
  Suppose that $\RR^d$ is equipped with a norm~$\TripleNorm{\cdot}$
  with unit ball denoted by \( \BALL_{\TripleNorm{\cdot}}
  =   \defset{\primal \in \RR^d}{\TripleNorm{\primal} \leq 1} \).
  The expression 
  \begin{equation}
    \TripleNorm{\dual}_\star = 
    \sup_{ \TripleNorm{\primal} \leq 1 } \proscal{\primal}{\dual} 
    \eqsepv \forall \dual \in \RR^d
    \label{eq:dual_norm}
  \end{equation}
  defines a norm on~$\RR^d$, 
  called the \emph{dual norm} \( \TripleNorm{\cdot}_\star \).
  We have 
  \begin{equation}
    \TripleNorm{\cdot}_\star = \sigma_{\BALL_{\TripleNorm{\cdot}}} 
    \mtext{ and } 
    \TripleNorm{\cdot} = \sigma_{\BALL_{\TripleNorm{\cdot}_\star}}
    \eqfinv
    \label{eq:norm_dual_norm}
  \end{equation}
  where \( \BALL_{\TripleNorm{\cdot}_\star} \), the unit ball of the dual norm,
  is the polar set~\( \BALL_{\TripleNorm{\cdot}}^{\odot} \) 
  of the unit ball~\( \BALL_{\TripleNorm{\cdot}} \):
  \begin{equation}
    \BALL_{\TripleNorm{\cdot}_\star}
    = \defset{\dual \in \RR^d}{\TripleNorm{\dual}_\star \leq 1}
    =\BALL_{\TripleNorm{\cdot}}^{\odot} 
    = \defset{\dual \in \RR^d}{\proscal{\primal}{\dual} \leq 1 
      \eqsepv \forall \primal \in \BALL_{\TripleNorm{\cdot}} } 
    \eqfinp
    \label{eq:norm_dual_norm_unit_ball}
  \end{equation}
\end{subequations}

\subsubsection{Properties of top-$k$ norms}

For all \( K \subset\ba{1,\ldots,d} \), 
we introduce \emph{degenerate} unit ``spheres'' and ``balls'' of \( \RR^d \),
equipped with the Euclidian norm~$\norm{\cdot}$, by
\begin{subequations}
  \begin{align}
    \ESPHERE_{K} 
    &= 
      \defset{\primal \in \RR^d}{ \norm{\primal_{K}} = 1} 
      \eqfinv 
      \label{eq:ESPHERE_K}
    \\
    \EBALL_{K} 
    &=
      \defset{\primal \in \RR^d}{ \norm{\primal_{K}} \le 1} 
      \eqfinv
      \label{eq:EBALL_K}
  \end{align}
\end{subequations}
where $\primal_{K}$ has been defined as the orthogonal projection of~\( \primal \) onto
the subspace~\( \FlatRR_{K} \) in~\eqref{eq:FlatRR}.
In what follows, the Euclidian unit sphere~$\SPHERE$ and
ball~$\BALL$ have been defined in~\eqref{eq:Euclidian_SPHERE}, and
the top-$k$ norm~\( \EuclidianTopNorm{k}{\cdot} \)  
has been introduced in Definition~\ref{de:symmetric_gauge_norm}.

\begin{proposition}
  Let \( k \in \ba{1,\ldots,d} \).
  \begin{itemize}
  \item 
    For any \( \primal \in \RR^d \), the following equalities and inequalities hold true
    \begin{equation}
      \sup_{ j=1,\ldots,d } \module{\primal_j} = 
      \norm{\primal}_{\infty} =  \EuclidianTopNorm{1}{\primal} \leq \cdots \leq 
      \EuclidianTopNorm{l}{\primal} \leq 
      \EuclidianTopNorm{l+1}{\primal} \leq \cdots \leq 
      \EuclidianTopNorm{d}{\primal} =  \norm{\primal} 
      \eqfinp
      \label{eq:k_approx_props}
    \end{equation}
  \item 
    We have 
    \begin{equation}
      \FlatRR_{K} \cap \SPHERE  = \ESPHERE_{K} \cap \SPHERE  
      \eqsepv \forall K \subset \ba{1,\ldots,d} 
      \eqfinp
      \label{eq:sphere_and_esphere_K}
    \end{equation}
  \item 
    The top-$k$ norm~\( \EuclidianTopNorm{k}{\cdot} \) 
    satisfies
    \begin{equation}
      \EuclidianTopNorm{k}{\cdot} = 
      \sigma_{ \cup_{ \cardinal{K} \leq k} \np{ \FlatRR_{K} \cap \BALL } }
      = \sup_{ \cardinal{K} \leq k } \sigma_{ \np{ \FlatRR_{K} \cap \BALL } }
      = \sup_{ \cardinal{K} \leq k } \sigma_{ \np{ \FlatRR_{K} \cap \SPHERE }  }
      = \sigma_{ \cup_{ \cardinal{K} \leq k} \np{ \FlatRR_{K} \cap \SPHERE }  } 
      \eqfinp
      \label{eq:k-norm-from-support}
    \end{equation}
  \item 
    The unit sphere~\( \EuclidianTop{k}{\SPHERE} \) 
    and ball~\( \EuclidianTop{k}{\BALL} \)
    of \( \RR^d \) for the top-$k$ norm
    \( \EuclidianTopNorm{k}{\cdot} \) 
    satisfy 
    \begin{subequations}
      \begin{align}
        \EuclidianTop{k}{\BALL} &= 
                                  \defset{\primal \in \RR^d}{ \EuclidianTopNorm{k}{\primal} \leq 1} 
                                  = \bigcap_{\cardinal{K} \leq k} \EBALL_{K}
                                  \eqfinv    
                                  \label{eq:symmetric_gauge_norm_unit-ball} 
        \\
        \EuclidianTop{k}{\SPHERE} 
                                &= 
                                  \defset{\primal \in \RR^d}{ \EuclidianTopNorm{k}{\primal} = 1} 
                                  = \EuclidianTop{k}{\BALL} \cap
                                  \Bp{ \bigcup_{\cardinal{K} \leq k} \ESPHERE_{K} }
                                  \eqfinp
                                  \label{eq:symmetric_gauge_norm_unit-sphere}
      \end{align}
    \end{subequations}
  \item 
    The unit balls~\( \EuclidianTop{l}{\BALL} \) satisfy the inclusions
    \begin{equation}
      \BALL=\EuclidianTop{d}{\BALL}
      \subset \cdots \subset \EuclidianTop{l+1}{\BALL} 
      \subset \EuclidianTop{l}{\BALL} \subset \cdots 
      \subset \EuclidianTop{1}{\BALL} 
      \eqfinp 
      \label{eq:symmetric_gauge_norm_unit-balls_inclusions}
    \end{equation}
  \item
    We have
    \begin{equation}
      \EuclidianTopNorm{k}{\dual} \leq \sqrt{k} 
      \EuclidianTopNorm{1}{\dual} 
      \eqsepv
      \forall \dual \in \RR^d 
      \eqsepv
      \forall k=1, \ldots, d 
      \eqfinp 
      \label{eq:symmetric_gauge_norm_norms_relation}
    \end{equation}
  \end{itemize}
\end{proposition}

\begin{proof} \quad

  \noindent $\bullet$ 
  The Equalities and Inequalities~\eqref{eq:k_approx_props} derive 
  from the very definition~\eqref{eq:symmetric_gauge_norm}
  of the top-$k$ norm~\( \EuclidianTopNorm{k}{\cdot} \).
  \smallskip

  \noindent $\bullet$ 
  We prove Equation~\eqref{eq:sphere_and_esphere_K}.
  We have that \( \primal=\primal_K+\primal_{-K} \),
  for any \( \primal\in\RR^d \),  
  and the decomposition is orthogonal, leading to
  \begin{equation}
    \bp{\forall \primal \in \RR^d } \qquad
    \primal=\primal_K+\primal_{-K} 
    \eqsepv \primal_K \perp \primal_{-K} 
    \mtext{ and }
    \norm{\primal}^2=\norm{\primal_K}^2+\norm{\primal_{-K}}^2
    \eqfinp
    \label{eq:decomposition_orthogonal}
  \end{equation}
  For \( K \subset \ba{1,\ldots,d} \), we have that 
  \begin{align*}
    \primal \in \SPHERE \mtext{ and } 
    \primal \in \ESPHERE_{K}  
    &\iff 
      1=\norm{\primal}^2 \mtext{ and } 1=\norm{\primal_K}^2 
      \tag{ by~\eqref{eq:Euclidian_SPHERE} and \eqref{eq:ESPHERE_K} }
    \\
    &\iff 
      1=\norm{\primal}^2=\norm{\primal_K}^2+\norm{\primal_{-K}}^2
      \mtext{ and } 1=\norm{\primal_K}^2
      \tag{ by~\eqref{eq:decomposition_orthogonal} }
    \\
    &\iff 
      \norm{\primal_{-K}}=0 \mtext{ and } 1=\norm{\primal_K} 
      \tag{ by~\eqref{eq:decomposition_orthogonal} }
    \\
    &\iff 
      \primal \in \FlatRR_{K} \cap \SPHERE
      \tag{ by~\eqref{eq:FlatRR} and \eqref{eq:Euclidian_SPHERE} }
      \eqfinp 
  \end{align*}
  \smallskip

  \noindent $\bullet$ 
  We prove Equation~\eqref{eq:k-norm-from-support}. 
  For this purpose, we first establish that
  \begin{equation}
    \sigma_{ \FlatRR_{K} \cap \BALL }(\dual) 
    =\norm{\dual_K} 
    \eqsepv \forall \dual \in \RR^d 
    \eqfinp   
    \label{eq:support_functions_unit_sphere_ball_K}
  \end{equation}
  Indeed, for \( \dual \in \RR^d \), we have
  \begin{align*}
    \sigma_{ \FlatRR_{K} \cap \BALL }(\dual)
    &=
      \sup_{ \primal \in \FlatRR_{K} \cap \BALL } \proscal{\primal}{\dual}
      \tag{ by definition~\eqref{eq:support_function} of a support function }
    \\
    &=
      \sup_{ \primal \in \FlatRR_{K} \cap \BALL } 
      \proscal{\primal_K+\primal_{-K}}{\dual_K+\dual_{-K}}
      \tag{ by the decomposition~\eqref{eq:decomposition_orthogonal} }
    \\
    &=
      \sup_{ \primal \in \FlatRR_{K} \cap \BALL } 
      \bp{ \proscal{\primal_K}{\dual_K}
      + \proscal{\primal_{-K}}{\dual_{-K}} }
      \tag{ because \( \primal_K \perp \dual_{-K} \) and
      \( \primal_{-K} \perp \dual_K \) by~\eqref{eq:decomposition_orthogonal} }
    \\
    &=
      \sup \ba{  \proscal{\primal_K}{\dual_K}
      + \proscal{\primal_{-K}}{\dual_{-K}}\, \vert \,
      \primal_{-K}=0 \mtext{ and } \norm{\primal_{K}} \leq 1 } 
      \tag{ by definition of \( \FlatRR_{K} \cap \BALL \) }
    \\
    &=
      \sup \ba{  \proscal{\primal_K}{\dual_K} \, \vert \;
      \norm{\primal_{K}} \leq 1 } 
    \\
    &=
      \norm{\dual_K} 
  \end{align*}
  as is well-known for the Euclidian norm~$\norm{\cdot}$, when restricted 
  to the subspace~\( \FlatRR_{K} \) 
  (because it is equal to its dual norm). 
  Then, for all $\dual \in \RR^d$, we have that 
  \begin{align*}
    \sigma_{ \cup_{ \cardinal{K} \leq k } \FlatRR_{K} \cap \BALL }(\dual)
    &= 
      \sup_{\cardinal{K} \leq k} \sigma_{ \FlatRR_{K} \cap \BALL }(\dual) 
      \tag{ as the support function turns a union of sets into a supremum }
    \\
    &= 
      \sup_{\cardinal{K} \leq k} \norm{\dual_K}
      \tag{by~\eqref{eq:support_functions_unit_sphere_ball_K} }
    \\
    & =\EuclidianTopNorm{k}{\dual} 
      \tag{by definition~\eqref{eq:symmetric_gauge_norm}
      of \( \EuclidianTopNorm{k}{\cdot} \) }
      \eqfinp
  \end{align*}
  Now, by~\eqref{eq:Euclidian_SPHERE} and \eqref{eq:FlatRR},
  it is straightforward that 
  \( \closedconvexhull\np{ \FlatRR_{K} \cap \SPHERE } = \FlatRR_{K} \cap \BALL \)
  and we deduce that 
  \[
    \EuclidianTopNorm{k}{\cdot} 
    =
    \sigma_{ \cup_{ \cardinal{K} \leq k } \np{ \FlatRR_{K} \cap \BALL } }
    =
    \sup_{\cardinal{K} \leq k} \sigma_{ \np{ \FlatRR_{K} \cap \BALL } }
    =
    \sup_{\cardinal{K} \leq k} 
    \sigma_{ \closedconvexhull\np{\np{ \FlatRR_{K} \cap \SPHERE } } }
    =
    \sup_{\cardinal{K} \leq k} \sigma_{ \np{ \FlatRR_{K} \cap \SPHERE }  }
    =
    \sigma_{ \cup_{ \cardinal{K} \leq k } \np{ \FlatRR_{K} \cap \SPHERE }  }
    \eqfinv
  \]
  giving Equation~\eqref{eq:k-norm-from-support}. 

  \smallskip

  \noindent $\bullet$ 
  We prove Equation~\eqref{eq:symmetric_gauge_norm_unit-ball}:
  \begin{align*}
    \EuclidianTop{k}{\BALL}  
    &= 
      \defset{\primal \in \RR^d}{ \EuclidianTopNorm{k}{\primal} \leq 1} 
      \tag{ by definition of the ball \( \EuclidianTop{k}{\BALL} \) }
    \\
    &= 
      \defset{\primal \in \RR^d}{ \sup_{ 
      \cardinal{K} \leq k}  
      \norm{\primal_K} \leq 1} 
      \tag{ by definition~\eqref{eq:symmetric_gauge_norm}
      of \( \EuclidianTopNorm{k}{\cdot} \) }
    \\
    &= 
      \bigcap_{ 
      \cardinal{K} \leq k} 
      \defset{\primal \in \RR^d}{ \norm{\primal_K} \leq 1} 
      =
      \bigcap_{ 
      \cardinal{K} \leq k}
      \EBALL_{K}
      \eqfinp
      \tag{ by definition~\eqref{eq:EBALL_K} of \( \EBALL_{K} \) }
  \end{align*}
  \smallskip

  \noindent $\bullet$ 
  We prove Equation~\eqref{eq:symmetric_gauge_norm_unit-sphere}: 
  \begin{align*}
    \EuclidianTop{k}{\SPHERE}  
    &= \defset{\primal \in \RR^d}{ \EuclidianTopNorm{k}{\primal} = 1} 
      \tag{ by definition of the unit sphere $\EuclidianTop{k}{\SPHERE}$ }
    \\
    &= \defset{\primal \in \RR^d}{ \sup_{ 
      \cardinal{K} \leq k}  \norm{\primal_K} = 1} 
      \tag{ by definition~\eqref{eq:symmetric_gauge_norm}
      of \( \EuclidianTopNorm{k}{\cdot} \) }
    \\
    &= \defset{\primal \in \RR^d}%
      { \sup_{ 
      \cardinal{K} \leq k}  \norm{\primal_K} \le 1}\\
    & \hphantom{===} \bigcap 
      \defset{\primal \in \RR^d}%
      { \exists K \subset \ba{1,\ldots,d} \eqsepv \cardinal{K} \leq k 
      \eqsepv \norm{\primal_K} = 1 }  
      \intertext{\quad} 
    &= \EuclidianTop{k}{\BALL} \cap 
      \Bp{\bigcup_{ 
      \cardinal{K} \leq k} 
      \defset{\primal \in \RR^d}{ \norm{\primal_K} = 1}} 
      \tag{ by definition of the ball \( \EuclidianTop{k}{\BALL} \) }
    \\
    &= \EuclidianTop{k}{\BALL} \cap 
      \Bp{
      \bigcup_{ 
      \cardinal{K} \leq k} 
      \ESPHERE_{K}}
      \eqfinp
      \tag{ by definition~\eqref{eq:ESPHERE_K} of \( \ESPHERE_{K} \) }
  \end{align*}
  \smallskip

  \noindent $\bullet$ 
  The inclusions~\eqref{eq:symmetric_gauge_norm_unit-balls_inclusions}
  directly follow from the Equalities and Inequalities~\eqref{eq:k_approx_props}.
  \smallskip

  \noindent $\bullet$
  We prove the Inequality~\eqref{eq:symmetric_gauge_norm_norms_relation}.
  Indeed, by definition~\eqref{eq:symmetric_gauge_norm}
  of \( \EuclidianTopNorm{k}{\cdot} \),
  for a given $\dual \in \RR^d$, there exists 
  \( K \subset \ba{1,\ldots,d} \) with \( \cardinal{K} \leq k \) 
  such that 
  \( \bp{ \EuclidianTopNorm{k}{\dual} }^2 =
  \sum_{k \in K} \module{\dual_k}^2 \leq
  \sum_{k \in K} \bp{ \EuclidianTopNorm{1}{\dual} }^2 
  \leq k \bp{ \EuclidianTopNorm{1}{\dual} }^2 \). 
  \smallskip

  This ends the proof.
\end{proof}

\subsubsection{Properties of $k$-support norms}

The $k$-support norm~\( \EuclidianSupportNorm{k}{\cdot} \)
has been introduced in Definition~\ref{de:symmetric_gauge_norm}
as the dual norm of the top-$k$ norm~\( \EuclidianTopNorm{k}{\cdot} \). 

\begin{proposition}
  Let \( k \in \ba{1,\ldots,d} \).
  \begin{itemize}
  \item 
    The unit balls~\( \EuclidianSupport{l}{\BALL} \) satisfy the inclusions
    \begin{equation}
      \EuclidianSupport{1}{\BALL} 
      \subset \cdots \subset
      \EuclidianSupport{l}{\BALL} \subset \EuclidianSupport{l+1}{\BALL} 
      \subset \cdots \subset \EuclidianSupport{d}{\BALL} = \BALL 
      \eqfinp 
      \label{eq:dual_support_norm_unit-balls_inclusions}
    \end{equation}
  \item 
    For any \( \primal\in\RR^d \), the following equalities and inequalities hold true
    \begin{equation}
      \norm{\primal} = \EuclidianSupportNorm{d}{\primal}
      \leq \cdots \leq 
      \EuclidianSupportNorm{l+1}{\primal}\leq 
      \EuclidianSupportNorm{l}{\primal}
      \leq \cdots \leq 
      \EuclidianSupportNorm{1}{\primal} = \sum_{ j=1 }^{d } \module{\primal_j} 
      \eqfinp
      \label{eq:dual_support_norm_inequalities}
    \end{equation}
  \item 
    The unit ball~\( \EuclidianSupport{k}{\BALL} \) of the 
    $k$-support norm~\( \EuclidianSupportNorm{k}{\cdot} \)
    satisfies
    \begin{equation}
      \EuclidianSupport{k}{\BALL} = 
      \defset{\primal \in \RR^d}{ \EuclidianSupportNorm{k}{\primal} \leq 1} 
      = \closedconvexhull\bp{ \bigcup_{ \cardinal{K} \leq k} \np{ \FlatRR_{K} \cap \BALL } }
      = \closedconvexhull\bp{ \bigcup_{ \cardinal{K} \leq k} \np{ \FlatRR_{K} \cap \SPHERE }  }
      \eqfinp
      \label{eq:dual_support_norm_unit_ball}
    \end{equation}
  \item 
    For \(l=0,1,\ldots,d \), we have 
    \begin{equation}
      \EuclidianTopNorm{l}{\cdot} = \sigma_{ \EuclidianSupport{l}{\BALL} } 
      \eqfinp
      \label{eq:symmetric_gauge_norm_egale_dual_norm}
    \end{equation}
  \end{itemize}
\end{proposition}

\begin{proof} \quad
  
  \noindent $\bullet$ 
  The inclusions~\eqref{eq:dual_support_norm_unit-balls_inclusions}
  directly follow from the
  inclusions~\eqref{eq:symmetric_gauge_norm_unit-balls_inclusions}
  and from~\eqref{eq:norm_dual_norm_unit_ball}
  as \( \EuclidianSupport{k}{\BALL} = \bp{\EuclidianTop{k}{\BALL}}^{\odot} \).
  \smallskip

  \noindent $\bullet$ 
  The Inequalities in~\eqref{eq:dual_support_norm_inequalities} derive 
  from the inclusions~\eqref{eq:dual_support_norm_unit-balls_inclusions}.
  The Equalities in~\eqref{eq:dual_support_norm_inequalities} are well-known.
  \smallskip

  \noindent $\bullet$ 
  We prove Equation~\eqref{eq:dual_support_norm_unit_ball}.
  On the one hand, by the first relation in~\eqref{eq:norm_dual_norm},
  we have that \( \EuclidianTopNorm{k}{\cdot}=\sigma_{ \EuclidianSupport{k}{\BALL} } \).
  On the other hand, 
  by~\eqref{eq:k-norm-from-support}, we have that 
  \( \EuclidianTopNorm{k}{\cdot}
  =
  \sigma_{ \cup_{\cardinal{K} \leq k} \np{ \FlatRR_{K} \cap \BALL } } 
  =
  \sigma_{ \cup_{\cardinal{K} \leq k} \np{ \FlatRR_{K} \cap \SPHERE }  } \).
  Then, as is well-known in convex analysis, we deduce that
  \( 
  \closedconvexhull\bp{ \EuclidianSupport{k}{\BALL} } 
  =
  \closedconvexhull\bp{ \bigcup_{ \cardinal{K} \leq k} \np{ \FlatRR_{K} \cap \BALL } }
  =
  \closedconvexhull\bp{ \bigcup_{ \cardinal{K} \leq k} \np{ \FlatRR_{K} \cap \SPHERE }  }
  \).
  As the unit ball~\( \EuclidianSupport{k}{\BALL} \) is closed and convex,
  we immediately obtain~\eqref{eq:dual_support_norm_unit_ball}.
  \smallskip

  \noindent $\bullet$ 
  We prove Equation~\eqref{eq:symmetric_gauge_norm_egale_dual_norm}.
  By Definition~\ref{de:symmetric_gauge_norm}, the $l$-support norm
  is the dual norm of the top-$l$ norm.
  Therefore, the top-$l$ norm is the dual norm 
  of the $l$-support norm and~\eqref{eq:symmetric_gauge_norm_egale_dual_norm}
  follows from~\eqref{eq:norm_dual_norm} for \(l=1,\ldots,d \).
  For \(l=0\), both conventions \( \EuclidianTopNorm{0}{\cdot} = 0 \)
  and \( \EuclidianSupport{0}{\BALL}=\{0\} \) lead to 
  \( \EuclidianTopNorm{0}{\cdot} = 0 =
  \sigma_{0} = \sigma_{ \EuclidianSupport{l}{\BALL} } \).
  \smallskip

  This ends the proof.
\end{proof}

\subsection{Properties of the level sets of the \lzeropseudonorm}

We establish useful connections between the \lzeropseudonorm\ in~\eqref{eq:pseudo_norm_l0} 
and the top-$k$ norm
\( \EuclidianTopNorm{k}{\cdot} \) in~\eqref{eq:symmetric_gauge_norm}.

\begin{proposition}
  Let \( k \in \ba{0,1,\ldots,d} \).
  We have
  \begin{subequations}
    \begin{align}
      \bp{\forall \primal \in \RR^d } \quad
      & 
        \lzero\np{\primal} = k
        \iff 
        0 
        \leq \cdots \leq 
        \EuclidianTopNorm{k-1}{\primal} < 
        \EuclidianTopNorm{k}{\primal} = 
        \cdots =
        \EuclidianTopNorm{d}{\primal} =  \norm{\primal} 
        \eqfinv
        \label{eq:level_curve_l0_characterization}
      \\
      \bp{\forall \primal \in \RR^d } \quad
      & 
        \primal \in \LevelSet{\lzero}{k} 
        \iff
        \lzero\np{\primal} \leq k
        \iff 
        \EuclidianTopNorm{k}{\primal} = \norm{\primal} 
        \eqfinv
        \label{equiv_norm0}
      \\
      \bp{\forall \primal \in \RR^d } \quad
      & 
        \primal \in \LevelSet{\lzero}{k}\backslash\{0\}
        \iff
        0<\lzero\np{\primal} \leq k \iff
        \primal\neq 0 \mtext{ and }
        \frac{\primal}{\norm{\primal}} \in \SPHERE \cap \EuclidianTop{k}{\SPHERE}  
        \eqfinp
        \label{equiv_norm0_bis}
    \end{align}
  \end{subequations}
  The intersection of the level set \( \LevelSet{\lzero}{k} \) 
  in~\eqref{eq:level_set_pseudonormlzero}
  of the \lzeropseudonorm\ in~\eqref{eq:pseudo_norm_l0} 
  with the Euclidian unit sphere~$\SPHERE$ has the two 
  following  expressions
  \begin{subequations}
    \begin{align}
      \LevelSet{\lzero}{k} \cap \SPHERE 
        &=
          \EuclidianSupport{k}{\BALL} \cap \SPHERE 
          \eqfinv
          \label{eq:level_set_l0_inter_sphere_b}
      \\
      \LevelSet{\lzero}{k} \cap \SPHERE 
        &=
          \overline{ \LevelCurve{\lzero}{k} \cap \SPHERE }
          \eqfinp
          \label{eq:closure_level_curve}
    \end{align}
  \end{subequations}
\end{proposition}

\begin{proof} \quad

  \noindent $\bullet$ 
  The Equivalences~\eqref{eq:level_curve_l0_characterization} and
  \eqref{equiv_norm0} 
  are well-known and easy to prove. 
  \smallskip

  \noindent $\bullet$ 
  We prove the Equivalence~\eqref{equiv_norm0_bis}.
  Indeed, using Equation~\eqref{equiv_norm0} we have that,
  for \( \primal \in \RR^d\backslash\{0\} \):
  \begin{align*}
    \lzero\np{\primal} \leq k 
    &  \iff \EuclidianTopNorm{k}{\primal} = \norm{\primal} 
      \iff \EuclidianTopNorm{k}{\frac{\primal}{\norm{\primal}}} = 1 
      \iff \frac{\primal}{\norm{\primal}} \in \EuclidianTop{k}{\SPHERE} 
      \iff \frac{\primal}{\norm{\primal}} \in 
      \SPHERE \cap \EuclidianTop{k}{\SPHERE}  
      \eqfinp
  \end{align*}
  \smallskip

  \noindent $\bullet$ 
  We prove Equation~\eqref{eq:level_set_l0_inter_sphere_b}.
  First, we observe that the level set \( \LevelSet{\lzero}{k} \) is closed
  because, by~\eqref{equiv_norm0}, it can be expressed as 
  \( \LevelSet{\lzero}{k} = \defset{ \primal\in\RR^d }%
  { \EuclidianTopNorm{k}{\primal} = \norm{\primal} } \).
  This also follows from the well-known property that 
  the pseudonorm~$\lzero$ is lower semi continuous.
  Second, we have
  \begin{align*}
    \LevelSet{\lzero}{k} \cap \SPHERE 
    &= \SPHERE \cap 
      \closedconvexhull\bp{\LevelSet{\lzero}{k} \cap \SPHERE} 
      \tag{by Lemma~\ref{lemma:convex_env} since 
      \( \LevelSet{\lzero}{k} \cap \SPHERE \subset \SPHERE \) 
      and is closed }
    \\
    &= \SPHERE \cap 
      \closedconvexhull\bp{ \bigcup_{ {\cardinal{K} \leq k}} \np{ \FlatRR_{K} \cap \SPHERE }  }
      \tag{ as \( \LevelSet{\lzero}{k} \cap \SPHERE =
      \bigcup_{ {\cardinal{K} \leq k}} \np{ \FlatRR_{K} \cap \SPHERE }  \) 
      by~\eqref{eq:level_set_pseudonormlzero} }
    \\
    &= \EuclidianSupport{k}{\BALL} \cap \SPHERE 
      \tag{ as 
      \( \closedconvexhull\bp{ \bigcup_{ {\cardinal{K} \leq k}} \np{ \FlatRR_{K} \cap \SPHERE } }
      = \EuclidianSupport{k}{\BALL} \) 
      by~\eqref{eq:dual_support_norm_unit_ball} }
      \eqfinp
  \end{align*}
  \smallskip

  \noindent $\bullet$ 
  We prove Equation~\eqref{eq:closure_level_curve}.
  For this purpose, we first establish the (known) fact that 
  \( \overline{ \LevelCurve{\lzero}{k} } = \LevelSet{\lzero}{k} \). 
  The inclusion \( \overline{ \LevelCurve{\lzero}{k} } 
  \subset \LevelSet{\lzero}{k} \) is easy.
  Indeed, as we have seen that \( \LevelSet{\lzero}{k} \) is closed, 
  we have \( \LevelCurve{\lzero}{k} \subset \LevelSet{\lzero}{k} 
  \Rightarrow
  \overline{ \LevelCurve{\lzero}{k} } \subset 
  \overline{ \LevelSet{\lzero}{k} } = \LevelSet{\lzero}{k} \).
  There remains to prove the reverse inclusion
  \( \LevelSet{\lzero}{k} \subset \overline{ \LevelCurve{\lzero}{k} } \).
  For this purpose, we consider
  \( \primal \in \LevelSet{\lzero}{k} \). 
  If \( \primal \in \LevelCurve{\lzero}{k} \), obviously 
  \( \primal \in \overline{ \LevelCurve{\lzero}{k} } \).
  Therefore, we suppose that \( \lzero\np{\primal}=l < k \).
  By definition of \( \lzero\np{\primal} \), there exists 
  \( L \subset \ba{1,\ldots,d} \) such that 
  \( \cardinal{L}=l < k \) and \( \primal = \primal_L \).
  For \( \epsilon > 0 \), define \( \primal^\epsilon \) as
  coinciding with  \( \primal \) except for 
  $k-l$ indices outside~$L$ for which the components are 
  \( \epsilon > 0 \).
  By construction \( \lzero\np{\primal^\epsilon}=k \) and
  \( \primal^\epsilon \to \primal \) when \( \epsilon \to 0 \).
  This proves that 
  \( \LevelSet{\lzero}{k} \subset \overline{ \LevelCurve{\lzero}{k} } \).

  Second, we prove that \( \LevelSet{\lzero}{k} \cap \SPHERE 
  = \overline{ \LevelCurve{\lzero}{k} \cap \SPHERE } \).
  The inclusion 
  \( \overline{ \LevelCurve{\lzero}{k} \cap \SPHERE } 
  \subset \LevelSet{\lzero}{k} \cap \SPHERE \),
  is easy.
  Indeed, 
  \( \overline{ \LevelCurve{\lzero}{k} } = \LevelSet{\lzero}{k} 
  \Rightarrow
  \overline{ \LevelCurve{\lzero}{k} \cap \SPHERE } \subset 
  \overline{ \SPHERE } \cap \overline{ \LevelCurve{\lzero}{k} } = 
  \LevelSet{\lzero}{k} \cap \SPHERE \).
  To prove the reverse inclusion
  \( \LevelSet{\lzero}{k} \cap \SPHERE 
  \subset \overline{ \LevelCurve{\lzero}{k} \cap \SPHERE } \),
  we consider \( \primal \in \LevelSet{\lzero}{k} \cap \SPHERE \).
  As we have just seen that \( \LevelSet{\lzero}{k} = 
  \overline{ \LevelCurve{\lzero}{k} }\), 
  we deduce that \( \primal \in \overline{ \LevelCurve{\lzero}{k} }\).
  Therefore, there exists a sequence
  \( \sequence{z_n}{n\in\NN} \) in \( \LevelCurve{\lzero}{k} \)
  such that \( z_n \to \primal \) when \( n \to +\infty \).
  Since \( \primal \in \SPHERE \), we can always suppose that 
  \( z_n \neq 0 \), for all $n\in\NN$. Therefore \( z_n/\norm{z_n} \) is well
  defined and, when \( n \to +\infty \), 
  we have \( z_n/\norm{z_n} \to \primal/\norm{\primal}=\primal \)
  since \( \primal \in \SPHERE = \defset{\primal \in \RR^d}{\norm{\primal} = 1} \).
  Now, on the one hand, 
  \( z_n/\norm{z_n} \in \LevelCurve{\lzero}{k} \), for all $n\in\NN$,
  and, on the other hand, \( z_n/\norm{z_n} \in \SPHERE \).
  As a consequence \( z_n/\norm{z_n} \in \LevelCurve{\lzero}{k} \cap \SPHERE \),
  and we conclude that \( \primal \in 
  \overline{ \LevelCurve{\lzero}{k} \cap \SPHERE } \). 
  Thus, we have proved that 
  \( \LevelSet{\lzero}{k} \cap \SPHERE 
  \subset \overline{ \LevelCurve{\lzero}{k} \cap \SPHERE } \).
  \smallskip

  This ends the proof. 
\end{proof}

\begin{lemma} 
  If $A$ is a subset of the Euclidian unit sphere~$\SPHERE$ of~$\RR^d$, 
  then $A = \convexhull(A) \cap \SPHERE$. If $A$ is a closed subset of the Euclidian unit sphere~$\SPHERE$ of~$\RR^d$, 
  then $A = \closedconvexhull(A) \cap \SPHERE$.
  \label{lemma:convex_env}
\end{lemma}

\begin{proof} 
  We first prove that $A = \convexhull(A) \cap \SPHERE$ when $A \subset \SPHERE$. 
  Since $A \subset \convexhull(A)$ and $A\subset \SPHERE$, 
  we immediately get that $A \subset\convexhull(A) \cap \SPHERE$.
  To prove the reverse inclusion, we first start by proving that 
  $\convexhull(A)\cap \SPHERE \subset \mathrm{extr}(\convexhull(A))$,
  the set of extreme points of~$\convexhull(A)$.
  
  The proof is by contradiction.
  Suppose indeed that there exists $x\in \convexhull(A)\cap \SPHERE$ 
  and $x\not\in \mathrm{extr}(\convexhull(A))$. Then, we could find 
  $y \in \convexhull(A)$ and $z \in \convexhull(A)$, distinct from $x$,
  and such that $x = \lambda y + (1-\lambda) z$ for some $\lambda\in (0,1)$.
  Notice that necessarily \( y \neq z \) (because, else, we would have
  $x=y=z$ which would contradict $y \neq x$ and $z \neq x$). 
  By assumption $A\subset \SPHERE$, 
  we deduce that $\convexhull(A) \subset \BALL= \defset{\primal \in
    \RR^d}{\norm{\primal} \leq 1}$, the unit ball, 
  and therefore that \( \norm{y} \leq 1 \) and \( \norm{z} \leq 1 \).
  If $y$ or $z$ were not in $\SPHERE$ --- that is, if either
  \( \norm{y} < 1 \) or \( \norm{z} < 1 \) --- then we would obtain that
  \( \norm{x} \leq \lambda \norm{y} + (1-\lambda) \norm{z} < 1 \) 
  since 
  $\lambda\in (0,1)$;
  we would thus arrive at a contradiction since $x$ could not be in $\SPHERE$. 
  Thus, both $y$ and $z$ must be in $\SPHERE$, 
  and we have a contradiction since no $x \in \SPHERE$, 
  the Euclidian unit sphere, can be obtained as a convex combination of 
  $y \in \SPHERE$ and $z \in \SPHERE$, with \( y \neq z \).

  Hence, we have proved by contradiction that 
  $\convexhull(A)\cap \SPHERE \subset \mathrm{extr}(\convexhull(A))$.
  We can conclude using the fact that 
  $\mathrm{extr}(\convexhull(A)) \subset A$ 
  (see \cite[Exercice 6.4]{hiriart1998optimisation}).
  \smallskip

  Now, we consider the case where the subset $A$ of the Euclidian unit sphere~$\SPHERE$
  is closed.  Using the first part of the proof we have that
  $A= \convexhull(A) \cap \SPHERE$.  Now, $A$ is closed by assumption and bounded
  since $A\subset \SPHERE$. Thus, $A$ is compact and, in a finite
  dimensional space, we have that $\convexhull(A)$ is compact~\cite[Th.~17.2]{Rockafellar:1970}, thus closed. We conclude that 
  $A= \convexhull(A) \cap \SPHERE = \overline{\convexhull(A)} \cap \SPHERE = 
  \closedconvexhull(A) \cap \SPHERE$, where the last equality comes from~\cite[Prop.~3.46]{Bauschke-Combettes:2017}.
  \smallskip

  This ends the proof. 
\end{proof}

\subsection{Additional results on the function~${\cal L}_0$}

In Proposition~\ref{pr:calL0}, we have provided an expression,
for the proper convex lsc function~${\cal L}_0$ 
in Theorem~\ref{th:pseudonormlzero_convex},
as the value of the minimization problem~\eqref{eq:pseudonormlzero_convex_minimum}.
Here, we provide a characterization of the optimal solutions of~\eqref{eq:pseudonormlzero_convex_minimum}.

We recall that the \emph{exposed face} of the 
closed convex set~$\Convex \subset \RR^d$ at~$\dual \in \RR^d $ is  
\cite[p.220]{Hiriart-Urruty-Lemarechal-I:1993}
\begin{equation}
  \FACE_{\Convex}(\dual) = \bset{\primal \in \Convex}%
  { \proscal{\primal}{\dual} = \sigma_{\Convex}(\dual) } 
  = \argmax_{\primal \in \Convex} \proscal{\primal}{\dual} 
  \eqfinp
  \label{eq:FACE}
\end{equation}
In the sequel, we will use the following relations regarding 
faces of unit balls:
\begin{subequations}
  \begin{align}
    \FACE_{ \EuclidianTop{l}{\BALL} }\np{ 0 } 
    &=
      \EuclidianTop{l}{\BALL} 
      \eqsepv \forall l=1,\ldots,d
      \eqfinv
      \label{eq:FACE_l_=BALL}
    \\
    \FACE_{ \EuclidianTop{l}{\BALL} }\np{ \bar\primal^{(l)} } 
    & \subset 
      \EuclidianTop{l}{\SPHERE}
      \eqsepv \mtext{ if } \bar\primal^{(l)} \not = 0 
      \eqfinv
      \label{eq:FACE_l_subset_SPHERE}
    \\
    \FACE_{ \EuclidianTop{d}{\BALL} }\np{ \bar\primal^{(d)} } 
    &= 
      \Ba{ \frac{ \bar\primal^{(d)} }{ \norm{ \bar\primal^{(d)} } } }
      \eqsepv \mtext{ if } \bar\primal^{(d)} \not = 0 
      \eqfinp
      \label{eq:FACE_d_=pointing_vector}
  \end{align}
\end{subequations}


\begin{proposition}
  Let \( \primal \in \RR^d \) be such that \( \norm{\primal}<1 \). 
  The sequence \( \np{ \bar\primal^{(1)}, \ldots, \bar\primal^{(d)} } \)
  of vectors of~$\RR^d$ is solution 
  of the minimization problem~\eqref{eq:pseudonormlzero_convex_minimum}
  or, equivalently, of the minimization problem
  \begin{equation}
    \min_{ \substack{%
        \primal^{(1)} \in \RR^d, \ldots, \primal^{(d)} \in \RR^d 
        \\
        \sum_{ l=1 }^{ d } \sigma_{ \EuclidianTop{l}{\BALL} }\np{\primal^{(l)}} \leq 1 
        \\
        \sum_{ l=1 }^{ d } \primal^{(l)} = \primal
      } }
    \sum_{ l=1 }^{ d } l \sigma_{ \EuclidianTop{l}{\BALL} }\np{\primal^{(l)}}
    \label{eq:optimization_problem_calL0}
  \end{equation}
  if and only if 
  \begin{enumerate}
  \item 
    \label{it:KKT_optimization_problem_calL0_lambda=0}
    either
    \( \sum_{ l=1 }^{ d } \module{ \primal_l } =
    \EuclidianSupportNorm{1}{\primal} \leq~1 \)
    and 
    \( \np{ \bar\primal^{(1)}, \ldots, \bar\primal^{(d)} } =
    \np{ \primal, 0, \ldots, 0} \)
    (and then, the minimum in~\eqref{eq:optimization_problem_calL0}
    or~\eqref{eq:pseudonormlzero_convex_minimum} is equal to~\(
    \EuclidianSupportNorm{1}{\primal} \)), 
  \item 
    \label{it:KKT_optimization_problem_calL0_lambda_not=0}
    or there exists \( \lambda > 0 \) such that
    \begin{subequations}
      \begin{align}
        \bigcap_{ l=1 }^{ d } \np{ l + \lambda }
        \FACE_{ \EuclidianTop{l}{\BALL} }\np{ \bar\primal^{(l)} } 
        \not = \emptyset 
        \eqfinv
        \label{eq:KKT_optimization_problem_calL0_lambda_not=0_a}
        \\
        \sum_{ l=1 }^{ d }\sigma_{ \EuclidianTop{l}{\BALL} }\np{\primal^{(l)}} = 1 
        \eqfinv
        \label{eq:KKT_optimization_problem_calL0_lambda_not=0_b}
        \\
        \sum_{ l=1 }^{ d }  \bar\primal^{(l)} = \primal
        \eqfinp
      \end{align}
      \label{eq:KKT_optimization_problem_calL0_lambda_not=0}
    \end{subequations}
  \end{enumerate}
  \label{pr:KKT_optimization_problem_calL0}
\end{proposition}

\begin{proof}
  The minimization problems~\eqref{eq:pseudonormlzero_convex_minimum}
  and \eqref{eq:optimization_problem_calL0}
  are the same because 
  \( \sigma_{ \EuclidianTop{l}{\BALL} }\np{\cdot} =
  \EuclidianSupportNorm{k}{\cdot} \) since the $k$-support norm
  is the dual norm, as in~\eqref{eq:dual_norm}, 
  of the top-$k$ norm
  (see Definition~\ref{de:symmetric_gauge_norm}).
  First, we establish necessary and sufficient Karush-Kuhn-Tucker (KKT) 
  conditions for the optimization
  problem~\eqref{eq:optimization_problem_calL0}.
  \begin{subequations}
    The optimization problem~\eqref{eq:optimization_problem_calL0}
    is the minimization of the proper convex lsc function
    \begin{equation}
      \fonctionprimal_0\np{ \primal^{(1)}, \ldots, \primal^{(d)} } = 
      \sum_{ l=1 }^{ d } l \sigma_{ \EuclidianTop{l}{\BALL} }\np{\primal^{(l)}}
    \end{equation}
    over a convex domain of \( \bp{ \RR^d }^d \) defined by 
    one scalar inequality constraint, 
    \( \fonctionprimal_1\np{ \primal^{(1)}, \ldots, \primal^{(d)} } \leq 0 \),
    represented by the proper convex lsc function
    \begin{equation}
      \fonctionprimal_1\np{ \primal^{(1)}, \ldots, \primal^{(d)} } = 
      \sum_{ l=1 }^{ d } \sigma_{ \EuclidianTop{l}{\BALL} }\np{\primal^{(l)}} -1
      \eqfinv
    \end{equation}
    and $d$ equality constraints,
    \( \fonctionprimal_{1+k}\np{ \primal^{(1)}, \ldots, \primal^{(d)} } = 0 \)
    for $k=1,\ldots,d$, 
    represented by the  $d$ affine functions
    \begin{equation}
      \fonctionprimal_{1+k}\np{ \primal^{(1)}, \ldots, \primal^{(d)} } = 
      \proscal{\sum_{ l=1 }^{ d } \bar\primal^{(l)} - \primal}{e_k}   
      \eqsepv k=1,\ldots,d 
      \eqfinv
    \end{equation}
    where $e_k$ is the $k$-canonical vector of~$\RR^d$.
    \label{eq:optimization_problem_calL0_functions}
  \end{subequations}
  It should be noted that all the functions 
  $\fonctionprimal_{0}$, $\fonctionprimal_{1}$, $\fonctionprimal_{2}$, \ldots, $\fonctionprimal_{1+d}$
  are proper and have \( \bp{ \RR^d }^d \) for effective domain. 

  As \( \EuclidianSupportNorm{d}{ \primal } = \norm{\primal}<1 \), 
  the sequence \( \bp{ \bar\primal^{(1)}, \ldots, \bar\primal^{(d)} } =
  \bp{ 0,0,\ldots,0,\primal } \) strictly satisfies the inequality constraint,
  that is, 
  \( \fonctionprimal_1\np{ 0,0,\ldots,0,\primal }= \norm{\primal}-1<0 \)
  and satisfies also the equality constraints
  \( \fonctionprimal_2\np{ 0,0,\ldots,0,\primal }= 
  \cdots = \fonctionprimal_{1+d}\np{ 0,0,\ldots,0,\primal }=0 \).
  By the Slater condition, the constraints are qualified.
  Therefore, the sequence \( \bp{ \bar\primal^{(1)}, \ldots, \bar\primal^{(d)} } \)
  is solution of the convex 
  optimization problem~\eqref{eq:optimization_problem_calL0}
  if and only if it satisfies the KKT conditions 
  (\cite[Corollary 28.3.1]{Rockafellar:1970},
  \cite[Example~$1^{'''}$, p.~64]{Rockafellar:1974},
  \cite[Chapter VII]{Hiriart-Urruty-Lemarechal-I:1993}),
  that is, 
  there exists \( \lambda \geq 0 \) and 
  \(\mu=\np{ \mu_1, \ldots, \mu_{d} } \in \RR^d \) 
  such that
  \begin{subequations}
    \begin{align}
      0 \in \partial\fonctionprimal_0
      \np{ \bar\primal^{(1)}, \ldots, \bar\primal^{(d)} } 
      + \lambda \partial\fonctionprimal_1
      \np{ \bar\primal^{(1)}, \ldots, \bar\primal^{(d)} } 
      + \sum_{ k=1 }^{ d } \mu_{k} 
      \partial\fonctionprimal_{1+k} 
      \np{ \bar\primal^{(1)}, \ldots, \bar\primal^{(d)} }
      \eqfinv
      \label{eq:KKT_one}
      \\
      \lambda \fonctionprimal_1
      \np{ \bar\primal^{(1)}, \ldots, \bar\primal^{(d)} } =0
      \eqfinv
      \\
      \fonctionprimal_1
      \np{ \bar\primal^{(1)}, \ldots, \bar\primal^{(d)} } \leq 0
      \eqfinv
      \\
      \forall k=1,\ldots,d \eqsepv 
      \fonctionprimal_{1+k} 
      \np{ \bar\primal^{(1)}, \ldots, \bar\primal^{(d)} } = 0
      \eqfinp  
    \end{align}
  \end{subequations}
  Since \( \partial\sigma_{ \EuclidianTop{l}{\BALL} }\np{\primal^{(l)}} 
  =\FACE_{ \EuclidianTop{l}{\BALL} }\np{ \bar\primal^{(l)} } \)
  \cite[Corollary 8.25]{Rockafellar-Wets:1998},
  for $l=1,\ldots,d$, 
  we have, by~\eqref{eq:optimization_problem_calL0_functions},
  \begin{subequations}
    \begin{align}
      \partial\fonctionprimal_0
      \np{ \bar\primal^{(1)}, \ldots, \bar\primal^{(d)} } 
      &=
        \bp{ \FACE_{ \EuclidianTop{1}{\BALL} }\np{ \bar\primal^{(1)} },
        \ldots, 
        l\FACE_{ \EuclidianTop{l}{\BALL} }\np{ \bar\primal^{(l)} }
        \ldots, 
        d\FACE_{ \EuclidianTop{d}{\BALL} }\np{ \bar\primal^{(d)} } }
        \eqfinv
      \\
      \partial\fonctionprimal_1
      \np{ \bar\primal^{(1)}, \ldots, \bar\primal^{(d)} } 
      &=
        \bp{ \FACE_{ \EuclidianTop{1}{\BALL} }\np{ \bar\primal^{(1)} },
        \ldots, 
        \FACE_{ \EuclidianTop{l}{\BALL} }\np{ \bar\primal^{(l)} }
        \ldots, 
        \FACE_{ \EuclidianTop{d}{\BALL} }\np{ \bar\primal^{(d)} } }
        \eqfinv
      \\
      \partial\fonctionprimal_{1+k} 
      \np{ \bar\primal^{(1)}, \ldots, \bar\primal^{(d)} }
      &= \np{ e_k, \ldots, e_k } 
        \eqsepv \forall k=1,\ldots,d 
        \eqfinp
    \end{align}
  \end{subequations}
  With these expressions, Equation~\eqref{eq:KKT_one} is equivalent to 
  \( \mu=\np{ \mu_1, \ldots, \mu_{d} } 
  = \sum_{ k=1 }^{ d } \mu_{k} e_k \in 
  l\FACE_{ \EuclidianTop{l}{\BALL} }\np{ \bar\primal^{(l)} } +
  \lambda \FACE_{ \EuclidianTop{l}{\BALL} }\np{ \bar\primal^{(l)} } \),
  for $l=1,\ldots,d$.

  We conclude that 
  the sequence \( \bp{ \bar\primal^{(1)}, \ldots, \bar\primal^{(d)} } \)
  of vectors of~$\RR^d$ is solution of the optimization problem~\eqref{eq:optimization_problem_calL0}
  if and only if there exists \( \lambda \geq 0 \) such that
  the following conditions are satisfied
  \begin{subequations}
    \begin{align}
      \bigcap_{ l=1 }^{ d } \Bc{ 
      l\FACE_{ \EuclidianTop{l}{\BALL} }\np{ \bar\primal^{(l)} } +
      \lambda \FACE_{ \EuclidianTop{l}{\BALL} }\np{ \bar\primal^{(l)} } }
      \not = \emptyset 
      \eqfinv
      \label{eq:KKT_optimization_problem_calL0_a}
      \\
      \lambda \Bp{ \sum_{ l=1 }^{ d } 
      \sigma_{ \EuclidianTop{l}{\BALL} }\np{\primal^{(l)}} -1 }=0
      \eqfinv
      \\
      \sum_{ l=1 }^{ d }\sigma_{ \EuclidianTop{l}{\BALL} }\np{\primal^{(l)}}  \leq 1 
      \eqfinv
      \\
      \sum_{ l=1 }^{ d }  \bar\primal^{(l)} = \primal
      \eqfinp
    \end{align}
    \label{eq:KKT_optimization_problem_calL0}
  \end{subequations}
  \smallskip

  Second, we turn to prove Item~\ref{it:KKT_optimization_problem_calL0_lambda=0}
  and Item~\ref{it:KKT_optimization_problem_calL0_lambda_not=0}. 

  \begin{enumerate}
  \item 
    If \( \lambda= 0 \) in~\eqref{eq:KKT_optimization_problem_calL0}, 
    we obtain 
    \begin{subequations}
      \begin{align}
        \bigcap_{ l=1 }^{ d } l\FACE_{ \EuclidianTop{l}{\BALL} }\np{ \bar\primal^{(l)} } 
        \not = \emptyset 
        \eqfinv
        \label{eq:KKT_optimization_problem_calL0_lambda=0_a}
        \\
        \sum_{ l=1 }^{ d }\sigma_{ \EuclidianTop{l}{\BALL} }\np{\primal^{(l)}}  \leq 1 
        \eqfinv
        \label{eq:KKT_optimization_problem_calL0_lambda=0_b}
        \\
        \sum_{ l=1 }^{ d }  \bar\primal^{(l)} = \primal
        \eqfinp
        \label{eq:KKT_optimization_problem_calL0_lambda=0_c}
      \end{align}
      \label{eq:KKT_optimization_problem_calL0_lambda=0}
    \end{subequations}
    We now show that~\eqref{eq:KKT_optimization_problem_calL0_lambda=0}
    holds true if and only if 
    \( \EuclidianSupportNorm{1}{\primal} \leq~1 \)
    and 
    \( \np{ \bar\primal^{(1)}, \bar\primal^{(2)}, \ldots, \bar\primal^{(d)} } =
    \np{ \primal, 0, \ldots, 0} \).
    \smallskip

    On the one hand, let \( \np{ \bar\primal^{(1)}, \ldots, \bar\primal^{(d)} } \)
    be a sequence of vectors of~$\RR^d$ which
    satisfies~\eqref{eq:KKT_optimization_problem_calL0_lambda=0}.
    If \( \np{ \bar\primal^{(1)}, \ldots, \bar\primal^{(d)} } =
    \np{0, \ldots, 0} \), then \( \primal = 0 \) and we indeed conclude that
    \( \EuclidianSupportNorm{1}{\primal} = \EuclidianSupportNorm{1}{0}=0 \leq~1 \)
    and 
    \( \np{ \bar\primal^{(1)}, \bar\primal^{(2)}, \ldots, \bar\primal^{(d)} } 
    = \np{ 0, 0, \ldots, 0} = \np{ \primal, 0, \ldots, 0} \).

    If \( \np{ \bar\primal^{(1)}, \ldots, \bar\primal^{(d)} } \not =
    \np{0, \ldots, 0} \), then 
    \( k= \min \defset{ l\in \{1, \ldots, d\} }%
    { \bar\primal^{(l)} \not = 0 } \) is well defined.
    By~\eqref{eq:KKT_optimization_problem_calL0_lambda=0_a},
    there exists \( \dual \in \bigcap_{ l=1 }^{ d } 
    l\FACE_{ \EuclidianTop{l}{\BALL} }\np{ \bar\primal^{(l)} } \),
    where \( \FACE_{ \EuclidianTop{l}{\BALL} }\np{ \bar\primal^{(l)} } 
    \subset \EuclidianTop{l}{\BALL} \)
    for any $l=1, \ldots, d$, by definition~\eqref{eq:FACE} of the face.
    Now, by the inclusion~\eqref{eq:FACE_l_subset_SPHERE},
    we have that 
    \( \FACE_{ \EuclidianTop{k}{\BALL} }\np{ \bar\primal^{(k)} } 
    \subset \EuclidianTop{k}{\SPHERE} \) 
    since \( \bar\primal^{(k)} \not = 0 \) by definition of~$k$. 
    Therefore, there exists \( \dual \in 
    \EuclidianTop{1}{\BALL} \cap k\EuclidianTop{k}{\SPHERE} \), that is,
    \( \EuclidianTopNorm{1}{\dual} =
    \max_{i=1,\ldots,d} \module{\dual_i} \leq 1 \) and
    \( \EuclidianTopNorm{k}{\dual}=k \). 
    Hence, 
    it easily follows from 
    definition~\eqref{eq:symmetric_gauge_norm}
    of \( \EuclidianTopNorm{k}{\cdot} \) that 
    (see also~\eqref{eq:symmetric_gauge_norm_norms_relation}) 
    \( k^2 = \bp{ \EuclidianTopNorm{k}{\dual} }^2 
    \leq k \bp{ \EuclidianTopNorm{1}{\dual} }^2 \leq k \).
    This gives $k=1$, hence \( \bar\primal^{(1)} \not = 0 \) 
    and \( \bar\primal^{(l)} = 0 \) for all $l=2,\ldots,d$
    by definition of~$k$.
    We conclude that necessarily 
    \( \np{ \bar\primal^{(1)}, \ldots, \bar\primal^{(d)} } =
    \np{ \primal, 0, \ldots, 0} \)
    by~\eqref{eq:KKT_optimization_problem_calL0_lambda=0_c}
    and \( \sigma_{ \EuclidianTop{1}{\BALL} }\np{\primal}=
    \EuclidianSupportNorm{1}{\primal} \leq 1 \)
    by~\eqref{eq:KKT_optimization_problem_calL0_lambda=0_b}.
    \smallskip

    On the other hand, suppose that \( \EuclidianSupportNorm{1}{\primal} \leq 1 \)
    and put \( \np{ \bar\primal^{(1)}, \ldots, \bar\primal^{(d)} } =
    \np{ \primal, 0, \ldots, 0} \).
    Then, Equations~\eqref{eq:KKT_optimization_problem_calL0_lambda=0_b}
    and~\eqref{eq:KKT_optimization_problem_calL0_lambda=0_c}
    are satisfied.
    So is~\eqref{eq:KKT_optimization_problem_calL0_lambda=0_a}
    because 
    \begin{align*}
      \bigcap_{ l=1 }^{ d } 
      l\FACE_{ \EuclidianTop{l}{\BALL} }\np{ \bar\primal^{(l)} } 
      &=
        \FACE_{ \EuclidianTop{1}{\BALL} }\np{ \primal }
        \cap \bgc{ \bigcap_{ l=2 }^{ d } 
        l\EuclidianTop{l}{\BALL} } 
        \tag{ by~\eqref{eq:FACE_l_=BALL} }
      \\
      &=
        \FACE_{ \EuclidianTop{1}{\BALL} }\np{ \primal }
        \not= \emptyset\eqfinv
    \end{align*}
    {because \( \FACE_{ \EuclidianTop{1}{\BALL} }\np{ \primal }
      \subset \EuclidianTop{1}{\BALL} 
      \subset \bigcap_{ l=2 }^{ d } \sqrt{l}\EuclidianTop{l}{\BALL} 
      \subset \bigcap_{ l=2 }^{ d } l\EuclidianTop{l}{\BALL} \)
      by the Inequality~\eqref{eq:symmetric_gauge_norm_norms_relation}.}
  \item 
    If \( \lambda> 0 \) in~\eqref{eq:KKT_optimization_problem_calL0}, 
    we obtain item~\ref{it:KKT_optimization_problem_calL0_lambda_not=0}.
    Indeed, \eqref{eq:KKT_optimization_problem_calL0_a}
    is equivalent to~\eqref{eq:KKT_optimization_problem_calL0_lambda_not=0_a}
    because \( l\FACE_{ \EuclidianTop{l}{\BALL} }\np{ \bar\primal^{(l)} } +
    \lambda \FACE_{ \EuclidianTop{l}{\BALL} }\np{ \bar\primal^{(l)} }
    = (l+\lambda) \FACE_{ \EuclidianTop{l}{\BALL} }\np{ \bar\primal^{(l)} } \)
    since the face~\( \FACE_{ \EuclidianTop{l}{\BALL} }\np{ \bar\primal^{(l)} } \)
    is convex, and $l>0$, $\lambda>0$.
  \end{enumerate}
  \smallskip

  This ends the proof. 
\end{proof}

Now, we specialize in the two-dimensional case~$d=2$.
Because the function~${\cal L}_0$ in~\eqref{eq:definition_calL0}
satisfies~\eqref{eq:calL0_module}, we restrict the following Proposition to 
\( \primal=\np{\primal_1,\primal_2} \in \RR_+^2 \).

\begin{proposition}
  Let \( \primal=\np{\primal_1,\primal_2} \in \RR_+^2 \) 
  be such that \( \primal_1^2+\primal_2^2<1 \). 
  The sequence \( \bp{ \bar\primal^{(1)}, \bar\primal^{(2)} } \)
  of vectors of~$\RR^2$ is solution of the optimization problem
  \begin{equation}
    {\cal L}_0\np{\primal}= 
    \min_{ \substack{%
        \primal^{(1)} \in \RR^2, \primal^{(2)} \in \RR^2
        \\
        \EuclidianSupportNorm{1}{\primal^{(1)}} + \EuclidianSupportNorm{2}{\primal^{(2)}}  \leq 1 
        \\
        \primal^{(1)} + \primal^{(2)} = \primal
      } }
    \EuclidianSupportNorm{1}{\primal^{(1)}} + 2 \EuclidianSupportNorm{2}{\primal^{(2)}}
    \label{eq:optimization_problem_calL0_d=2}
  \end{equation}
  if and only if one of the following statements holds true:
  \begin{subequations}
    \begin{enumerate}
    \item 
      \label{it:KKT_optimization_problem_calL0_d=2_one}
      \( \primal_1+\primal_2 
      \leq 1 \), 
      and then 
      \( \bp{ \bar\primal^{(1)}, \bar\primal^{(2)} } =\np{\primal,0} \),
      and 
      \begin{equation}
        {\cal L}_0\bp{ \np{\primal_1,\primal_2} } = \primal_1+\primal_2 
        \eqfinv 
      \end{equation}
    \item 
      \label{it:KKT_optimization_problem_calL0_d=2_two}
      \( \primal_1 > 0 \),
      \( \primal_1 + \np{\sqrt{2}-1}\primal_2 \geq 1 \),
      \( \primal_1 > \primal_2 \),
      and then 
      \begin{align*}
        \bar\primal^{(1)} 
        = 
        \Bp{ \frac{ 1-\np{ \primal_1^2+ \primal_2^2 } }{ 2(1-\primal_1) } , 0 }
        \eqfinv
        \quad 
        \bar\primal^{(2)}
        = 
        \Bp{ \frac{ 2\primal_1 -\primal_1^2+ \primal_2^2 -1 }%
        { 2(1-\primal_1) } , \primal_2 } 
        \eqfinv
      \end{align*}
      \begin{equation}
        {\cal L}_0\bp{ \np{\primal_1,\primal_2} } = 
        \frac{3}{2} - \frac{\primal_1}{2} + \frac{\primal_2^2}{ 2(1-\primal_1) } 
        \eqfinv 
      \end{equation}
    \item 
      \label{it:KKT_optimization_problem_calL0_d=2_two_bis}
      \( \primal_2 > 0 \),
      \( \primal_2 + \np{\sqrt{2}-1}\primal_1 \geq 1 \),
      \( \primal_2 > \primal_1 \),
      and then 
      \begin{align*}
        \bar\primal^{(1)} 
        = 
        \Bp{ 0, \frac{ 1-\np{ \primal_1^2+ \primal_2^2 } }{ 2(1-\primal_2) } }
        \eqfinv
        \quad 
        \bar\primal^{(2)}
        = 
        \Bp{ \primal_1, \frac{ 2\primal_2 -\primal_2^2+ \primal_1^2 -1 }%
        { 2(1-\primal_2) } }
        \eqfinv
      \end{align*}
      \begin{equation}
        {\cal L}_0\bp{ \np{\primal_1,\primal_2} } = 
        \frac{3}{2} - \frac{\primal_2}{2} + \frac{\primal_1^2}{ 2(1-\primal_2) } 
        \eqfinv 
      \end{equation}
    \item 
      \label{it:KKT_optimization_problem_calL0_d=2_three}
      \( \primal_1 + \primal_2 > 1 \), 
      \( \np{\sqrt{2}-1}{\primal_1}+{\primal_2} < 1 \),
      \(  {\primal_1}+\np{\sqrt{2}-1}{\primal_2} < 1 \),
      and then 
      \begin{align*}
        \bar\primal^{(1)} 
        &= 
          \Bp{ \frac{ 1- \np{\sqrt{2}-1}{\primal_1} - {\primal_2} }{ 2\np{\sqrt{2}-1} } , 
          \frac{ 1- {\primal_1} - \np{\sqrt{2}-1}{\primal_2} }{ 2\np{\sqrt{2}-1} } }
          \eqfinv
        \\
        \bar\primal^{(2)}
        &= 
          \Bp{ \frac{ {\primal_1} + {\primal_2} -1 }{ 2\np{\sqrt{2}-1} } , 
          \frac{ {\primal_1} + {\primal_2} -1 }{ 2\np{\sqrt{2}-1} } }
          \eqfinv
      \end{align*}
      \begin{equation}
        {\cal L}_0\bp{ \np{\primal_1,\primal_2} } = 
        \frac{\primal_1 + \primal_2 -2+\sqrt{2} }{\sqrt{2}-1} 
        \eqfinp 
      \end{equation}
    \end{enumerate}
  \end{subequations}
  \label{pr:KKT_optimization_problem_calL0_d=2}
\end{proposition}

\begin{proof}
  By Proposition~\ref{pr:KKT_optimization_problem_calL0},
  the sequence \( \bp{ \bar\primal^{(1)}, \bar\primal^{(2)} } \)
  of vectors of~$\RR^2$ is solution of the optimization 
  problem~\eqref{eq:optimization_problem_calL0_d=2}
  if and only if 

  \noindent $-$ 
  either 
  \( \primal_1+\primal_2 = \EuclidianSupportNorm{1}{\primal} \leq 1 \)
  and 
  \( \bp{ \bar\primal^{(1)}, \bar\primal^{(2)} } =\np{\primal,0} \),
  which is equivalent to Item~\ref{it:KKT_optimization_problem_calL0_d=2_one}, 

  \noindent $-$ 
  or there exists \( \lambda > 0 \) such that
  \begin{subequations}
    \begin{align}
      (1+\lambda) \FACE_{ \EuclidianTop{1}{\BALL} }\np{ \bar\primal^{(1)} }
      \cap
      (2+\lambda) \FACE_{ \EuclidianTop{2}{\BALL} }\np{ \bar\primal^{(2)} } 
      \not = \emptyset 
      \eqfinv
      \label{eq:KKT_optimization_problem_calL0_lambda_not=0_d=2_a}
      \\
      \EuclidianSupportNorm{1}{\bar\primal^{(1)}} + \EuclidianSupportNorm{2}{\bar\primal^{(2)}} = 1 
      \eqfinv
      \label{eq:KKT_optimization_problem_calL0_lambda_not=0_d=2_b}
      \\
      \bar\primal^{(1)} + \bar\primal^{(2)} = \primal
      \eqfinp
      \label{eq:KKT_optimization_problem_calL0_lambda_not=0_d=2_c}
    \end{align}
    \label{eq:KKT_optimization_problem_calL0_lambda_not=0_d=2}
  \end{subequations}
  We are going to prove, in several steps, that 
  \( \bp{ \bar\primal^{(1)}, \bar\primal^{(2)} } \)
  satisfies~\eqref{eq:KKT_optimization_problem_calL0_lambda_not=0_d=2}
  for a certain \( \lambda >0 \)
  if and only if it satisfies
  Item~\ref{it:KKT_optimization_problem_calL0_d=2_two},
  Item~\ref{it:KKT_optimization_problem_calL0_d=2_two_bis}
  or Item~\ref{it:KKT_optimization_problem_calL0_d=2_three}.
  For this purpose, we will use the relations
  \begin{subequations}
    \begin{align}
      \EuclidianTop{1}{\BALL} 
      &=
        [-1,1]^2  \eqfinv 
        \label{eq:BALL_1_d=2}
      \\
      \FACE_{ \EuclidianTop{1}{\BALL} }\np{ \bar\primal^{(1)} } 
      & =
        \begin{cases}
          [-1,1]^2 
          & 
          \mtext{ if } \lzero\np{ \bar\primal^{(1)} } = 0 \eqfinv 
          \\
          \textrm{sign}\np{ \bar\primal^{(1)}_1 } \times [-1,1]
          & 
          \mtext{ if } \lzero\np{ \bar\primal^{(1)} } = 1
          \mtext{ with } \bar\primal^{(1)}_2=0  \eqfinv 
          \\
          [-1,1] \times \textrm{sign}\np{ \bar\primal^{(1)}_2 } 
          & 
          \mtext{ if } \lzero\np{ \bar\primal^{(1)} } = 1
          \mtext{ with } \bar\primal^{(1)}_1=0  \eqfinv 
          \\
          \textrm{sign}\np{ \bar\primal^{(1)} } 
          &
          \mtext{ if } \lzero\np{ \bar\primal^{(1)} } = 2 \eqfinv
        \end{cases}
            \label{eq:FACE_1_d=2}
    \end{align}
  \end{subequations}
  where \( \textrm{sign}\np{ \bar\primal^{(1)} }
  = \bp{ \textrm{sign}\np{ \bar\primal^{(1)}_1 }, 
    \textrm{sign}\np{ \bar\primal^{(1)}_2 } } \)
  is the vector of~$\RR^2$ made of the signs ($-1,0,1$) of the 
  two components. 
  \smallskip

  \noindent $\bullet$ 
  Suppose that 
  \( \bp{ \bar\primal^{(1)}, \bar\primal^{(2)} } =\np{\primal,0} \)
  satisfies~\eqref{eq:KKT_optimization_problem_calL0_lambda_not=0_d=2}
  for a certain \( \lambda >0 \).
  We will show that this is equivalent to 
  \( 0< \primal_1 \), \( 0< \primal_2 \)
  and \( \primal_1+\primal_2 =1 \), which implies 
  Item~\ref{it:KKT_optimization_problem_calL0_d=2_one}.
  \smallskip

  \noindent    By~\eqref{eq:FACE_l_=BALL} for $l=d=2$, we get that 
  \( (2+\lambda) \FACE_{ \EuclidianTop{2}{\BALL} }\np{ 0 }
  =(2+\lambda) \BALL \), where $\BALL$ is the Euclidian unit ball of~$\RR^2$,
  so that Equation~\eqref{eq:KKT_optimization_problem_calL0_lambda_not=0_d=2}
  is equivalent to 
  \begin{equation}
    (1+\lambda) \FACE_{ \EuclidianTop{1}{\BALL} }\np{ \primal }  
    \cap
    (2+\lambda) \BALL 
    \not = \emptyset 
    \eqsepv
    \EuclidianSupportNorm{1}{\primal} = \primal_1+\primal_2 = 1 
    \eqfinp
    \label{eq:KKT_optimization_problem_calL0_lambda_not=0_d=2_bis}
  \end{equation}
  By~\eqref{eq:FACE_1_d=2},
  we distinguish the following subcases that correspond to 
  different expressions for 
  \( \FACE_{ \EuclidianTop{1}{\BALL} }\np{ \primal } \). 
  \begin{itemize}

  \item[-]
    If \( \lzero\np{ \primal } = 0 \), then \( \primal_1=\primal_2=0 \).
    But this contradicts \( \primal_1+\primal_2 = 1 \)
    in~\eqref{eq:KKT_optimization_problem_calL0_lambda_not=0_d=2_bis}.

  \item[-] 
    If \( \lzero\np{ \primal } = 1 \) with \( \primal_2=0 \), then 
    \( \primal=\np{ 1 , 0 } \) because \( \primal_1+\primal_2 = 1 \)
    by~\eqref{eq:KKT_optimization_problem_calL0_lambda_not=0_d=2_bis},
    and \( \primal=\np{\primal_1,\primal_2} \in \RR_+^2 \) by hypothesis.
    But this contradicts the assumption that 
    \( \primal_1^2+\primal_2^2<1 \). 

  \item[-] 
    If \( \lzero\np{ \primal } = 1 \) with \( \primal_1=0 \),
    we also arrive at a contradiction.

  \item[-] 
    If \( \lzero\np{ \primal } = 2 \), then 
    \( (1+\lambda)\FACE_{ \EuclidianTop{1}{\BALL} }\np{ \primal }
    =\{ \np{ 1+\lambda, 1+\lambda } \} \)
    by~\eqref{eq:FACE_1_d=2}.

    On the one hand (necessity),
    we show that necessarily \( 0< \lambda \leq \sqrt{2} \).
    Indeed, \eqref{eq:KKT_optimization_problem_calL0_lambda_not=0_d=2_bis}
    implies that 
    \( \norm{ \np{ (1+\lambda)\textrm{sign}\np{ \primal_1 } ,
        (1+\lambda)\textrm{sign}\np{ \primal_2 } } } \leq 2+\lambda \),
    which gives  \( \sqrt{2}(1+\lambda) \leq 2+\lambda \), hence 
    \( 0< \lambda \leq \sqrt{2} \).

    On the other hand (sufficiency), if we put
    \( \bp{ \bar\primal^{(1)}, \bar\primal^{(2)} } =\np{\primal,0} \)
    where 
    \( \EuclidianSupportNorm{1}{\primal} = \primal_1+\primal_2 =1 \)
    and \( \lzero\np{ \primal } = 2 \), that is, 
    \( 0< \primal_1 \), \( 0< \primal_2 \), 
    then 
    \eqref{eq:KKT_optimization_problem_calL0_lambda_not=0_d=2_bis}
    is satisfied for any 
    \( 0< \lambda \leq \sqrt{2} \).
  \end{itemize}
  Therefore, we have proven that 
  \( \bp{ \bar\primal^{(1)}, \bar\primal^{(2)} } =\np{\primal,0} \)
  satisfies~\eqref{eq:KKT_optimization_problem_calL0_lambda_not=0_d=2}
  for a certain \( \lambda >0 \) if and only if 
  \( 0< \primal_1 \), \( 0< \primal_2 \)
  and \( \primal_1+\primal_2 =1 \) (condition
  included in Item~\ref{it:KKT_optimization_problem_calL0_d=2_one}).
  \smallskip

  \noindent $\bullet$ 
  Suppose that 
  \( \bp{ \bar\primal^{(1)}, \bar\primal^{(2)} } =\np{0,\primal} \)
  satisfies~\eqref{eq:KKT_optimization_problem_calL0_lambda_not=0_d=2}
  for a certain \( \lambda >0 \).
  We will show that this case is impossible.
  Indeed, Equation~\eqref{eq:KKT_optimization_problem_calL0_lambda_not=0_d=2_b}
  implies that \( \sqrt{\primal_1^2+\primal_2^2}=
  \norm{\primal}=\EuclidianSupportNorm{2}{\primal} = 1 \). 
  But this contradicts the assumption that 
  \( \primal=\np{\primal_1,\primal_2} \in \RR_+^2 \) 
  is such that \( \primal_1^2+\primal_2^2<1 \). 
  \smallskip

  \noindent $\bullet$ 
  Suppose that 
  \( \bar\primal^{(1)} \not = 0 \) and \( \bar\primal^{(2)} \not = 0 \)
  are such that 
  \( \bp{ \bar\primal^{(1)}, \bar\primal^{(2)} } \)
  satisfies~\eqref{eq:KKT_optimization_problem_calL0_lambda_not=0_d=2}
  for a certain \( \lambda >0 \).
  We will show that this is equivalent to 
  Item~\ref{it:KKT_optimization_problem_calL0_d=2_two},
  Item~\ref{it:KKT_optimization_problem_calL0_d=2_two_bis}
  or Item~\ref{it:KKT_optimization_problem_calL0_d=2_three}.
  But, before that, notice that, as 
  \(  \EuclidianSupportNorm{1}{\bar\primal^{(1)}} + \EuclidianSupportNorm{2}{\bar\primal^{(2)}} =1 \),
  by~\eqref{eq:KKT_optimization_problem_calL0_lambda_not=0_b},
  then 
  \begin{equation}
    {\cal L}_0\np{\primal} = 1 + \EuclidianSupportNorm{2}{\bar\primal^{(2)}}
    = 2 - \EuclidianSupportNorm{1}{\bar\primal^{(1)}} 
    \eqfinv 
    \label{eq:practical}
  \end{equation}
  which will be practical to obtain formulas for~\( {\cal L}_0\np{\primal} \). 
  \smallskip

  As \( \bar\primal^{(2)} \not = 0 \), then 
  \( \FACE_{ \EuclidianTop{2}{\BALL} }\np{ \bar\primal^{(2)} } 
  = \{ \frac{ \bar\primal^{(2)} }{ \norm{ \bar\primal^{(2)} } } \} \)
  by~\eqref{eq:FACE_d_=pointing_vector}.
  Therefore, Equation~\eqref{eq:KKT_optimization_problem_calL0_lambda_not=0_d=2}
  is equivalent to 
  \begin{subequations}
    \begin{align}
      (2+\lambda) \frac{ \bar\primal^{(2)} }{ \norm{ \bar\primal^{(2)} } } 
      \in (1+\lambda) \FACE_{ \EuclidianTop{1}{\BALL} }\np{ \bar\primal^{(1)} }
      \eqfinv
      \label{eq:KKT_optimization_problem_calL0_lambda_not=0_d=2_ter_a}
      \\
      \EuclidianSupportNorm{1}{\bar\primal^{(1)}} + \EuclidianSupportNorm{2}{\bar\primal^{(2)}} 
      = \module{ \bar\primal^{(1)}_1 } + \module{ \bar\primal^{(1)}_2 } +
      \sqrt{ \module{ \bar\primal^{(2)}_1 }^2 + \module{ \bar\primal^{(2)}_2 }^2 }
      = 1 
      \eqfinv
      \label{eq:KKT_optimization_problem_calL0_lambda_not=0_d=2_ter_b}
      \\
      \bar\primal^{(1)} + \bar\primal^{(2)} = \primal
      \eqfinp
      \label{eq:KKT_optimization_problem_calL0_lambda_not=0_d=2_ter_c}    
    \end{align}
    \label{eq:KKT_optimization_problem_calL0_lambda_not=0_d=2_ter}
  \end{subequations}
  By~\eqref{eq:FACE_1_d=2},
  we distinguish the following four subcases that correspond to 
  different expressions for the face~\( \FACE_{ \EuclidianTop{1}{\BALL} }\np{ \bar\primal^{(1)} } \).
  \begin{itemize}
  \item[-] 
    As \( \bar\primal^{(1)} \not = 0 \), we do not consider the case 
    \( \lzero\np{ \bar\primal^{(1)} } = 0 \).

  \item[-] 
    Suppose that \( \lzero\np{ \bar\primal^{(1)} } = 1 \) with \( \bar\primal^{(1)}_2=0 \).
    Then, on the one hand,
    \( \FACE_{ \EuclidianTop{1}{\BALL} }\np{ \bar\primal^{(1)} } 
    =\textrm{sign}\np{ \bar\primal^{(1)}_1 } \times [-1,1] \) 
    by~\eqref{eq:FACE_1_d=2}, so that 
    Equation~\eqref{eq:KKT_optimization_problem_calL0_lambda_not=0_d=2_ter_a}
    is equivalent to 
    \begin{subequations}
      \begin{align*}
        \frac{ \bar\primal^{(2)}_1 }%
        { \sqrt{ \module{ \bar\primal^{(2)}_1 }^2 + \module{ \bar\primal^{(2)}_2 }^2 } } =
        \frac{ 1+ \lambda }{ 2+\lambda } \textrm{sign}\np{ \bar\primal^{(1)}_1 } 
        \eqfinv
        \quad 
        \frac{ \module{ \bar\primal^{(2)}_2 } }%
        { \sqrt{ \module{ \bar\primal^{(2)}_1 }^2 + \module{ \bar\primal^{(2)}_2 }^2 } } 
        \leq
        \frac{ 1+ \lambda }{ 2+\lambda } 
        \eqfinp 
      \end{align*}
    \end{subequations}
    On the other hand, 
    \( \bar\primal^{(1)}=\np{\bar\primal^{(1)}_1,0} \)
    where \( \bar\primal^{(1)}_1 \not = 0 \), so that 
    Equations~\eqref{eq:KKT_optimization_problem_calL0_lambda_not=0_d=2_ter}
    are equivalent to 
    \begin{align*}
      \module{ \bar\primal^{(1)}_1 } + 
      \sqrt{ \module{ \bar\primal^{(2)}_1 }^2 + \module{ \bar\primal^{(2)}_2 }^2 }
      = 1 
      \eqfinv
      \quad 
      \bar\primal^{(1)}_1 + \bar\primal^{(2)}_1 = \primal_1 
      \eqfinv
      \quad 
      \bar\primal^{(2)}_2 = \primal_2 
      \eqfinp 
    \end{align*}
    Therefore, 
    Equation~\eqref{eq:KKT_optimization_problem_calL0_lambda_not=0_d=2_ter}
    is equivalent to 
    \begin{subequations}
      \begin{align}
        \frac{ \bar\primal^{(2)}_1 }%
        { 1 - \module{ \bar\primal^{(1)}_1 } } 
        =
        \frac{ 1+ \lambda }{ 2+\lambda } \textrm{sign}\np{ \bar\primal^{(1)}_1 } 
        \eqfinv
        \label{eq:KKT_optimization_problem_calL0_lambda_not=0_d=2_cuatro_a}
        \\
        \module{ \bar\primal^{(2)}_2 }
        \leq 
        \textrm{sign}\np{ \bar\primal^{(1)}_1 } \bar\primal^{(2)}_1 
        \eqfinv
        \label{eq:KKT_optimization_problem_calL0_lambda_not=0_d=2_cuatro_b}
        \\
        \module{ \bar\primal^{(1)}_1 } + 
        \sqrt{ \module{ \bar\primal^{(2)}_1 }^2 + \module{ \bar\primal^{(2)}_2 }^2 }
        = 1 
        \eqfinv
        \label{eq:KKT_optimization_problem_calL0_lambda_not=0_d=2_cuatro_c}
        \\
        \bar\primal^{(1)}_1 + \bar\primal^{(2)}_1 = \primal_1 
        \eqfinv
        \label{eq:KKT_optimization_problem_calL0_lambda_not=0_d=2_cuatro_d}
        \\
        \bar\primal^{(2)}_2 = \primal_2 
        \eqfinv 
        \label{eq:KKT_optimization_problem_calL0_lambda_not=0_d=2_cuatro_e}
      \end{align}
      \label{eq:KKT_optimization_problem_calL0_lambda_not=0_d=2_cuatro}
    \end{subequations}
    and we will now show that there exists \( \lambda > 0 \) such that
    \eqref{eq:KKT_optimization_problem_calL0_lambda_not=0_d=2_cuatro}
    holds true if and only if 
    Item~\ref{it:KKT_optimization_problem_calL0_d=2_two} holds true. 

    On the one hand (necessity),
    from~\eqref{eq:KKT_optimization_problem_calL0_lambda_not=0_d=2_cuatro_a}, 
    we deduce that \( \bar\primal^{(2)}_1 \) and 
    \( \bar\primal^{(1)}_1 \) have the same sign;
    this common sign must therefore be \( \textrm{sign}\np{ \primal_1 } \),
    as \( \bar\primal^{(1)}_1 + \bar\primal^{(2)}_1 = \primal_1 \)
    by~\eqref{eq:KKT_optimization_problem_calL0_lambda_not=0_d=2_cuatro_d};
    since \( \primal=\np{\primal_1,\primal_2} \in \RR_+^2 \), we obtain that 
    \( \primal_1 \geq 0 \), hence \( \bar\primal^{(1)}_1 > 0 \) 
    and \( \primal_1 > 0 \).
    Therefore, we easily get that 
    \( \bar\primal^{(1)}=\np{\bar\primal^{(1)}_1,0} \), 
    where \( \bar\primal^{(1)}_1 > 0\), 
    and that 
    \( \bar\primal^{(2)}=\np{\primal_1 - \bar\primal^{(1)}_1,\primal_2} \),
    by~\eqref{eq:KKT_optimization_problem_calL0_lambda_not=0_d=2_cuatro_d}--\eqref{eq:KKT_optimization_problem_calL0_lambda_not=0_d=2_cuatro_e},
    with \( \primal_1 > \bar\primal^{(1)}_1 \), since \( \bar\primal^{(2)}_1 > 0\).
    Replacing the values
    in~\eqref{eq:KKT_optimization_problem_calL0_lambda_not=0_d=2_cuatro_c}
    --- 
    where \( \primal_1 > \bar\primal^{(1)}_1 > 0 \) and
    \( 1 > \primal_1 \) since \( \primal_1^2+\primal_2^2<1 \)
    --- 
    we get 
    \( \bar\primal^{(1)}_1 + 
    \sqrt{ \np{ \primal_1 - \bar\primal^{(1)}_1 }^2 + \primal_2^2 }
    = 1 \), from which we deduce that 
    \( \bar\primal^{(1)}_1 = \frac{1 - \np{ \primal_1^2+ \primal_2^2 } }%
    {2(1-\primal_1)} \); we have that \( \bar\primal^{(1)}_1 > 0\) because $\primal_1<1$; 
    the condition \( \primal_1 > \bar\primal^{(1)}_1 \) implies that 
    \( \primal_1 + \primal_2 > 1 \). 
    From~\eqref{eq:KKT_optimization_problem_calL0_lambda_not=0_d=2_cuatro_a}, 
    we deduce that 
    \( \frac{ \bar\primal^{(2)}_1 }{ 1 - \module{ \bar\primal^{(1)}_1 } } =
    \frac{ 1+ \lambda }{ 2+\lambda } \in ]1/2,1[\), 
    hence that 
    \( \frac{ \primal_1 - \bar\primal^{(1)}_1 }{1 - \bar\primal^{(1)}_1 } < 1 \)
    and
    \( 1/2 < \frac{ \primal_1 - \bar\primal^{(1)}_1 }{1 - \bar\primal^{(1)}_1 } \)
    by~\eqref{eq:KKT_optimization_problem_calL0_lambda_not=0_d=2_cuatro_b};
    we are going to detail these two inequalities, one after the other.
    We have that \( \frac{ \primal_1 - \bar\primal^{(1)}_1 }{1 - \bar\primal^{(1)}_1 } < 1 \)
    because \( 1 > \primal_1 > \bar\primal^{(1)}_1 \).
    The condition \( 1/2 < \frac{ \primal_1 - \bar\primal^{(1)}_1 }{1 - \bar\primal^{(1)}_1 } \) implies that  
    \( 0 < \primal_2 - \sqrt{3}(1-\primal_1) \);
    from~\eqref{eq:KKT_optimization_problem_calL0_lambda_not=0_d=2_cuatro_b}, 
    with \( \bar\primal^{(1)}_1 > 0\) and 
    \( \bar\primal^{(2)}_2 = \primal_2 > 0 \),
    we get that 
    \( \primal_2 \leq \primal_1 - \bar\primal^{(1)}_1 =
    \primal_1 - \frac{1 - \np{ \primal_1^2+ \primal_2^2 } }%
    {2(1-\primal_1)} \);
    rearranging terms, we find that this latter inequality 
    is equivalent to 
    \( \bp{ \primal_2 - \np{\sqrt{2}+1}\np{1-\primal_1} }
    \bp{ \primal_2 + \np{\sqrt{2}-1}\np{1-\primal_1} } \geq 0 \);
    as $\primal_1<1$ and \( \primal_2 >0 \), we finally get that 
    \( \primal_2 - \np{\sqrt{2}+1}\np{1-\primal_1} \geq 0 \).
    From \( \primal_2 \leq \primal_1 - \bar\primal^{(1)}_1 \) 
    where \( \bar\primal^{(1)}_1 > 0 \),
    we also deduce that necessarily \( \primal_1 > \primal_2 \).

    Finally, 
    Equation~\eqref{eq:KKT_optimization_problem_calL0_lambda_not=0_d=2_cuatro}
    implies that 
    \( \primal_1 > 0 \),
    \( \primal_1 + \primal_2 > 1 \),
    \( \primal_2 - \sqrt{3}(1-\primal_1) > 0 \),
    \( \primal_2 - \np{\sqrt{2}+1}\np{1-\primal_1}  \geq 0 \), 
    \( \primal_1 > \primal_2 \),
    and
    \( \bar\primal^{(1)} =
    \bp{ \frac{ 1-\np{ \primal_1^2+ \primal_2^2 } }{ 2(1-\primal_1) } , 0 } \),
    \( \bar\primal^{(2)} =
    \bp{ \frac{ 2\primal_1 -\primal_1^2+ \primal_2^2 -1 }%
      { 2(1-\primal_1) } , \primal_2 } \): thus, 
    using the property that 
    \[
      \primal_2 - \np{\sqrt{2}+1}\np{1-\primal_1}  \geq 0 \text{ and }
      1 > \primal_1 \Rightarrow 
      \begin{cases}
        \primal_1+\primal_2  \geq \sqrt{2}\np{1-\primal_1}+1 >1 
        \\ 
        \primal_2 - \sqrt{3}(1-\primal_1) > \primal_2 - \np{\sqrt{2}+1}\np{1-\primal_1}  \geq 0 
        \eqfinv 
      \end{cases}
    \]
    we obtain that \( \primal_2 - \np{\sqrt{2}+1}\np{1-\primal_1}  \geq 0 \)
    and \( 1 > \primal_1 \); multiplying the first inequality by $\sqrt{2}-1$,
    we finally obtain \( \primal_1 + \np{\sqrt{2}-1}\primal_2  \geq 1 \) 
    and \( 1 > \primal_1 \), 
    that is, Item~\ref{it:KKT_optimization_problem_calL0_d=2_two}. 
    
    On the other hand (sufficiency), 
    if we suppose that Item~\ref{it:KKT_optimization_problem_calL0_d=2_two} holds
    it is straightforward to follow all the above computations
    and to obtain that 
    Equation~\eqref{eq:KKT_optimization_problem_calL0_lambda_not=0_d=2_cuatro}
    holds true with \( \lambda > 0 \) the unique solution to 
    \( \frac{ \bar\primal^{(2)}_1 }{ 1 - \module{ \bar\primal^{(1)}_1 } } =
    \frac{ 1+ \lambda }{ 2+\lambda } \in ]1/2,1[\).

    By~\eqref{eq:practical}, we obtain that 
    \( {\cal L}_0\bp{ \np{\primal_1,\primal_2} } = 
    2- \module{ \bar\primal^{(1)}_1 } - \module{ \bar\primal^{(1)}_2 }
    = \frac{3}{2} - \frac{\primal_1}{2} + 
    \frac{ \primal_2^2 }{2(1-\primal_1)} \). 
    
  \item[-] 
    If \( \lzero\np{ \bar\primal^{(1)} } = 1 \) with \( \bar\primal^{(1)}_1=0 \),
    we do the same analysis, and we obtain 
    Item~\ref{it:KKT_optimization_problem_calL0_d=2_two_bis},
    and \( {\cal L}_0\bp{ \np{\primal_1,\primal_2} } 
    = \frac{3}{2} - \frac{\primal_2}{2} + 
    \frac{ \primal_1^2 }{2(1-\primal_2)} \). 

  \item[-] 
    Suppose that 
    \( \lzero\np{ \bar\primal^{(1)} } = 2 \).
    In this case, we have that 
    \(  \FACE_{ \EuclidianTop{1}{\BALL} }\np{ \bar\primal^{(1)} }
    = \{ \bp{ \textrm{sign}\np{ \bar\primal^{(1)}_1 } ,
      \textrm{sign}\np{ \bar\primal^{(1)}_2 } } \} \) by~\eqref{eq:FACE_1_d=2}.
    Therefore, 
    Equation~\eqref{eq:KKT_optimization_problem_calL0_lambda_not=0_d=2_ter}
    is equivalent to 
    \begin{subequations}
      \begin{align}
        \frac{ \bar\primal^{(2)}_1 }%
        { \sqrt{ \bmodule{ \bar\primal^{(2)}_1 }^2 + \bmodule{ \bar\primal^{(2)}_2 }^2 } }
        =
        \frac{ 1+ \lambda }{ 2+\lambda } \textrm{sign}\np{ \bar\primal^{(1)}_1 } 
        \eqfinv
        \label{eq:KKT_optimization_problem_calL0_lambda_not=0_d=2_cinco_a}
        \\
        \frac{ \bar\primal^{(2)}_2 }%
        { \sqrt{ \bmodule{ \bar\primal^{(2)}_1 }^2 + \bmodule{ \bar\primal^{(2)}_2 }^2 } }
        =
        \frac{ 1+ \lambda }{ 2+\lambda } \textrm{sign}\np{ \bar\primal^{(1)}_2 } 
        \eqfinv
        \label{eq:KKT_optimization_problem_calL0_lambda_not=0_d=2_cinco_b}
        \\
        \bmodule{ \bar\primal^{(1)}_1 } + \bmodule{ \bar\primal^{(1)}_2 } +
        \sqrt{ \bmodule{ \bar\primal^{(2)}_1 }^2 + \bmodule{ \bar\primal^{(2)}_2 }^2 }
        = 1 
        \eqfinv
        \label{eq:KKT_optimization_problem_calL0_lambda_not=0_d=2_cinco_c}
        \\
        \bar\primal^{(1)}_1 + \bar\primal^{(2)}_1 = \primal_1 
        \eqfinv
        \label{eq:KKT_optimization_problem_calL0_lambda_not=0_d=2_cinco_d}
        \\
        \bar\primal^{(1)}_2 + \bar\primal^{(2)}_2 = \primal_2 
        \eqfinv 
        \label{eq:KKT_optimization_problem_calL0_lambda_not=0_d=2_cinco_e}
      \end{align}
      \label{eq:KKT_optimization_problem_calL0_lambda_not=0_d=2_cinco}
    \end{subequations}
    and we will now show that there exists \( \lambda > 0 \) such that
    Equation~\eqref{eq:KKT_optimization_problem_calL0_lambda_not=0_d=2_cinco}
    holds true if and only if 
    Item~\ref{it:KKT_optimization_problem_calL0_d=2_three} holds true. 

    On the one hand (necessity),
    from~\eqref{eq:KKT_optimization_problem_calL0_lambda_not=0_d=2_cinco_a}--\eqref{eq:KKT_optimization_problem_calL0_lambda_not=0_d=2_cinco_b},
    we deduce that 
    \( \module{ \bar\primal^{(2)}_1 } =  \module{ \bar\primal^{(2)}_2 } \)
    --- because \( \module{ \textrm{sign}\np{ \bar\primal^{(1)}_1 } }
    = \module{ \textrm{sign}\np{ \bar\primal^{(1)}_2 } } = 1 \)
    since \( \lzero\np{ \bar\primal^{(1)} } = 2 \) --- 
    and that \( \textrm{sign}\np{ \bar\primal^{(1)} } =
    \textrm{sign}\np{ \bar\primal^{(2)} } \).
    This common sign must therefore be \( \textrm{sign}\np{ \primal } \),
    as \( \bar\primal^{(1)} + \bar\primal^{(2)} = \primal \) 
    by~\eqref{eq:KKT_optimization_problem_calL0_lambda_not=0_d=2_cinco_d}--\eqref{eq:KKT_optimization_problem_calL0_lambda_not=0_d=2_cinco_e}. 
    Since \( \lzero\np{ \bar\primal^{(1)} } = 2 \) and 
    \( \primal=\np{\primal_1,\primal_2} \in \RR_+^2 \), we get that 
    \( \primal_1 > 0 \) and  \( \primal_2 > 0 \), 
    so that we put \( \bar\primal^{(2)}_1 = \bar\primal^{(2)}_2 = \beta >0 \).
    By~\eqref{eq:KKT_optimization_problem_calL0_lambda_not=0_d=2_cinco_d}--\eqref{eq:KKT_optimization_problem_calL0_lambda_not=0_d=2_cinco_e},
    we get that 
    \( \bar\primal^{(1)}_1 = \primal_1 - \beta  > 0 \)
    and \( \bar\primal^{(1)}_2 = \primal_2 - \beta  > 0 \);
    replacing the values
    in~\eqref{eq:KKT_optimization_problem_calL0_lambda_not=0_d=2_cinco_c},
    we obtain that 
    \( \primal_1 - \beta + \primal_2 - \beta + \sqrt{2} \beta = 1 \);
    this gives \( \beta = \frac{\primal_1 + \primal_2 -1 }{2-\sqrt{2}} \).
    Therefore, 
    \( \beta > 0 \iff \primal_1 + \primal_2 >1 \),
    \( \beta < \primal_1 \iff \np{\sqrt{2}-1}\primal_1 + \primal_2 < 1 \)
    and
    \( \beta < \primal_2 \iff \primal_1 + \np{\sqrt{2}-1}\primal_2 < 1 \).

    Finally, 
    Equation~\eqref{eq:KKT_optimization_problem_calL0_lambda_not=0_d=2_cinco}
    implies that 
    \( \primal_1 + \primal_2 > 1 \),
    \( \np{\sqrt{2}-1}\primal_1 + \primal_2 < 1 \),
    \( \primal_1 + \np{\sqrt{2}-1}\primal_2 < 1 \),
    and 
    \( \bar\primal^{(1)} = 
    \bp{ \frac{ 1- \np{\sqrt{2}-1}{\primal_1} - {\primal_2} }{ 2\np{\sqrt{2}-1} } , 
      \frac{ 1- {\primal_1} - \np{\sqrt{2}-1}{\primal_2} }{ 2\np{\sqrt{2}-1} } } \),
    \( \bar\primal^{(2)} =
    \bp{ \frac{ {\primal_1} + {\primal_2} -1 }{ 2\np{\sqrt{2}-1} } , 
      \frac{ {\primal_1} + {\primal_2} -1 }{ 2\np{\sqrt{2}-1} } } \):
    thus, Item~\ref{it:KKT_optimization_problem_calL0_d=2_three} holds true. 
    
    On the other hand (sufficiency), 
    if we suppose that Item~\ref{it:KKT_optimization_problem_calL0_d=2_three} holds
    true, 
    it is straightforward to follow all the above computations
    and to obtain that 
    Equation~\eqref{eq:KKT_optimization_problem_calL0_lambda_not=0_d=2_cinco}
    holds true with \( \lambda= \sqrt{2} \).
    
    By~\eqref{eq:practical}, we obtain that 
    \( {\cal L}_0\bp{ \np{\primal_1,\primal_2} } = 
    1+ \sqrt{ \bmodule{ \bar\primal^{(2)}_1 }^2 + \bmodule{ \bar\primal^{(2)}_2 }^2 } 
    = 1 + \frac{\primal_1 + \primal_2 -1 }{\sqrt{2}-1} \). 
  \end{itemize}
  \smallskip

  This ends the proof. 
\end{proof}

\newcommand{\noopsort}[1]{} \ifx\undefined\allcaps\def\allcaps#1{#1}\fi

\end{document}